\documentclass[reqno]{gtpart}

\title{Tied--boxed algebras}
\author[D. Arcis and J. Espinoza]{Diego Arcis\\Jorge Espinoza}

\address{
	Departamento de Matem\'aticas, Universidad de La Serena, Cisternas 1200 -- 1700000 La Serena, Chile.\textcolor{white}{$\underbrace{1}$}\newline
	Instituto de Matem\'aticas, Universidad de Talca, Campus Norte, Camino Lircay S/N -- 3460000 Talca, Chile.\textcolor{white}{$\underbrace{.}$}
}

\email{diego.arcis@userena.cl}
\email{joespinoza@utalca.cl}


\usepackage[utf8]{inputenc}

\usepackage{
	accents,
	amsthm,
	amsmath,
	amsfonts,
	cancel, 
	caption,
    float,
    graphicx,
    mathtools,
    mathrsfs,
    soul, 
    subcaption,
    young,
    ytableau
}

\usepackage[all]{xy}
\usepackage[bottom=3.5cm,left=2.8cm,right=2cm,top=3cm]{geometry}
\usepackage[vcentermath]{youngtab}
\usepackage{bbm}

\newtheorem{thm}{Theorem}[section]
\newtheorem{crl}[thm]{Corollary}
\newtheorem{lem}[thm]{Lemma}
\newtheorem{pro}[thm]{Proposition}

\theoremstyle{definition}

\newtheorem{rem}[thm]{Remark}

\numberwithin{equation}{section}

\newcommand{\B}{\mathcal{B}}
\newcommand{\E}{\mathcal{E}}
\newcommand{\G}{\mathcal{G}}
\newcommand{\I}{\mathcal{I}}

\renewcommand{\C}{\mathcal{C}}
\renewcommand{\J}{\mathcal{J}}
\renewcommand{\L}{\mathcal{L}}
\renewcommand{\P}{\mathcal{P}}
\renewcommand{\R}{\mathcal{R}}
\renewcommand{\S}{\mathcal{S}}
\renewcommand{\H}{\mathcal{H}}

\newcommand{\Br}{\mathcal{B}r}

\newcommand{\BR}{\mathcal{BR}}
\newcommand{\LP}{\mathcal{LP}}
\newcommand{\MC}{\mathcal{MC}}
\newcommand{\MPar}{\mathcal{M}\mathrm{Par}}
\newcommand{\RP}{\mathcal{RP}}

\newcommand{\CC}{\mathfrak{C}}

\newcommand{\s}{\mathfrak{s}}
\newcommand{\V}{\mathfrak{v}}

\renewcommand{\b}{\mathfrak{b}}
\renewcommand{\c}{\mathfrak{c}}
\renewcommand{\t}{\mathfrak{t}}
\renewcommand{\u}{\mathfrak{u}}

\newcommand{\N}{\mathbb{N}}

\newcommand{\qring}{\mathbb{S}}

\newcommand{\rE}{\operatorname{E}}

\newcommand{\Ker}{\operatorname{Ker}}
\newcommand{\End}{\operatorname{End}}
\newcommand{\Mat}{\operatorname{Mat}}
\newcommand{\Par}{\operatorname{Par}}
\newcommand{\Std}{\operatorname{Std}}
\newcommand{\Tab}{\operatorname{Tab}}

\newcommand{\FTL}{\operatorname{FLT}}
\newcommand{\PTL}{\operatorname{PTL}}

\newcommand{\id}{{\operatorname{id}}}
\newcommand{\ord}{{\operatorname{ord}}}

\newcommand{\comp}{\operatorname{\textbf{comp}}}
\newcommand{\shape}{\operatorname{shape}}

\newcommand{\bE}{\mathbf{E}}
\newcommand{\bG}{\mathbf{G}}
\newcommand{\bH}{\mathbf{H}}
\newcommand{\bI}{\mathbf{I}}
\newcommand{\bJ}{\mathbf{J}}
\newcommand{\bK}{\mathbf{K}}

\newcommand{\bbB}{\mathbb{B}}
\newcommand{\bbE}{\mathbb{E}}

\newcommand{\bmm}{\mathbf{m}}

\newcommand{\bzz}{\mathbf{z}}

\newcommand{\bms}{\mathbbm{s}}
\newcommand{\bmt}{\mathbbm{t}}

\newcommand{\et}{\boldsymbol{\mathbbm{t}}}
\newcommand{\es}{\boldsymbol{\mathbbm{s}}}

\newcommand{\bt}{\boldsymbol{\mathfrak{t}}}
\newcommand{\bs}{{\boldsymbol{\mathfrak{s}}}}
\newcommand{\bu}{\boldsymbol{\mathfrak{u}}}
\newcommand{\bv}{\boldsymbol{\mathfrak{v}}}
\newcommand{\blam}{{\boldsymbol\lambda}}
\newcommand{\bmu}{{\boldsymbol\mu}}

\newcommand{\ttau}{\tilde{\tau}}

\newcommand{\yyoung}[1]{\scalebox{0.84}{\ytableausetup{boxsize=1.3em,centertableaux}\ytableaushort{#1}}}
\newcommand{\yyng}[1]{\ytableausetup{boxsize=1.05em,centertableaux}\ydiagram{#1}}

\newcommand{\td}{{\tt d}}
\newcommand{\te}{{\tt e}}
\newcommand{\tz}{{\tt z}}
\newcommand{\tsz}{\tz_\square} 

\newcommand{\aspdf}{1}
\newcommand{\vcdraw}[1]{\vcenter{\hbox{#1}}}

\newcommand{\figtwosev}{
	\ifnum\aspdf=1\[
		\vcdraw{\includegraphics{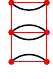}}
		\,\,\,=\,\,
		(q+q^{-1})\!\!
		\vcdraw{\includegraphics{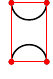}}
	\]\else
		\begin{tikzpicture}
			\vpartition[tkzpic=0,type=0,height=0.5,floor=1]{{1,2},{-1,-2}}
			\vpartition[tkzpic=0,type=0,height=0.5,bulla=0]{{1,2},{-1,-2}}
			\tie[style=solid]{{1,0.5},{1,0},{2,0},{2,1},{1,1},{1,0.5},{2,0.5}}
		\end{tikzpicture}
		\begin{tikzpicture}
			\strands[tkzpic=0,type=0]{t1}
			\tie[style=solid]{{1,0},{2,0},{2,1},{1,1},{1,0}}
		\end{tikzpicture}
	\fi
}

\newcommand{\figtwosix}{ 
	\ifnum\aspdf=1\[
		\vcdraw{\includegraphics{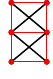}}
		\,\,\,=\,
		\vcdraw{\includegraphics{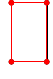}}
		\,\,+\,\,(q-q^{-1})\!\!
		\vcdraw{\includegraphics{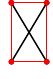}}
		\quad\Leftrightarrow\quad
		\vcdraw{\includegraphics{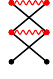}}
		\,\,\,=\,
		\vcdraw{\includegraphics{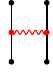}}
		\,\,+\,\,(q-q^{-1})\!\!
		\vcdraw{\includegraphics{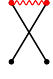}}
	\]\else
		\begin{tikzpicture}
			\vpartition[tkzpic=0,type=0,height=0.5,floor=1]{{1,-2},{2,-1}}
			\vpartition[tkzpic=0,type=0,height=0.5,bulla=0]{{1,-2},{2,-1}}
			\tie[style=solid]{{1,0.5},{1,0},{2,0},{2,1},{1,1},{1,0.5},{2,0.5}}
		\end{tikzpicture}
		\begin{tikzpicture}
			\vpartition[tkzpic=0,type=0]{{1,-1},{2,-2}}
			\tie[style=solid]{{1,1},{2,1},{2,0},{1,0},{1,1}}
		\end{tikzpicture}
		\begin{tikzpicture}
			\vpartition[tkzpic=0,type=0]{{1,-2},{2,-1}}
			\tie[style=solid]{{1,1},{2,1},{2,0},{1,0},{1,1}}
		\end{tikzpicture}
		\begin{tikzpicture}
			\vpartition[tkzpic=0,type=0,height=0.5,floor=1]{{1,-2},{2,-1}}
			\vpartition[tkzpic=0,type=0,height=0.5,bulla=0]{{1,-2},{2,-1}}
			\tie[snake=true,snakelen=3,style=solid]{{1,1},{2,1}}
			\tie[snake=true,snakelen=3,style=solid]{{1,0.5},{2,0.5}}
		\end{tikzpicture}
		\begin{tikzpicture}
			\vpartition[tkzpic=0,type=0]{{1,-1},{2,-2}}
			\tie[snake=true,snakelen=3,style=solid]{{1,0.5},{2,0.5}}
		\end{tikzpicture}
		\begin{tikzpicture}
			\vpartition[tkzpic=0,type=0]{{1,-2},{2,-1}}
			\tie[snake=true,snakelen=3,style=solid]{{1,1},{2,1}}
		\end{tikzpicture}
	\fi
}

\newcommand{\figtwofiv}{\[\begin{array}{c}
	\left(\left(\yyoung{45}\,,\,\yyoung{9{10}}\,,\,\yyoung{{11}{12}}\,,\,\yyoung{12,3}\,,\,\yyoung{67,8}\,,\,\yyoung{{13}{14}{15}}\right)\Bigg|\left(\,\,\yyoung{1,2,3}\,,\,\yyoung{4,5}\,,\,\yyoung{6}\right)\right)\\[0.8cm]
	\bs=\left(\,\yyoung{12,3}\,,\,\yyoung{45}\,,\,\yyoung{67,8}\,,\,\yyoung{9{10}}\,,\,\yyoung{{11}{12}}\,,\,\yyoung{{13}{14}{15}}\right)
\end{array}\]}

\newcommand{\figtwofou}{
	\ifnum\aspdf=1\[
		\vcdraw{\includegraphics{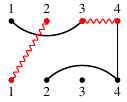}}
	\]\else
		\begin{tikzpicture}
			\vpartition[tkzpic=0,type=2]{{1,3},{-2,-4,4}}
			\tie[snake=true,snakelen=3,style=solid]{{4,1},{3,1}}
			\tie[snake=true,snakelen=3,style=solid]{{1,0},{2,1}}
		\end{tikzpicture}
	\fi
}

\newcommand{\figtwothr}{
	\centering
	\ifnum\aspdf=1\[
		\vcdraw{\includegraphics{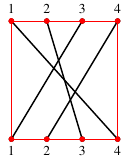}}\quad=\quad
		\vcdraw{\includegraphics{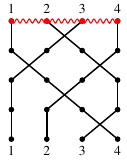}}\quad=\quad
		\vcdraw{\includegraphics{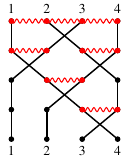}}\quad=\quad
		\vcdraw{\includegraphics{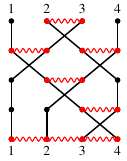}}
	\]\else
		\begin{tikzpicture}
			\permutation[tkzpic=0,height=2,type=2]{4,3,1,2}
			\tie[style=solid,height=2]{{1,1},{2,1},{3,1},{4,1},{4,0},{3,0},{2,0},{1,0},{1,1}}
		\end{tikzpicture}
		\begin{tikzpicture}
			\permutation[tkzpic=0,height=0.5,type=1,bullb=0,floor=3]{1,3,2,4}
			\permutation[tkzpic=0,height=0.5,type=0,bullb=0,floor=2]{2,1,4,3}
			\permutation[tkzpic=0,height=0.5,type=0,bullb=0,floor=1]{1,3,2,4}
			\permutation[tkzpic=0,height=0.5,type=-1]{1,2,4,3}
			\tie[snake=true,snakelen=3,style=solid,height=2]{{1,1},{2,1},{3,1},{4,1}}
		\end{tikzpicture}
		\begin{tikzpicture}
			\permutation[tkzpic=0,height=0.5,type=1,bullb=0,floor=3]{1,3,2,4}
			\permutation[tkzpic=0,height=0.5,type=0,bullb=0,floor=2]{2,1,4,3}
			\permutation[tkzpic=0,height=0.5,type=0,bullb=0,floor=1]{1,3,2,4}
			\permutation[tkzpic=0,height=0.5,type=-1]{1,2,4,3}
			\tie[snake=true,snakelen=3,style=solid,height=2]{{1,1},{2,1},{3,1},{4,1}}
			\tie[snake=true,snakelen=3,style=solid,height=2]{{1,0.75},{2,0.75}}
			\tie[snake=true,snakelen=3,style=solid,height=2]{{3,0.75},{4,0.75}}
			\tie[snake=true,snakelen=3,style=solid,height=2]{2,3}
			\tie[snake=true,snakelen=3,style=solid,height=2]{{3,0.25},{4,0.25}}
		\end{tikzpicture}
		\begin{tikzpicture}
			\permutation[tkzpic=0,height=0.5,type=1,bullb=0,floor=3]{1,3,2,4}
			\permutation[tkzpic=0,height=0.5,type=0,bullb=0,floor=2]{2,1,4,3}
			\permutation[tkzpic=0,height=0.5,type=0,bullb=0,floor=1]{1,3,2,4}
			\permutation[tkzpic=0,height=0.5,type=-1]{1,2,4,3}
			\tie[snake=true,snakelen=3,style=solid,height=2]{{1,0},{2,0},{3,0},{4,0}}
			\tie[snake=true,snakelen=3,style=solid,height=2]{{1,0.75},{2,0.75}}
			\tie[snake=true,snakelen=3,style=solid,height=2]{{3,0.75},{4,0.75}}
			\tie[snake=true,snakelen=3,style=solid,height=2]{{2,1},{3,1}}
			\tie[snake=true,snakelen=3,style=solid,height=2]{2,3}
			\tie[snake=true,snakelen=3,style=solid,height=2]{{3,0.25},{4,0.25}}
		\end{tikzpicture}
	\fi
}

\newcommand{\figtwotwo}{
	\ifnum\aspdf=1\[
		\vcdraw{\includegraphics{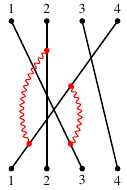}}\quad\,=\quad
		\vcdraw{\includegraphics{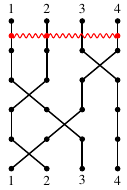}}\quad\,=\quad
		\vcdraw{\includegraphics{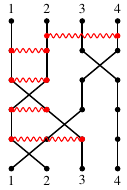}}
	\]\else
		\begin{tikzpicture}
			\permutation[tkzpic=0,height=2.5]{3,2,4,1}
			\tie[snake=true,snakelen=3,style=solid,bend=30]{{1.5,0.42},{2,2}}
			\tie[snake=true,snakelen=3,style=solid,bend=-30]{{2.67,0.42},{2.68,1.4}}
		\end{tikzpicture}
		\begin{tikzpicture}
			\permutation[tkzpic=0,height=0.5,floor=4,type=1]{1,2,3,4}
			\permutation[tkzpic=0,height=0.5,floor=3,type=0,bulla=0]{1,2,4,3}
			\permutation[tkzpic=0,height=0.5,floor=2,type=0,bulla=0]{2,1,3,4}
			\permutation[tkzpic=0,height=0.5,floor=1,type=0,bulla=0]{1,3,2,4}
			\permutation[tkzpic=0,height=0.5,type=-1,bulla=0]{2,1,3,4}
			\tie[snake=true,snakelen=3,style=solid]{{1,2.25},{2,2.25},{4,2.25}}
		\end{tikzpicture}
		\begin{tikzpicture}
			\permutation[tkzpic=0,height=0.5,floor=4,type=1]{1,2,3,4}
			\permutation[tkzpic=0,height=0.5,floor=3,type=0,bulla=0]{1,2,4,3}
			\permutation[tkzpic=0,height=0.5,floor=2,type=0,bulla=0]{2,1,3,4}
			\permutation[tkzpic=0,height=0.5,floor=1,type=0,bulla=0]{1,3,2,4}
			\permutation[tkzpic=0,height=0.5,type=-1,bulla=0]{2,1,3,4}
			\tie[snake=true,snakelen=3,style=solid]{{2,2.25},{4,2.25}}
			\tie[snake=true,snakelen=3,style=solid]{{1,2},{2,2}}
			\tie[snake=true,snakelen=3,style=solid]{{1,1.5},{2,1.5}}
			\tie[snake=true,snakelen=3,style=solid]{{1,1},{2,1}}
			\tie[snake=true,snakelen=3,style=solid]{{1,0.5},{3,0.5}}
		\end{tikzpicture}
	\fi
}

\newcommand{\figtwoone}{
	\ifnum\aspdf=1\[
		\vcdraw{\includegraphics{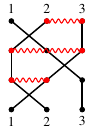}}\quad=\!\quad
		\vcdraw{\includegraphics{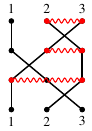}}\kern+4em
		\vcdraw{\includegraphics{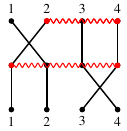}}\quad=\quad
		\vcdraw{\includegraphics{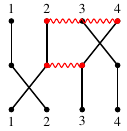}}
	\]\else
		\begin{tikzpicture}
			\permutation[tkzpic=0,height=0.5,floor=2,type=1]{2,1,3}
			\permutation[tkzpic=0,height=0.5,floor=1,bulla=0,bullb=0,type=0]{1,3,2}
			\permutation[tkzpic=0,height=0.5,type=-1]{2,1,3}
			\tie[snake=true,snakelen=3,style=solid]{{2,1.5},{3,1.5}}
			\tie[snake=true,snakelen=3,style=solid]{{1,1.0},{3,1.0}}
			\tie[snake=true,snakelen=3,style=solid]{{1,0.5},{2,0.5}}
		\end{tikzpicture}
		\begin{tikzpicture}
			\permutation[tkzpic=0,height=0.5,floor=2,type=1]{1,3,2}
			\permutation[tkzpic=0,height=0.5,floor=1,bulla=0,bullb=0,type=0]{2,1,3}
			\permutation[tkzpic=0,height=0.5,type=-1]{1,3,2}
			\tie[snake=true,snakelen=3,style=solid]{{2,1.5},{3,1.5}}
			\tie[snake=true,snakelen=3,style=solid]{{2,1.0},{3,1.0}}
			\tie[snake=true,snakelen=3,style=solid]{{1,0.5},{3,0.5}}
		\end{tikzpicture}
		\begin{tikzpicture}
			\permutation[tkzpic=0,height=0.75,floor=1,type=1]{2,1,3,4}
			\permutation[tkzpic=0,height=0.75,bulla=0,type=-1]{1,2,4,3}
			\tie[snake=true,snakelen=3,style=solid,height=1.5]{{2,1},{4,1}}
			\tie[snake=true,snakelen=3,style=solid,height=1.5]{{2,1},{4,1}}
			\tie[snake=true,snakelen=3,style=solid,height=1.5]{{1,0.5},{4,0.5}}
		\end{tikzpicture}
		\begin{tikzpicture}
			\permutation[tkzpic=0,height=0.75,floor=1,type=1]{1,2,4,3}
			\permutation[tkzpic=0,height=0.75,bulla=0,type=-1]{2,1,3,4}
			\tie[snake=true,snakelen=3,style=solid,height=1.5]{{2,1},{4,1}}
			\tie[snake=true,snakelen=3,style=solid,height=1.5]{{2,0.5},{3,0.5}}
		\end{tikzpicture}
	\fi
}

\newcommand{\figtty}{
	\ifnum\aspdf=1\[
		\begin{array}{c}
			\includegraphics[scale=0.9]{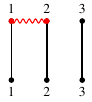}\quad
			\includegraphics[scale=0.9]{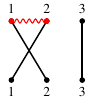}\quad
			\includegraphics[scale=0.9]{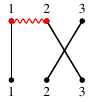}\quad
			\includegraphics[scale=0.9]{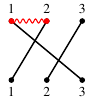}\quad
			\includegraphics[scale=0.9]{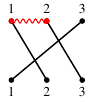}\quad
			\includegraphics[scale=0.9]{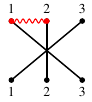}\quad
			\includegraphics[scale=0.9]{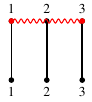}\quad
			\includegraphics[scale=0.9]{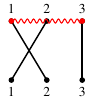}\\[0.7mm]
			\includegraphics[scale=0.9]{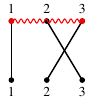}\quad
			\includegraphics[scale=0.9]{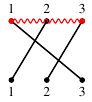}\quad
			\includegraphics[scale=0.9]{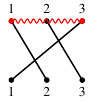}\quad
			\includegraphics[scale=0.9]{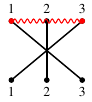}\quad
			\includegraphics[scale=0.9]{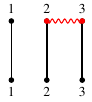}\quad
			\includegraphics[scale=0.9]{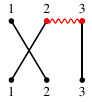}\quad
			\includegraphics[scale=0.9]{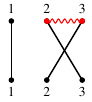}\quad
			\includegraphics[scale=0.9]{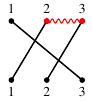}\\[0.7mm]
			\includegraphics[scale=0.9]{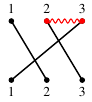}\quad
			\includegraphics[scale=0.9]{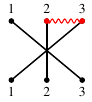}\quad
			\includegraphics[scale=0.9]{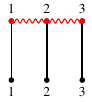}\quad
			\includegraphics[scale=0.9]{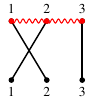}\quad
			\includegraphics[scale=0.9]{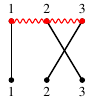}\quad
			\includegraphics[scale=0.9]{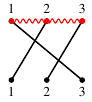}\quad
			\includegraphics[scale=0.9]{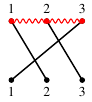}\quad
			\includegraphics[scale=0.9]{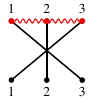}
		\end{array}
	\]\else
		\begin{tikzpicture}
			\permutation[tkzpic=0]{1,2,3}\tie[snake=true,snakelen=3,style=solid]{{1,1},{2,1}}
		\end{tikzpicture}
		\begin{tikzpicture}
			\permutation[tkzpic=0]{2,1,3}\tie[snake=true,snakelen=3,style=solid]{{1,1},{2,1}}
		\end{tikzpicture}
		\begin{tikzpicture}
			\permutation[tkzpic=0]{1,3,2}\tie[snake=true,snakelen=3,style=solid]{{1,1},{2,1}}
		\end{tikzpicture}
		\begin{tikzpicture}
			\permutation[tkzpic=0]{3,1,2}\tie[snake=true,snakelen=3,style=solid]{{1,1},{2,1}}
		\end{tikzpicture}
		\begin{tikzpicture}
			\permutation[tkzpic=0]{2,3,1}\tie[snake=true,snakelen=3,style=solid]{{1,1},{2,1}}
		\end{tikzpicture}
		\begin{tikzpicture}
			\permutation[tkzpic=0]{3,2,1}\tie[snake=true,snakelen=3,style=solid]{{1,1},{2,1}}
		\end{tikzpicture}
		\begin{tikzpicture}
			\permutation[tkzpic=0]{1,2,3}\tie[snake=true,snakelen=3,style=solid]{{1,1},{3,1}}
		\end{tikzpicture}
		\begin{tikzpicture}
			\permutation[tkzpic=0]{2,1,3}\tie[snake=true,snakelen=3,style=solid]{{1,1},{3,1}}
		\end{tikzpicture}
		\begin{tikzpicture}
			\permutation[tkzpic=0]{1,3,2}\tie[snake=true,snakelen=3,style=solid]{{1,1},{3,1}}
		\end{tikzpicture}
		\begin{tikzpicture}
			\permutation[tkzpic=0]{3,1,2}\tie[snake=true,snakelen=3,style=solid]{{1,1},{3,1}}
		\end{tikzpicture}
		\begin{tikzpicture}
			\permutation[tkzpic=0]{2,3,1}\tie[snake=true,snakelen=3,style=solid]{{1,1},{3,1}}
		\end{tikzpicture}
		\begin{tikzpicture}
			\permutation[tkzpic=0]{3,2,1}\tie[snake=true,snakelen=3,style=solid]{{1,1},{3,1}}
		\end{tikzpicture}
		\begin{tikzpicture}
			\permutation[tkzpic=0]{1,2,3}\tie[snake=true,snakelen=3,style=solid]{{2,1},{3,1}}
		\end{tikzpicture}
		\begin{tikzpicture}
			\permutation[tkzpic=0]{2,1,3}\tie[snake=true,snakelen=3,style=solid]{{2,1},{3,1}}
		\end{tikzpicture}
		\begin{tikzpicture}
			\permutation[tkzpic=0]{1,3,2}\tie[snake=true,snakelen=3,style=solid]{{2,1},{3,1}}
		\end{tikzpicture}
		\begin{tikzpicture}
			\permutation[tkzpic=0]{3,1,2}\tie[snake=true,snakelen=3,style=solid]{{2,1},{3,1}}
		\end{tikzpicture}
		\begin{tikzpicture}
			\permutation[tkzpic=0]{2,3,1}\tie[snake=true,snakelen=3,style=solid]{{2,1},{3,1}}
		\end{tikzpicture}
		\begin{tikzpicture}
			\permutation[tkzpic=0]{3,2,1}\tie[snake=true,snakelen=3,style=solid]{{2,1},{3,1}}
		\end{tikzpicture}
		\begin{tikzpicture}
			\permutation[tkzpic=0]{1,2,3}\tie[snake=true,snakelen=3,style=solid]{{1,1},{2,1},{3,1}}
		\end{tikzpicture}
		\begin{tikzpicture}
			\permutation[tkzpic=0]{2,1,3}\tie[snake=true,snakelen=3,style=solid]{{1,1},{2,1},{3,1}}
		\end{tikzpicture}
		\begin{tikzpicture}
			\permutation[tkzpic=0]{1,3,2}\tie[snake=true,snakelen=3,style=solid]{{1,1},{2,1},{3,1}}
		\end{tikzpicture}
		\begin{tikzpicture}
			\permutation[tkzpic=0]{3,1,2}\tie[snake=true,snakelen=3,style=solid]{{1,1},{2,1},{3,1}}
		\end{tikzpicture}
		\begin{tikzpicture}
			\permutation[tkzpic=0]{2,3,1}\tie[snake=true,snakelen=3,style=solid]{{1,1},{2,1},{3,1}}
		\end{tikzpicture}
		\begin{tikzpicture}
			\permutation[tkzpic=0]{3,2,1}\tie[snake=true,snakelen=3,style=solid]{{1,1},{2,1},{3,1}}
		\end{tikzpicture}
	\fi
}

\newcommand{\figntn}{
	\ifnum\aspdf=1\[
		\vcdraw{\includegraphics{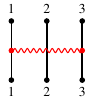}}\quad=\quad
		\vcdraw{\includegraphics{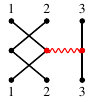}}\quad=\quad
		\vcdraw{\includegraphics{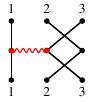}}
	\]\else
		\begin{tikzpicture}
			\permutation[tkzpic=0]{1,2,3}
			\tie[snake=true,snakelen=3,style=solid]{1,3}
		\end{tikzpicture}
		\begin{tikzpicture}
			\permutation[tkzpic=0,floor=1,type=1,bullb=0,height=0.5]{2,1,3}
			\permutation[tkzpic=0,type=-1,height=0.5]{2,1,3}
			\tie[snake=true,snakelen=3,style=solid]{2,3}
		\end{tikzpicture}
		\begin{tikzpicture}
			\permutation[tkzpic=0,floor=1,type=1,bullb=0,height=0.5]{1,3,2}
			\permutation[tkzpic=0,type=-1,height=0.5]{1,3,2}
			\tie[snake=true,snakelen=3,style=solid]{1,2}
		\end{tikzpicture}
	\fi
}

\newcommand{\figetn}{
	\ifnum\aspdf=1\[
		\includegraphics{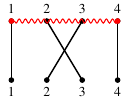}\qquad
		\includegraphics{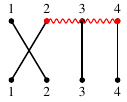}
	\]\else
		\begin{tikzpicture}
			\permutation[tkzpic=0,type=2]{1,3,2,4}
			\tie[snake=true,snakelen=3,style=solid]{{1,1},{4,1}}
		\end{tikzpicture}
		\begin{tikzpicture}
			\permutation[tkzpic=0,type=2]{2,1,3,4}
			\tie[snake=true,snakelen=3,style=solid]{{2,1},{4,1}}
		\end{tikzpicture}
	\]\fi
}

\newcommand{\figsvt}{
	\ifnum\aspdf=1\[\begin{array}{c}
		\vcdraw{\includegraphics[scale=0.9]{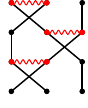}}\,\,\,=\,
		\vcdraw{\includegraphics[scale=0.9]{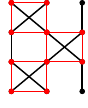}}\,\,\,=\,
		\vcdraw{\includegraphics[scale=0.9]{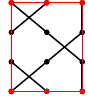}}\,\,\,=\,
		\vcdraw{\includegraphics[scale=0.9]{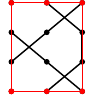}}\,\,\,=\,
		\vcdraw{\includegraphics[scale=0.9]{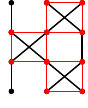}}\,\,\,=\,
		\vcdraw{\includegraphics[scale=0.9]{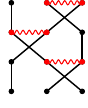}}\\[1cm]
		\vcdraw{\includegraphics[scale=0.9]{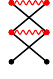}}\,\,\,=\,
		\vcdraw{\includegraphics[scale=0.9]{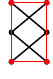}}\,\,\,=\,
		\vcdraw{\includegraphics[scale=0.9]{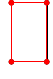}}\,\,\,=\,
		\vcdraw{\includegraphics[scale=0.9]{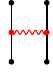}}\kern+3.5em
		\vcdraw{\includegraphics[scale=0.9]{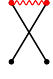}}\,\,\,=\,
		\vcdraw{\includegraphics[scale=0.9]{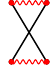}}\,\,\,=\,
		\vcdraw{\includegraphics[scale=0.9]{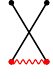}}
	\end{array}\]\else
		\begin{tikzpicture}
			\vpartition[tkzpic=0,height=0.5,type=0,floor=2]{{1,-2},{2,-1},{3,-3}}
			\vpartition[tkzpic=0,height=0.5,type=0,floor=1,bulla=0,bullb=0]{{1,-1},{2,-3},{3,-2}}
			\vpartition[tkzpic=0,height=0.5,type=0]{{1,-2},{2,-1},{3,-3}}
			\tie[snake=true,snakelen=3,style=solid]{{1,0.5},{2,0.5}}
			\tie[snake=true,snakelen=3,style=solid]{{2,1},{3,1}}
			\tie[snake=true,snakelen=3,style=solid]{{1,1.5},{2,1.5}}
		\end{tikzpicture}
		\begin{tikzpicture}
			\vpartition[tkzpic=0,height=0.5,type=0,floor=2]{{1,-2},{2,-1},{3,-3}}
			\vpartition[tkzpic=0,height=0.5,type=0,floor=1,bulla=0,bullb=0]{{1,-1},{2,-3},{3,-2}}
			\vpartition[tkzpic=0,height=0.5,type=0]{{1,-2},{2,-1},{3,-3}}
			\tie[style=solid]{{1,0.5},{2,0.5},{2,0},{1,0},{1,0.5}}
			\tie[style=solid]{{2,1},{3,1},{3,0.5},{2,0.5},{2,1}}
			\tie[style=solid]{{1,1.5},{2,1.5},{2,1},{1,1},{1,1.5}}
		\end{tikzpicture}
		\begin{tikzpicture}
			\vpartition[tkzpic=0,height=0.5,type=0,floor=2]{{1,-2},{2,-1},{3,-3}}
			\vpartition[tkzpic=0,height=0.5,type=0,floor=1,bulla=0,bullb=0]{{1,-1},{2,-3},{3,-2}}
			\vpartition[tkzpic=0,height=0.5,type=0]{{1,-2},{2,-1},{3,-3}}
			\tie[style=solid]{{1,0},{2,0},{3,0},{3,1.5},{2,1.5},{1,1.5},{1,0}}
		\end{tikzpicture}
		\begin{tikzpicture}
			\vpartition[tkzpic=0,height=0.5,type=0,floor=2]{{2,-3},{3,-2},{1,-1}}
			\vpartition[tkzpic=0,height=0.5,type=0,floor=1,bulla=0,bullb=0]{{3,-3},{1,-2},{2,-1}}
			\vpartition[tkzpic=0,height=0.5,type=0]{{2,-3},{3,-2},{1,-1}}
			\tie[style=solid]{{1,0},{2,0},{3,0},{3,1.5},{2,1.5},{1,1.5},{1,0}}
		\end{tikzpicture}
		\begin{tikzpicture}
			\vpartition[tkzpic=0,height=0.5,type=0,floor=2]{{2,-3},{3,-2},{1,-1}}
			\vpartition[tkzpic=0,height=0.5,type=0,floor=1,bulla=0,bullb=0]{{3,-3},{1,-2},{2,-1}}
			\vpartition[tkzpic=0,height=0.5,type=0]{{2,-3},{3,-2},{1,-1}}
			\tie[style=solid]{{2,0.5},{3,0.5},{3,0},{2,0},{2,0.5}}
			\tie[style=solid]{{1,1},{2,1},{2,0.5},{1,0.5},{1,1}}
			\tie[style=solid]{{2,1.5},{3,1.5},{3,1},{2,1},{2,1.5}}
		\end{tikzpicture}
		\begin{tikzpicture}
			\vpartition[tkzpic=0,height=0.5,type=0,floor=2]{{2,-3},{3,-2},{1,-1}}
			\vpartition[tkzpic=0,height=0.5,type=0,floor=1,bulla=0,bullb=0]{{3,-3},{1,-2},{2,-1}}
			\vpartition[tkzpic=0,height=0.5,type=0]{{2,-3},{3,-2},{1,-1}}
			\tie[snake=true,snakelen=3,style=solid]{{2,0.5},{3,0.5}}
			\tie[snake=true,snakelen=3,style=solid]{{1,1},{2,1}}
			\tie[snake=true,snakelen=3,style=solid]{{2,1.5},{3,1.5}}
		\end{tikzpicture}
		\begin{tikzpicture}
			\vpartition[tkzpic=0,height=0.5,type=0,floor=1,bullb=0]{{1,-2},{2,-1},{3,-3},{4,-4}}
			\vpartition[tkzpic=0,height=0.5,type=0]{{1,-1},{2,-2},{3,-4},{4,-3}}
			\tie[snake=true,snakelen=3,style=solid]{{1,1},{2,1}}
			\tie[snake=true,snakelen=3,style=solid]{{3,0.5},{4,0.5}}
		\end{tikzpicture}
		\begin{tikzpicture}
			\vpartition[tkzpic=0,height=0.5,type=0,floor=1,bullb=0]{{1,-2},{2,-1},{3,-3},{4,-4}}
			\vpartition[tkzpic=0,height=0.5,type=0]{{1,-1},{2,-2},{3,-4},{4,-3}}
			\tie[style=solid]{{1,1},{2,1},{2,0},{1,0},{1,1}}
			\tie[style=solid]{{3,1},{4,1},{4,0},{3,0},{3,1}}
		\end{tikzpicture}
		\begin{tikzpicture}
			\vpartition[tkzpic=0,height=0.5,type=0,floor=1,bullb=0]{{3,-4},{4,-3},{1,-1},{2,-2}}
			\vpartition[tkzpic=0,height=0.5,type=0]{{3,-3},{4,-4},{1,-2},{2,-1}}
			\tie[style=solid]{{1,1},{2,1},{2,0},{1,0},{1,1}}
			\tie[style=solid]{{3,1},{4,1},{4,0},{3,0},{3,1}}
		\end{tikzpicture}
		\begin{tikzpicture}
			\vpartition[tkzpic=0,height=0.5,type=0,floor=1,bullb=0]{{3,-4},{4,-3},{1,-1},{2,-2}}
			\vpartition[tkzpic=0,height=0.5,type=0]{{3,-3},{4,-4},{1,-2},{2,-1}}
			\tie[snake=true,snakelen=3,style=solid]{{3,1},{4,1}}
			\tie[snake=true,snakelen=3,style=solid]{{1,0.5},{2,0.5}}
		\end{tikzpicture}
		\begin{tikzpicture}
			\vpartition[tkzpic=0,height=0.5,type=0,floor=1,bullb=0]{{1,-2},{2,-1}}
			\vpartition[tkzpic=0,height=0.5,type=0]{{1,-2},{2,-1}}
			\tie[snake=true,snakelen=3,style=solid]{{1,1},{2,1}}
			\tie[snake=true,snakelen=3,style=solid]{{1,0.5},{2,0.5}}
		\end{tikzpicture}
		\begin{tikzpicture}
			\vpartition[tkzpic=0,height=0.5,type=0,floor=1,bullb=0]{{1,-2},{2,-1}}
			\vpartition[tkzpic=0,height=0.5,type=0]{{1,-2},{2,-1}}
			\tie[style=solid]{{1,1},{2,1},{2,0},{1,0},{1,1}}
		\end{tikzpicture}
		\begin{tikzpicture}
			\vpartition[tkzpic=0,type=0]{{1,-1},{2,-2}}
			\tie[style=solid]{{1,1},{2,1},{2,0},{1,0},{1,1}}
		\end{tikzpicture}
		\begin{tikzpicture}
			\vpartition[tkzpic=0,type=0]{{1,-1},{2,-2}}
			\tie[snake=true,snakelen=3,style=solid]{{1,0.5},{2,0.5}}
		\end{tikzpicture}
		\begin{tikzpicture}
			\vpartition[tkzpic=0,type=0]{{1,-2},{2,-1}}
			\tie[snake=true,snakelen=3,style=solid]{{1,1},{2,1}}
			\tie[snake=true,snakelen=3,style=solid]{{1,0},{2,0}}
		\end{tikzpicture}
		\begin{tikzpicture}
			\vpartition[tkzpic=0,type=0]{{1,-2},{2,-1}}
			\tie[snake=true,snakelen=3,style=solid]{{1,1},{2,1}}
		\end{tikzpicture}
		\begin{tikzpicture}
			\vpartition[tkzpic=0,type=0]{{1,-2},{2,-1}}
			\tie[snake=true,snakelen=3,style=solid]{{1,0},{2,0}}
		\end{tikzpicture}
	\fi
}

\newcommand{\figstn}{
	\ifnum\aspdf=1\[
		\vcdraw{\includegraphics{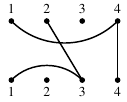}}\,\,*\,
		\vcdraw{\includegraphics{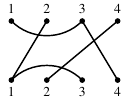}}\,\,\,=\,\,
		\vcdraw{\includegraphics{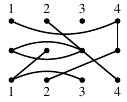}}\,\,\,=\,\,
		\vcdraw{\includegraphics{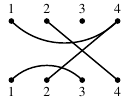}}
	\]\else
		\vpartition[type=2]{{1,4,-4},{2,-3,-1}}
		\vpartition[type=2]{{1,3,-4},{2,-1,-3},{4,-2}}
		\begin{tikzpicture}
			\vpartition[tkzpic=0,height=0.5,bend=25,type=1,floor=1]{{1,4,-4},{2,-3,-1}}
			\vpartition[tkzpic=0,height=0.5,bend=25,bulla=0,type=-1]{{1,3,-4},{2,-1,-3},{4,-2}}
		\end{tikzpicture}
		\vpartition[type=2]{{1,4,-2},{2,-4},{-1,-3}}
	\fi
}

\newcommand{\figfif}{
	\ifnum\aspdf=1\[
		\vcdraw{\includegraphics{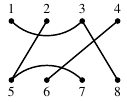}}\,\,=\,
		\vcdraw{\includegraphics{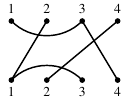}}
	\]\else
		\vpartition{{1,3,-4},{2,-1,-3},{4,-2}}
		\vpartition[type=2]{{1,3,-4},{2,-1,-3},{4,-2}}
	\fi
}

\newcommand{\figftn}{
	\centering
	\ifnum\aspdf=1
		\includegraphics[scale=0.9]{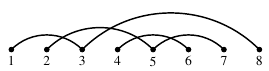}
	\else
		\arcpartition{{1,3,8},{2,5,7},{4,6}}
	\fi
}

\newcommand{\figthi}{
	\centering
	\ifnum\aspdf=1
		\includegraphics{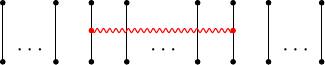}
	\else
		\begin{tikzpicture}
			\tie[style=solid,color=black]{{-0.5,0},{-0.5,1}}
			\tie[style=solid,color=black]{{1,0},{1,1}}
			\tie[style=solid,color=black]{{2,0},{2,1}}
			\tie[style=solid,color=black]{{3,0},{3,1}}
			\tie[style=solid,color=black]{{5,0},{5,1}}
			\tie[style=solid,color=black]{{6,0},{6,1}}
			\tie[style=solid,color=black]{{7,0},{7,1}}
			\tie[style=solid,color=black]{{8.5,0},{8.5,1}}
			\tie[snake=true,snakelen=3,style=solid]{{2,0.53},{6,0.53}}
			\node at(-0.41,0.2){$\cdots$};
			\node at(1.85,0.2){$\cdots$};
			\node at(4.09,0.2){$\cdots$};
		\end{tikzpicture}		
	\fi
}

\newcommand{\figtwe}{
	\centering
	\ifnum\aspdf=1
		\includegraphics[scale=0.9]{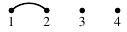}\kern+1.5em
		\includegraphics[scale=0.9]{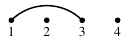}\kern+1.5em
		\includegraphics[scale=0.9]{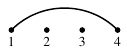}\kern+1.5em
		\includegraphics[scale=0.9]{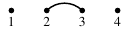}\kern+1.5em
		\includegraphics[scale=0.9]{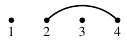}\kern+1.5em
		\includegraphics[scale=0.9]{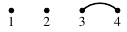}
	\else
		\arcpartition[num=4]{{1,2}}
		\arcpartition[num=4]{{1,3}}
		\arcpartition[num=4]{{1,4}}
		\arcpartition[num=4]{{2,3}}
		\arcpartition[num=4]{{2,4}}
		\arcpartition[num=4]{{3,4}}
	\fi
}

\newcommand{\figele}{
	\ifnum\aspdf=1\[\begin{array}{c}
		\left(
			\vcdraw{\includegraphics{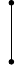}}\cdots
			\vcdraw{\includegraphics{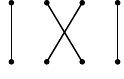}}\,\cdots
			\vcdraw{\includegraphics{pics/002.pdf}}\,\,\,\,,\,\,\,
			\vcdraw{\includegraphics{pics/002.pdf}}\cdots
			\vcdraw{\includegraphics{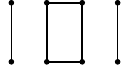}}\,\cdots
			\vcdraw{\includegraphics{pics/002.pdf}}
		\right)\\[1cm]
		\vcdraw{\includegraphics{pics/002.pdf}}\cdots
		\vcdraw{\includegraphics{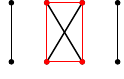}}\,\cdots
		\vcdraw{\includegraphics{pics/002.pdf}}\quad=\quad
		\vcdraw{\includegraphics{pics/002.pdf}}\cdots
		\vcdraw{\includegraphics{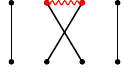}}\,\cdots
		\vcdraw{\includegraphics{pics/002.pdf}}
	\end{array}\]\else
		\vpartition[type=0,bend=0]{{1,-1},{2,3,-3,-2,2},{4,-4}}
		\vpartition[type=0]{{1,-1},{4,-4},{2,-3},{3,-2}}
		\begin{tikzpicture}
			\vpartition[tkzpic=0,type=0]{{1,-1},{4,-4},{2,-3},{3,-2}}
			\tie[style=solid]{{2,1},{3,1}}
			\tie[style=solid]{{2,0},{3,0}}
			\tie[style=solid,bull=0]{{2,1},{2,0}}
			\tie[style=solid,bull=0]{{3,1},{3,0}}
		\end{tikzpicture}
		\begin{tikzpicture}
			\vpartition[tkzpic=0,type=0]{{1,-1},{4,-4},{2,-3},{3,-2}}
			\tie[snake=true,snakelen=3,style=solid]{{2,1},{3,1}}
		\end{tikzpicture}		
	\fi
}

\newcommand{\figten}{
	\ifnum\aspdf=1\[\begin{array}{c}
		\vcdraw{\includegraphics[scale=0.9]{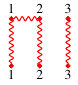}}\quad
		\vcdraw{\includegraphics[scale=0.9]{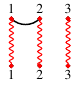}}\quad
		\vcdraw{\includegraphics[scale=0.9]{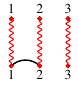}}\quad
		\vcdraw{\includegraphics[scale=0.9]{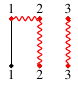}}\quad
		\vcdraw{\includegraphics[scale=0.9]{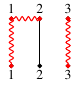}}\quad
		\vcdraw{\includegraphics[scale=0.9]{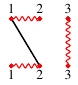}}\quad
		\vcdraw{\includegraphics[scale=0.9]{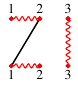}}\quad
		\vcdraw{\includegraphics[scale=0.9]{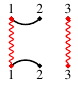}}\quad
		\vcdraw{\includegraphics[scale=0.9]{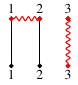}}\quad
		\vcdraw{\includegraphics[scale=0.9]{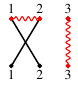}}\\[0.7cm]
		\vcdraw{\includegraphics[scale=0.9]{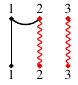}}\quad
		\vcdraw{\includegraphics[scale=0.9]{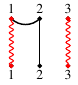}}\quad
		\vcdraw{\includegraphics[scale=0.9]{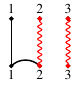}}\quad
		\vcdraw{\includegraphics[scale=0.9]{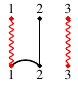}}\quad
		\vcdraw{\includegraphics[scale=0.9]{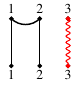}}\quad
		\vcdraw{\includegraphics[scale=0.9]{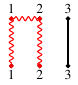}}\quad
		\vcdraw{\includegraphics[scale=0.9]{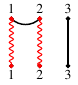}}\quad
		\vcdraw{\includegraphics[scale=0.9]{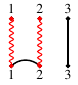}}\quad
		\vcdraw{\includegraphics[scale=0.9]{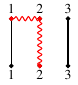}}\quad
		\vcdraw{\includegraphics[scale=0.9]{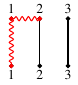}}\\[0.7cm]
		\vcdraw{\includegraphics[scale=0.9]{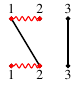}}\quad
		\vcdraw{\includegraphics[scale=0.9]{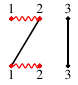}}\quad
		\vcdraw{\includegraphics[scale=0.9]{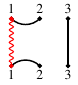}}\quad
		\vcdraw{\includegraphics[scale=0.9]{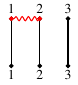}}\quad
		\vcdraw{\includegraphics[scale=0.9]{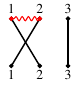}}\quad
		\vcdraw{\includegraphics[scale=0.9]{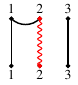}}\quad
		\vcdraw{\includegraphics[scale=0.9]{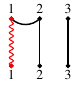}}\quad
		\vcdraw{\includegraphics[scale=0.9]{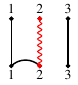}}\quad
		\vcdraw{\includegraphics[scale=0.9]{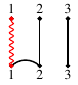}}\quad
		\vcdraw{\includegraphics[scale=0.9]{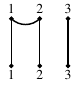}}
	\end{array}\]\else
		\begin{tikzpicture}[scale=0.8]
			\vpartition[type=2,tkzpic=0]{{1},{2},{-1},{-2},{3}}
			\tie[snake=true,snakelen=3,style=solid]{{3,0},{3,1}}
			\tie[snake=true,snakelen=3,style=solid]{{1,0},{1,1},{2,1},{2,0}}
		\end{tikzpicture}
		\begin{tikzpicture}[scale=0.8]
			\vpartition[type=2,tkzpic=0]{{1,2},{-1},{-2},{3}}
			\tie[snake=true,snakelen=3,style=solid]{{3,0},{3,1}}
			\tie[snake=true,snakelen=3,style=solid]{{1,0},{1,1}}
			\tie[snake=true,snakelen=3,style=solid]{{2,1},{2,0}}
		\end{tikzpicture}
		\begin{tikzpicture}[scale=0.8]
			\vpartition[type=2,tkzpic=0]{{1},{2},{-1,-2},{3}}
			\tie[snake=true,snakelen=3,style=solid]{{3,0},{3,1}}
			\tie[snake=true,snakelen=3,style=solid]{{1,0},{1,1}}
			\tie[snake=true,snakelen=3,style=solid]{{2,1},{2,0}}
		\end{tikzpicture}
		\begin{tikzpicture}[scale=0.8]
			\vpartition[type=2,tkzpic=0]{{1,-1},{2},{-2},{3}}
			\tie[snake=true,snakelen=3,style=solid]{{3,0},{3,1}}
			\tie[snake=true,snakelen=3,style=solid]{{1,1},{2,1},{2,0}}
		\end{tikzpicture}
		\begin{tikzpicture}[scale=0.8]
			\vpartition[type=2,tkzpic=0]{{1},{2,-2},{-1},{3}}
			\tie[snake=true,snakelen=3,style=solid]{{3,0},{3,1}}
			\tie[snake=true,snakelen=3,style=solid]{{1,0},{1,1},{2,1}}
		\end{tikzpicture}
		\begin{tikzpicture}[scale=0.8]
			\vpartition[type=2,tkzpic=0]{{1,-2},{2},{-1},{3}}
			\tie[snake=true,snakelen=3,style=solid]{{3,0},{3,1}}
			\tie[snake=true,snakelen=3,style=solid]{{1,1},{2,1}}
			\tie[snake=true,snakelen=3,style=solid]{{1,0},{2,0}}
		\end{tikzpicture}
		\begin{tikzpicture}[scale=0.8]
			\vpartition[type=2,tkzpic=0]{{1},{2,-1},{-2},{3}}
			\tie[snake=true,snakelen=3,style=solid]{{3,0},{3,1}}
			\tie[snake=true,snakelen=3,style=solid]{{1,1},{2,1}}
			\tie[snake=true,snakelen=3,style=solid]{{1,0},{2,0}}
		\end{tikzpicture}
		\begin{tikzpicture}[scale=0.8]
			\vpartition[type=2,tkzpic=0]{{1,2},{-1,-2},{3}}
			\tie[snake=true,snakelen=3,style=solid]{{3,0},{3,1}}
			\tie[snake=true,snakelen=3,style=solid]{{1,0},{1,1}}
		\end{tikzpicture}
		\begin{tikzpicture}[scale=0.8]
			\vpartition[type=2,tkzpic=0]{{1,-1},{2,-2},{3}}
			\tie[snake=true,snakelen=3,style=solid]{{3,0},{3,1}}
			\tie[snake=true,snakelen=3,style=solid]{{1,1},{2,1}}
		\end{tikzpicture}
		\begin{tikzpicture}[scale=0.8]
			\vpartition[type=2,tkzpic=0]{{1,-2},{2,-1},{3}}
			\tie[snake=true,snakelen=3,style=solid]{{3,0},{3,1}}
			\tie[snake=true,snakelen=3,style=solid]{{1,1},{2,1}}
		\end{tikzpicture}
		\begin{tikzpicture}[scale=0.8]
			\vpartition[type=2,tkzpic=0]{{-1,1,2},{-2},{3}}
			\tie[snake=true,snakelen=3,style=solid]{{3,0},{3,1}}
			\tie[snake=true,snakelen=3,style=solid]{{2,1},{2,0}}
		\end{tikzpicture}
		\begin{tikzpicture}[scale=0.8]
			\vpartition[type=2,tkzpic=0]{{1,2,-2},{-1},{3}}
			\tie[snake=true,snakelen=3,style=solid]{{3,0},{3,1}}
			\tie[snake=true,snakelen=3,style=solid]{{1,1},{1,0}}
		\end{tikzpicture}
		\begin{tikzpicture}[scale=0.8]
			\vpartition[type=2,tkzpic=0]{{2},{1,-1,-2},{3}}
			\tie[snake=true,snakelen=3,style=solid]{{3,0},{3,1}}
			\tie[snake=true,snakelen=3,style=solid]{{2,1},{2,0}}
		\end{tikzpicture}
		\begin{tikzpicture}[scale=0.8]
			\vpartition[type=2,tkzpic=0]{{1},{-1,-2,2},{3}}
			\tie[snake=true,snakelen=3,style=solid]{{3,0},{3,1}}
			\tie[snake=true,snakelen=3,style=solid]{{1,1},{1,0}}
		\end{tikzpicture}
		\begin{tikzpicture}[scale=0.8]
			\vpartition[type=2,tkzpic=0]{{-1,1,2,-2},{3}}
			\tie[snake=true,snakelen=3,style=solid]{{3,0},{3,1}}
		\end{tikzpicture}
		\begin{tikzpicture}[scale=0.8]
			\vpartition[type=2,tkzpic=0]{{1},{2},{-1},{-2},{3,-3}}
			\tie[snake=true,snakelen=3,style=solid]{{1,0},{1,1},{2,1},{2,0}}
		\end{tikzpicture}
		\begin{tikzpicture}[scale=0.8]
			\vpartition[type=2,tkzpic=0]{{1,2},{-1},{-2},{3,-3}}
			\tie[snake=true,snakelen=3,style=solid]{{1,0},{1,1}}
			\tie[snake=true,snakelen=3,style=solid]{{2,1},{2,0}}
		\end{tikzpicture}
		\begin{tikzpicture}[scale=0.8]
			\vpartition[type=2,tkzpic=0]{{1},{2},{-1,-2},{3,-3}}
			\tie[snake=true,snakelen=3,style=solid]{{1,0},{1,1}}
			\tie[snake=true,snakelen=3,style=solid]{{2,1},{2,0}}
		\end{tikzpicture}
		\begin{tikzpicture}[scale=0.8]
			\vpartition[type=2,tkzpic=0]{{1,-1},{2},{-2},{3,-3}}
			\tie[snake=true,snakelen=3,style=solid]{{1,1},{2,1},{2,0}}
		\end{tikzpicture}
		\begin{tikzpicture}[scale=0.8]
			\vpartition[type=2,tkzpic=0]{{1},{2,-2},{-1},{3,-3}}
			\tie[snake=true,snakelen=3,style=solid]{{1,0},{1,1},{2,1}}
		\end{tikzpicture}
		\begin{tikzpicture}[scale=0.8]
			\vpartition[type=2,tkzpic=0]{{1,-2},{2},{-1},{3,-3}}
			\tie[snake=true,snakelen=3,style=solid]{{1,1},{2,1}}
			\tie[snake=true,snakelen=3,style=solid]{{1,0},{2,0}}
		\end{tikzpicture}
		\begin{tikzpicture}[scale=0.8]
			\vpartition[type=2,tkzpic=0]{{1},{2,-1},{-2},{3,-3}}
			\tie[snake=true,snakelen=3,style=solid]{{1,1},{2,1}}
			\tie[snake=true,snakelen=3,style=solid]{{1,0},{2,0}}
		\end{tikzpicture}
		\begin{tikzpicture}[scale=0.8]
			\vpartition[type=2,tkzpic=0]{{1,2},{-1,-2},{3,-3}}
			\tie[snake=true,snakelen=3,style=solid]{{1,0},{1,1}}
		\end{tikzpicture}
		\begin{tikzpicture}[scale=0.8]
			\vpartition[type=2,tkzpic=0]{{1,-1},{2,-2},{3,-3}}
			\tie[snake=true,snakelen=3,style=solid]{{1,1},{2,1}}
		\end{tikzpicture}
		\begin{tikzpicture}[scale=0.8]
			\vpartition[type=2,tkzpic=0]{{1,-2},{2,-1},{3,-3}}
			\tie[snake=true,snakelen=3,style=solid]{{1,1},{2,1}}
		\end{tikzpicture}
		\begin{tikzpicture}[scale=0.8]
			\vpartition[type=2,tkzpic=0]{{-1,1,2},{-2},{3,-3}}
			\tie[snake=true,snakelen=3,style=solid]{{2,1},{2,0}}
		\end{tikzpicture}
		\begin{tikzpicture}[scale=0.8]
			\vpartition[type=2,tkzpic=0]{{1,2,-2},{-1},{3,-3}}
			\tie[snake=true,snakelen=3,style=solid]{{1,1},{1,0}}
		\end{tikzpicture}
		\begin{tikzpicture}[scale=0.8]
			\vpartition[type=2,tkzpic=0]{{2},{1,-1,-2},{3,-3}}
			\tie[snake=true,snakelen=3,style=solid]{{2,1},{2,0}}
		\end{tikzpicture}
		\begin{tikzpicture}[scale=0.8]
			\vpartition[type=2,tkzpic=0]{{1},{-1,-2,2},{3,-3}}
			\tie[snake=true,snakelen=3,style=solid]{{1,1},{1,0}}
		\end{tikzpicture}
		\begin{tikzpicture}[scale=0.8]
			\vpartition[type=2,tkzpic=0]{{-1,1,2,-2},{3,-3}}
		\end{tikzpicture}
	\fi
}

\newcommand{\fignin}{
	\ifnum\aspdf=1\[\begin{array}{c}
		\vcdraw{\includegraphics[scale=0.9]{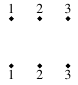}}\quad
		\vcdraw{\includegraphics[scale=0.9]{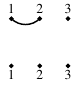}}\quad
		\vcdraw{\includegraphics[scale=0.9]{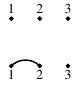}}\quad
		\vcdraw{\includegraphics[scale=0.9]{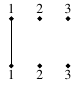}}\quad
		\vcdraw{\includegraphics[scale=0.9]{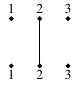}}\quad
		\vcdraw{\includegraphics[scale=0.9]{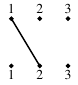}}\quad
		\vcdraw{\includegraphics[scale=0.9]{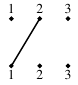}}\quad
		\vcdraw{\includegraphics[scale=0.9]{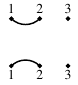}}\quad
		\vcdraw{\includegraphics[scale=0.9]{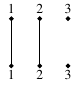}}\quad
		\vcdraw{\includegraphics[scale=0.9]{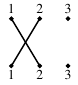}}\\[0.7cm]
		\vcdraw{\includegraphics[scale=0.9]{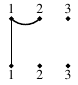}}\quad
		\vcdraw{\includegraphics[scale=0.9]{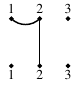}}\quad
		\vcdraw{\includegraphics[scale=0.9]{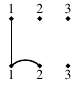}}\quad
		\vcdraw{\includegraphics[scale=0.9]{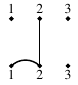}}\quad
		\vcdraw{\includegraphics[scale=0.9]{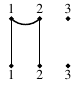}}\quad
		\vcdraw{\includegraphics[scale=0.9]{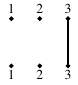}}\quad
		\vcdraw{\includegraphics[scale=0.9]{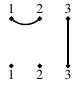}}\quad
		\vcdraw{\includegraphics[scale=0.9]{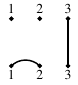}}\quad
		\vcdraw{\includegraphics[scale=0.9]{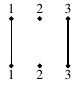}}\quad
		\vcdraw{\includegraphics[scale=0.9]{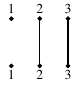}}\\[0.7cm]
		\vcdraw{\includegraphics[scale=0.9]{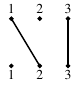}}\quad
		\vcdraw{\includegraphics[scale=0.9]{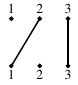}}\quad
		\vcdraw{\includegraphics[scale=0.9]{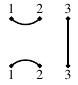}}\quad
		\vcdraw{\includegraphics[scale=0.9]{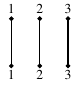}}\quad
		\vcdraw{\includegraphics[scale=0.9]{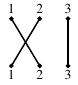}}\quad
		\vcdraw{\includegraphics[scale=0.9]{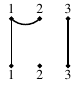}}\quad
		\vcdraw{\includegraphics[scale=0.9]{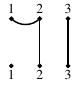}}\quad
		\vcdraw{\includegraphics[scale=0.9]{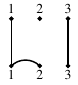}}\quad
		\vcdraw{\includegraphics[scale=0.9]{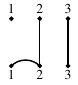}}\quad
		\vcdraw{\includegraphics[scale=0.9]{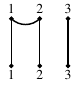}}
	\end{array}\]\else
		\vpartition[type=2]{{1},{2},{-1},{-2},{3}}
		\vpartition[type=2]{{1,2},{-1},{-2},{3}}
		\vpartition[type=2]{{1},{2},{-1,-2},{3}}
		\vpartition[type=2]{{1,-1},{2},{-2},{3}}
		\vpartition[type=2]{{1},{2,-2},{-1},{3}}
		\vpartition[type=2]{{1,-2},{2},{-1},{3}}
		\vpartition[type=2]{{1},{2,-1},{-2},{3}}
		\vpartition[type=2]{{1,2},{-1,-2},{3}}
		\vpartition[type=2]{{1,-1},{2,-2},{3}}
		\vpartition[type=2]{{1,-2},{2,-1},{3}}
		\vpartition[type=2]{{-1,1,2},{-2},{3}}
		\vpartition[type=2]{{1,2,-2},{-1},{3}}
		\vpartition[type=2]{{2},{1,-1,-2},{3}}
		\vpartition[type=2]{{1},{-1,-2,2},{3}}
		\vpartition[type=2]{{-1,1,2,-2},{3}}
		\vpartition[type=2]{{1},{2},{-1},{-2},{3,-3}}
		\vpartition[type=2]{{1,2},{-1},{-2},{3,-3}}
		\vpartition[type=2]{{1},{2},{-1,-2},{3,-3}}
		\vpartition[type=2]{{1,-1},{2},{-2},{3,-3}}
		\vpartition[type=2]{{1},{2,-2},{-1},{3,-3}}
		\vpartition[type=2]{{1,-2},{2},{-1},{3,-3}}
		\vpartition[type=2]{{1},{2,-1},{-2},{3,-3}}
		\vpartition[type=2]{{1,2},{-1,-2},{3,-3}}
		\vpartition[type=2]{{1,-1},{2,-2},{3,-3}}
		\vpartition[type=2]{{1,-2},{2,-1},{3,-3}}
		\vpartition[type=2]{{-1,1,2},{-2},{3,-3}}
		\vpartition[type=2]{{1,2,-2},{-1},{3,-3}}
		\vpartition[type=2]{{2},{1,-1,-2},{3,-3}}
		\vpartition[type=2]{{1},{-1,-2,2},{3,-3}}
		\vpartition[type=2]{{-1,1,2,-2},{3,-3}}
	\fi
}

\newcommand{\figeig}{
	\centering
	\ifnum\aspdf=1
		\includegraphics{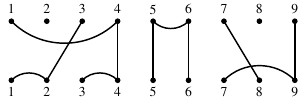}
	\else
		\vpartition[type=2]{{1,4,-4,-3},{3,-2,-1},{-5,5,6,-6},{-7,-9,9},{-8,7}}
	\fi
}

\newcommand{\figsev}{
	\ifnum\aspdf=1\[
		\vcdraw{\includegraphics{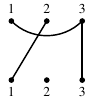}}\kern+0.45em/\kern+0.45em
		\vcdraw{\includegraphics{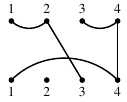}}\kern+0.55em=\kern+0.45em
		\vcdraw{\includegraphics{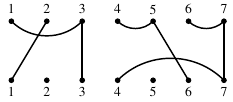}}
	\]\else
		\vpartition[type=2]{{1,3,-3},{2,-1}}
		\vpartition[type=2]{{-1,-4,4,3},{-3,2,1}}
		\vpartition[type=2]{{1,3,-3},{2,-1},{-4,-7,7,6},{-6,5,4}}
	\fi
}

\newcommand{\figsix}{
	\ifnum\aspdf=1\[\begin{array}{c}
		\left(
			\vcdraw{\includegraphics{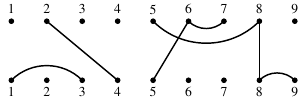}}\,\,\,\,,\,\,\,
			\vcdraw{\includegraphics{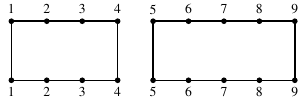}}
		\right)\\[1cm]
		\vcdraw{\includegraphics{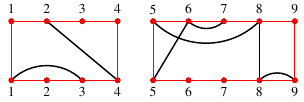}}\quad=\quad
		\vcdraw{\includegraphics{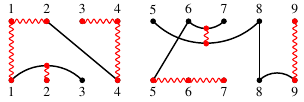}}
	\end{array}\]\else
		\vpartition[type=2]{{2,-4},{5,8,-8,-9},{-1,-3},{-5,6,7}}
		\vpartition[type=2,bend=0]{{1,2,3,4,-4,-3,-2,-1,1},{5,6,7,8,9,-9,-8,-7,-6,-5,5}}
		\begin{tikzpicture}
			\vpartition[type=2,tkzpic=0]{{2,-4},{5,8,-8,-9},{-1,-3},{-5,6,7}}
			\tie[style=solid]{{1,1},{2,1},{3,1},{4,1},{4,0},{3,0},{2,0},{1,0},{1,1}}
			\tie[style=solid]{{5,1},{6,1},{7,1},{8,1},{9,1},{9,0},{8,0},{7,0},{6,0},{5,0},{5,1}}
		\end{tikzpicture}
		\begin{tikzpicture}
			\vpartition[type=2,tkzpic=0]{{2,-4},{5,8,-8,-9},{-1,-3},{-5,6,7}}
			\tie[snake=true,snakelen=3,style=solid]{{1,0},{1,1},{2,1}}
			\tie[snake=true,snakelen=3,style=solid]{{3,1},{4,1},{4,0}}
			\tie[snake=true,snakelen=3,style=solid]{{2,0},{2,0.25}}
			\tie[snake=true,snakelen=3,style=solid]{{9,1},{9,0}}
			\tie[snake=true,snakelen=3,style=solid]{{5,0},{6,0},{7,0}}
			\tie[snake=true,snakelen=3,style=solid]{{6.5,0.88},{6.5,0.62}}
		\end{tikzpicture}
	\fi
}

\newcommand{\figfiv}{
	\ifnum\aspdf=1\[\begin{array}{c}
		\left(
			\vcdraw{\includegraphics{pics/002.pdf}}\cdots
			\vcdraw{\includegraphics{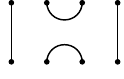}}\,\cdots
			\vcdraw{\includegraphics{pics/002.pdf}}\,\,\,\,,\,\,\,
			\vcdraw{\includegraphics{pics/002.pdf}}\cdots
			\vcdraw{\includegraphics{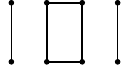}}\,\cdots
			\vcdraw{\includegraphics{pics/002.pdf}}
		\right)\\[1cm]
		\vcdraw{\includegraphics{pics/002.pdf}}\cdots
		\vcdraw{\includegraphics{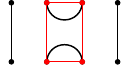}}\,\cdots
		\vcdraw{\includegraphics{pics/002.pdf}}\quad=\quad
		\vcdraw{\includegraphics{pics/002.pdf}}\cdots
		\vcdraw{\includegraphics{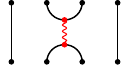}}\,\cdots
		\vcdraw{\includegraphics{pics/002.pdf}}
	\end{array}\]\else
		\vpartition[type=0,bend=0]{{1,-1},{2,3,-3,-2,2},{4,-4}}
		\strands[type=0,nstr=4]{t2}
		\begin{tikzpicture}
			\strands[tkzpic=0,type=0,nstr=4]{t2}
			\tie[style=solid]{{2,1},{3,1}}
			\tie[style=solid]{{2,0},{3,0}}
			\tie[style=solid,bull=0]{{2,1},{2,0}}
			\tie[style=solid,bull=0]{{3,1},{3,0}}
		\end{tikzpicture}
		\begin{tikzpicture}
			\strands[tkzpic=0,type=0,nstr=4]{t2}
			\tie[snake=true,snakelen=3,style=solid]{{2.5,0.29},{2.5,0.71}}
		\end{tikzpicture}		
	\fi
}

\newcommand{\figfou}{
	\centering
	\ifnum\aspdf=1
		\includegraphics{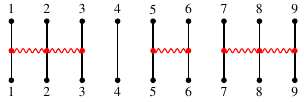}
	\else
		\begin{tikzpicture}
			\permutation[tkzpic=0]{1,2,3,4,5,6,7,8,9}
			\tie[snake=true,snakelen=3,style=solid]{1,2,3}
			\tie[snake=true,snakelen=3,style=solid]{5,6}
			\tie[snake=true,snakelen=3,style=solid]{7,8,9}
		\end{tikzpicture}
	\fi
}

\newcommand{\figthr}{
	\centering
	\ifnum\aspdf=1\[
		\vcdraw{\includegraphics{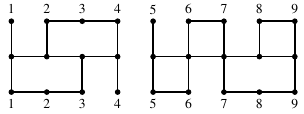}}\quad=\!\quad
		\vcdraw{\includegraphics{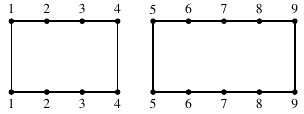}}
	\]\else
		\begin{tikzpicture}
			\vpartition[bend=0,floor=1,tkzpic=0,type=1,bullb=0,height=0.6]{{1,-1},{2,3,4,-4,-3,-2,2},{5,-5},{6,7,-7,-6,6},{8,9,-9,-8,8}}
			\vpartition[bend=0,tkzpic=0,type=-1,height=0.6]{{1,2,3,-3,-2,-1,1},{4,-4},{5,6,-6,-5,5},{7,8,9,-9,-8,-7,7}}
		\end{tikzpicture}
		\vpartition[bend=0,type=2,height=1.2]{{1,2,3,4,-4,-3,-2,-1,1},{5,6,7,8,9,-9,-8,-7,-6,-5,5}}
	\fi
}

\newcommand{\figtwo}{
	\centering
	\ifnum\aspdf=1\[
		\vcdraw{\includegraphics{pics/002.pdf}}\cdots
		\vcdraw{\includegraphics{pics/003.pdf}}\,\cdots
		\vcdraw{\includegraphics{pics/002.pdf}}
	\]\else
		\vpartition[type=0]{{1,-1}}
		\vpartition[type=0,bend=0]{{1,-1},{2,3,-3,-2,2},{4,-4}}
	\fi
}

\newcommand{\figone}{
	\centering
	\ifnum\aspdf=1
		\includegraphics{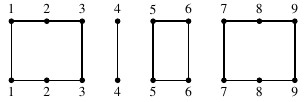}
	\else
		\vpartition[type=2,bend=0]{{1,2,3,-3,-2,-1,1},{4,-4},{5,6,-6,-5,5},{7,8,9,-9,-8,-7,7}}
	\fi
}

\newcommand{\figzer}{
	\centering
	\ifnum\aspdf=1
		\includegraphics{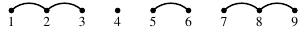}
	\else
		\arcpartition{{1,2,3},{5,6},{7,8,9}}
	\fi
}

\begin{document}

\begin{abstract}
We introduce two new algebras that we call \emph{tied--boxed Hecke algebra} and \emph{tied--boxed Temperley--Lieb algebra}. The first one is a subalgebra of the algebra of braids and ties introduced by Aicardi and Juyumaya, and the second one is a tied--version of the well known Temperley--Lieb algebra. We study their representation theory and give cellular bases for them. Furthermore, we explore a strong connection between the tied--boxed Temperley--Lieb algebra and the so--called partition Temperley--Lieb algebra given by Juyumaya. Also, we show that both structures inherit diagrammatic interpretations from a new class of monoids that we call \emph{boxed ramified monoids}. Additionally, we give presentations for the singular part of the ramified symmetric monoid and for the boxed ramified monoid associated to the Brauer monoid.
\end{abstract}

\maketitle

\setcounter{tocdepth}{2}

\section{Introduction}\label{097}

The \emph{algebra of braids and ties} or simply the \emph{bt--algebra} $\E_n(q)$ was introduced by Aicardi and Juyumaya in \cite{Ju99b,AiJu00} by means an abstracting procedure of the so--called \emph{Yokonuma--Hecke algebra} \cite{Yo67,JuKa01}. This algebra can be regarded as a generalization of the well known \emph{Iwahori--Hecke algebra} $\H_n(q)$ \cite{Iw64}, and has a rich interpretation as a diagram algebra involving both the classic \emph{braid generators} and a new family of generators, called \emph{ties}, which arise from the combinatorics of set partitions \cite{AiJu00,RyH11}. The bt--algebra was introduced with the purpose of constructing new representations of the braid group, however, in recent years, this algebra and its derivatives have been studied from various perspectives.

The representation theory of the algebra of braids and ties was initiated by Ryom-Hansen in \cite{RyH11}. Specifically, he defined a tensor representation of $\E_n(q)$ which used to classify the simple modules of this algebra. As a consequence, he obtained a linear basis for the bt--algebra and proved that its dimension is $n!\b_n$, where $\b_n$ denotes the $n$th \emph{Bell number} \cite[A000110]{OEIS}. These results where used by Aicardi and Juyumaya in \cite{AiJu16b} to show that $\E_n(q)$ supports a \emph{Markov trace} and so constructed polynomial invariants for classical and singular knots in the sense of Jones \cite{Jn87}. Aicardi and Juyumaya observed that their invariants can be recovered by a new class of knotted objects that called \emph{tied links}, which in turn can be obtained by closing a new class of braided object that called \emph{tied braids} \cite{AiJu16}. Tied braids (of $n$ strands) form a monoid $T\B_n$ called the \emph{tied braid monoid}, which play the role of algebraic counterpart for tied links just as the braid group is the algebraic counterpart of classical links by means the \emph{Alexander theorem} \cite{Al23a}. Furthermore, this monoid of tied braids was shown to be a semidirect product $T\B_n=\P_n\rtimes\B_n$, where $\P_n$ denotes the monoid of set partitions of $[n]=\{1,\ldots,n\}$ together with the classic product given by refinement \cite{AiJu21}. More results related to knot theory via the bt--algebra and its derivatives can be found in the literature, such as other invariants for tied and singular links \cite{AiJu18,AiJu21}, generalizations of the bt--algebra \cite{AiJu20,Fl19,AiJu21,ArJu21}, tied links in other topological settings \cite{Fl21,Di21}, among others.

The complex generic representation theory of the algebra of braids and ties was studied by Banjo in \cite{Bn13}. For this purpose, she showed that the algebra $\E_n(1)$ is isomorphic to the so--called \emph{small ramified partition algebra} \cite{Mr11}, which is generated by some pairs of set partitions of $[n]$ called \emph{ramified partitions} \cite{MrEl04}. On the other hand, cellularity of the algebra of braids and ties was studied by the second author and Ryom-Hansen in \cite{EsRyH18}. Specifically, they constructed an explicit cellular basis for $\E_n(q)$, whose elements are indexed by a special type of pairs of multitableaux. This construction allowed them to define an explicit isomorphism between the bt--algebra and a direct sum of matrix algebras over certain wreath product algebras indexed by partitions of numbers, whose cellularity was studied by Geetha and Goodman in \cite{GeGo13}. This last result was later generalized by Marin in \cite{Ma18} by defining bt--algebra versions for arbitrary Coxeter groups. See Subsection \ref{107}, Subsection \ref{108} and Subsection \ref{109}.

A first generalization of the tied braid monoid was given by the first author and Juyumaya in \cite{ArJu21} by replacing the respective factors of the semidirect product decomposition of $T\B_n$ with a monoid of set partitions \cite{Re97} and a monoid acting in some Coxeter group. Another generalization was given by the same authors together with Aicardi in \cite{AiArJu23,PreAiArJu22}. Specifically, inspired by the work of Banjo \cite{Bn13}, they introduced a new class of monoids, formed by ramified partitions, that called \emph{ramified monoids}, which are defined for every submonoid of the partition monoid. See Subsection \ref{101}. Particularly, the algebra of braids and ties $\E_n(q)$ can be regarded as a $q$--deformation of the monoid algebra of the ramified monoid $\R(\S_n)$ associated to the symmetric group $\S_n$. See Subsection \ref{106}. Also, due to their combinatorial background, ramified monoids have a rich interpretation in terms of diagrams, and thus any algebra obtained by deformation of their monoid algebras inherit this diagrammatic interpretation. The ramified monoid $\R(\Br_n)$ associated with the \emph{Brauer monoid} $\Br_n$ was studied in \cite{AiArJu23} and the ramified monoid associated with the \emph{symmetric inverse monoid} was studied in \cite{PreAiArJu22}. At the time of writing this paper, the structure of the ramified monoid $\R(\J_n)$ associated to the Jones monoid $\J_n$ is still unknown.

A tied version of the well known \emph{BMW algebra} is the \emph{t--BMW algebra} introduced by Aicardi and Juyumaya in \cite{AiJu18}, which can be regarded as deformation of the monoid algebra of $\R(\Br_n)$. Just as there is no known structure for the monoid $\R(\J_n)$, there also no a tied version of the \emph{Temperley--Lieb algebra} that is obtained as a deformation of the monoid algebra of this monoid. However, there are two tied versions of this algebra that are not directly related to the monoid $\R(\J_n)$. The first one is the \emph{partition Temperley--Lieb algebra} $\PTL_n(q)$ introduced by Juyumaya in \cite{Ju13} in a purely algebraic manner. The second one is the \emph{tied Temperley--Lieb algebra} introduced by Aicardi, Juyumaya and Papi in \cite{PreAiJuPa21}, which is related to a submonoid of $\R(\J_n)$ called the \emph{planar ramified monoid} of $\J_n$ \cite{PreAiArJu22}. A generalization of the partition Temperley--Lieb algebra was studied by Ryom-Hansen in \cite{RyH22}. Specifically, inspired by the construction of \emph{generalized Temperley--Lieb algebras} given in \cite{Ha99}, he introduced certain \emph{permutation modules}, which can be regarded as $\E_n(q)$--submodules of a tensor space module introduced in \cite{RyH11}. To be precise, Ryom--Hansen proved that the annihilator ideal $\I$ of the action of $\E_n(q)$ on this tensor space module is a free $\Z[q,q^{-1}]$--module with a basis given by a dual cellular basis of the one in \cite{EsRyH18}. Thus, he obtained that the quotient $\E_n(q)/\I$ is a simultaneous generalization of the genelized Temperley--Lieb algebra and the partition Temperley--Lieb algebra. In particular, for $N=2$, it is obtained that $\E_n(q)/\I\simeq\PTL_n(q)$, and thus $\PTL_n(q)$ is cellular as well.

In the present paper, we introduce two tied--like algebras that inherit the diagrammatic interpretation from a new class of submonoids of the ramified ones, called \emph{boxed ramified monoids}, which generalize a monoid in \cite[Section 7]{AiArJu23}. We study the representation theory of these algebras and give cellular bases for them. The first one is a subalgebra of the algebra of braids and ties, which we call the \emph{tied--boxed Hecke algebra}. The second one is a new tied version of the Temperley--Lieb algebra, which we call the \emph{tied--boxed Temperley--Lieb algebra}. Additionally, we determine presentations for the singular part of the ramified monoid associated to the symmetric group and for the boxed ramified monoid associated to the Brauer monoid.

The paper is organized as follows.

In Section \ref{098}, we give the main combinatorial and algebraic ingredients that we will use along the paper. Specifically, we recall compositions and multicompositions (Subsection \ref{099}), set partitions and related monoids such as partition monoids and ramified monoids (Subsection \ref{100} and Subsection \ref{101}), and also some cellular algebras such as the Iwahori--Hecke algebra and the Temperley--Lieb algebra (Subsection \ref{102}). Section \ref{105} is devoted to recall the definition of algebra of braids and ties and some important results mainly coming from representation theory (Subsection \ref{107}). We also recall the ramified symmetric monoid, from which the bt--algebra inherits its diagram structure, and study its center (Subsection \ref{106}). In Section \ref{012}, we introduce a new family of monoids that we call \emph{boxed ramified monoids}, which will give diagrammatic structures to the algebras that we will study in the following sections. These monoids are submonoids of the ramified ones, which are defined by imposing a linearity condition on the right components of ramified partitions.

In Section \ref{135}, we introduce the \emph{tied-boxed Hecke algebra} $bH_n(q)$ by generators and relations and study its representation theory (Subsection \ref{110}). Specifically, we establish its connection with the algebra of braids and ties (Proposition \ref{123}), enabling us to compute its dimension (Theorem \ref{130}) and thereby obtain a faithful representation of $bH_n(q)$ in a tensor space module. Subsequently, we construct a complete set of orthogonal idempotents for $bH_n(q)$ (Proposition \ref{propE}), leading to a decomposition of it into a direct sum of two--sided ideals. This allows us to provide an explicit cellular basis for $bH_n(q)$ indexed by linear set partitions and standard Young tableaux of the initial kind (Theorem \ref{131}). Furthermore, we present a diagrammatic representation of $bH_n(q)$ inherited from the boxed ramified monoid of the symmetric group $\BR(\S_n)$ (Subsection \ref{111}). Specifically, we give a presentation for $\BR(\S_n)$ (Theorem \ref{017}), concluding that the tied-boxed Hecke algebra is a $q$--deformation of the monoid algebra of $\BR(\S_n)$. We also prove that the center of $\BR(\S_n)$ coincides with the monoid of compositions of integers (Proposition \ref{089}). Additionally, we delve into the study of the singular part of $\BR(\S_n)$ (Subsection \ref{112}), providing a presentation for it (Theorem \ref{055}).

Finally, in Section \ref{113}, we introduce the \emph{tied--boxed Temperley--Lieb algebra} $bTL_n(q)$ as a quotient of the tied--boxed Hecke algebra by a two--sided ideal generated by \emph{Steinberg elements}. We give a presentation for $bTL_n(q)$ (Proposition \ref{134}) and show that it can be decomposed as a direct sum of two--sided ideals, obtaining an explicit cellular basis for $bTL_n(q)$, which is indexed by multipartitions whose components have Young diagrams with at most two columns and standard multitableaux of the initial kind (Theorem \ref{132}). Then, we explore a strong connection between $bTL_n(q)$ and the partition Temperley--Lieb algebra $\PTL_n(q)$ (Subsection \ref{114}). Specifically, we prove that $\PTL_n(q)$ is isomorphic to a direct sum of matrix algebras over certain wreath product algebras, obtaining a formula for the dimension of $\PTL_n(q)$, which is confirmed in \cite{RyH22} by considering the case $N=2$ of the \emph{generalized partition Temperley--Lieb}. We next observe that the natural embedding of $bH_n(q)$ in $\E_n(q)$ induces an embedding of $bTL_n(q)$ in $\PTL_n(q)$ (Corollary \ref{inclusionbTLn}). Furthermore, we present a diagrammatic realization of $bTL_n(q)$ inherited from the boxed ramified monoid of the Jones monoid (Subsection \ref{115}). Specifically, we recall the presentation for $\BR(\J_n)$ given in \cite[Theorem 56]{AiArJu23} and conclude that the tied--boxed Temperley--Lieb algebra is a $q$--deformation of the monoid algebra of $\BR(\J_n)$. (Theorem \ref{073}). Additionally, we study the boxed ramified monoid of the Brauer monoid (Subsection \ref{116}), providing a presentation for it (Theorem \ref{016}).


\section{Preliminaries}\label{098}

For every pair of integers $m,n$ with $m\leq n$ we will denote by $[m,n]$ the discrete interval $\{m,\ldots,n\}$. Also, if $n$ is positive we set $[n]=[1,n]$ and $[n]_0=[0,n]$.

In what follows $n$ denotes a positive integer, and $\qring$ denotes the ring $\mathbb{C}[q,q^{-1}]$, where $q$ is an indeterminate. Here, an \emph{$\qring$--algebra} means an associative $\qring$--algebra with unit.

\subsection{Compositions and multicompositions}\label{099}

A \emph{composition} of \emph{size} $n$ is a finite sequence $\mu$ formed by $\ell(\mu)$ elements in $[n]_0$, called \emph{parts}, such that their sum is $n$. We denote by $\C_n$ the collection of compositions of $n$ with no null parts. It is known that $|\C_n|=2^{n-1}$ \cite[A000079]{OEIS}. We say that a composition of $n$ is a \emph{partition} if each pair of consecutive parts of it is nonincreasing. The collection of set partitions of $n$ will be denoted by $\Par_n$.

The \emph{Young diagram} or simply the \emph{diagram} of a composition $\mu=(\mu_1,\ldots,\mu_k)$ is the set $[\mu]$ defined as follows:\[[\mu]=\{(i,j)\in\N^2\mid 1\leq j\leq\mu_i\}.\]Elements of $[\mu]$ are called \emph{nodes} and are represented by boxes located at positions $(i,j)$ of an array of size $\ell(\mu)\times n$. See Figure \ref{090}.
\begin{figure}[H]
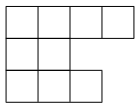
\centering\yng(4,2,3)\caption{Young diagram of $(4,2,3)$.}\label{090}\end{figure}
A \emph{$\mu$--tableau} is a bijection of $[\mu]$ with $[n]$, which is represented by inserting the image of each node, under this bijection, into its respective box. We denote by $\Tab(\mu)$ the set of all $\mu$--tableaux. A tableau is called \emph{row--standard} if its entries increase from left to right. A row--standard tableau is called \emph{standard} if it is defined by a partition and the entries of its columns increase from top to bottom. See Figure \ref{091}. We denote by $\Std(\mu)$ the collection of standard $\mu$--tableaux.
\begin{figure}[H]\centering\yyoung{1378,25,46}\caption{A standard $(4,2,2)$--tableau.}\label{091}\end{figure}
Observe that the symmetric group $\S_n$ acts on the right on the set $\Tab(\mu)$ by permuting the entries in each $\mu$--tableau. For $\mu\in\C_n$, we will denote by $\t^\mu$ the standard $\mu$--tableau in which the numbers in $[n]$ are entered in increasing order from left to right along the rows of $[\mu]$. See Figure \ref{092}.
\begin{figure}[H]\centering\yyoung{1234,56,78}\caption{The standard $\mu$--tableau $\t^\mu$ for $\mu=(4,2,2)$.}\label{092}\end{figure}
For each composition $\mu=(\mu_1,\ldots,\mu_k)\in \C_n$, we denote by $\S_\mu$ the \emph{stabilizer subgroup} of the tableau $\t^\mu$. This subgroup is called the \emph{Young subgroup} of $\mu$, and is isomorphic to a direct product of symmetric groups, that is\[\S_\mu\simeq\S_{\mu_1}\times\cdots\times\S_{\mu_k}.\]For a row--standard $\mu$--tableau $\s$, we denote by $d(\s)$ the unique permutation satisfying $\t^\mu d(\s)=\s$.

Given compositions $\mu,\mu'$ of size $n$, we say that $\mu$ is \emph{finer} than $\mu'$, or that $\mu'$ is \emph{coarser} than $\mu$, denoted by $\mu\triangleleft\mu'$, if $\mu'$ can be obtained by adding consecutive parts of $\mu$. This relations gives to $\C_n$ a structure of lattice with supremum operation $\vee$. The pair $(\C_n,\vee)$ will be called the \emph{monoid of compositions of $n$}.

A \emph{multicomposition} of \emph{size} $n$ is a tuple $\bmu=(\mu^{(1)},\ldots,\mu^{(k)})$ such that $\mu^{(i)}\in\C_{n_i}$ and $n=n_1+\cdots+n_k$. We call $\bmu$ a \emph{multipartition} if each $\mu_i$ is a partition. The collection of multicompositions (resp. multipartitions) of $n$ is denoted by $\MC_n$ (resp. $\MPar_n$). For example, $\bmu=((1,2),(3,2,1),(1,1,1))$ is a multicomposition of $12$. The \emph{Young diagram} or simply the \emph{diagram} of a multicomposition $\bmu=(\mu^{(1)},\ldots,\mu^{(k)})$ is the set $[\bmu]$ defined as follows:\begin{equation}[\bmu]=\{(i,j,c)\in\N^3\mid1\leq j\leq\mu_i^{(c)},\,1\leq c\leq k\}.\end{equation}In other words, the Young diagram of a multicomposition is the tuple formed by the diagrams of its components. See Figure \ref{120}.\begin{figure}[H]\centering$[\bmu]=\left(\,\yng(1,2)\,,\,\yng(3,2,1)\,,\,\yng(1,1,1)\,\right)$\caption{The Young diagram of $\bmu=((1,2),(3,2,1),(1,1,1))$.}\label{120}\end{figure}
For $\bmu\in\MC_n$, a \emph{$\bmu$-multitableau} is a tuple obtained from $[\bmu]$ by inserting the numbers of $[n]$ into its boxes. See Figure \ref{119}. We will write $\shape(\bt)=\bmu$ to indicate that $\bt$ is a $\bmu$-multitableau.
\begin{figure}[H]\centering$\left(\,\yyoung{7,5{10}}\,,\,\yyoung{49{12},16,3}\,,\,\yyoung{2,8,{11}}\,\right)$\caption{A $\bmu$-multitableau with $\bmu$ as in Figure \ref{120}.}\label{119}\end{figure}A multitableau is called \emph{row--standard} if each one of its components is \emph{row--standard}. See Figure \ref{119}. A row--standard multitableau is called \emph{standard} if it is defined by a multipartition and each one of its components is a standard tableau. See Figure \ref{118}. For each multicomposition $\bmu$ and each row--standard $\bmu$--multitableau $\bs$, the elements $\bt^\bmu$ and $d(\bs)$ are defined in the obvious manner as generalizations of $\t^\mu$ and $d(\s)$, respectively.\begin{figure}[H]
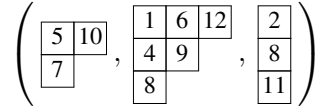
\centering$\left(\,\yyoung{5{10},7}\,,\,\yyoung{16{12},49,8}\,,\,\yyoung{2,8,{11}}\,\right)$\caption{A standard multitableu.}\label{118}\end{figure}

Let $\bmu=(\mu^{(1)},\ldots,\mu^{(k)})$ be a multicomposition of size $n$ such that $\mu^{(j)}\in\C_{n_j}$ for all $j\in[k]$. We denote by $\comp(\bmu)$ the composition given by the size of each component of $\bmu$. That is $\comp(\bmu)=(n_1,\ldots,n_k)$. For example, the composition associate to the multicomposition $\bmu=((1,2),(3,2,1),(1,1,1))$ is $\comp(\bmu)=(3,6,3)$.


\subsection{Set partitions}\label{100}

In what follows $A$ denotes a nonempty set.

A set partition of $A$ is a collection of pairwise disjoint nonempty subsets of it, called \emph{blocks}, such that their union is $A$. The collection of set partitions of $A$ is denoted by $\P(A)$, and $\P([n])$ is denoted by $\P_n$ instead. The cardinality of $\P_n$ is the $n$th \emph{Bell number} $\b_n$ \cite[A000110]{OEIS}. 

Set partitions $\bI=\{I_1,\ldots,I_k\}$ of $[n]$ are denoted by $\bI=(I_{a_1},\ldots,I_{a_k})$, where $\min(I_{a_1})<\cdots<\min(I_{a_k})$. These set partitions are usually represented by \emph{arc diagrams}, which connect, in a transitive manner, the elements of $[n]$ that belong to the same block. See Figure \ref{029}. Also, we set $|\bI|=k$ and $\|\bI\|=(|I_1|,\ldots,|I_k|)$. For $\bI\in\P_n$ and $i,j\in[n]$, we will denote by $i\sim_\bI j$ the fact that $i,j$ belong to the same block of $\bI$.
\begin{figure}[H]
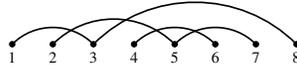
\figftn\caption{Set partition $(\{1,3,8\},\{2,5,7\},\{4,6\})$ of $[8]$.}\label{029}\end{figure}

Given $\bI,\bJ\in\P(A)$, we say that $\bI$ is \emph{finer} than $\bJ$, or that $\bJ$ is \emph{coarser} than $\bI$, denoted by $\bI\preceq\bJ$, if each block of $\bJ$ is a union of blocks of $\bI$.  This relation gives to $\P(A)$ a structure of lattice with supremum operation $\vee$. The pair $(\P_n,\vee)$ is called \emph{the monoid of set partitions} of $[n]$. It was shown in \cite[Theorem 2]{Fi03} that $\P_n$ can be presented by $\frac{n(n-1)}{2}$ \cite[A000217]{OEIS} generators $e_{i,j}$ with $i<j$, subject to the following relations:\begin{equation}\label{047}e_{i,j}^2=e_{i,j},\quad e_{i,j}e_{r,s}=e_{r,s}e_{i,j},\quad e_{i,j}e_{i,k}=e_{i,j}e_{j,k}=e_{i,k}e_{j,k},\end{equation}where each $e_{i,j}$ is the set partition of $[n]$ whose unique block that is not a singleton is $\{i,j\}$. See Figure \ref{005}.
\begin{figure}[H]
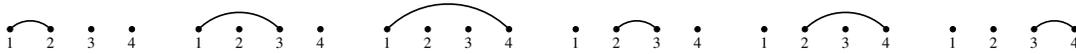
\figtwe\caption{Generators $e_{i,j}$ of $\P_4$.}\label{005}\end{figure}

Given $\bI\in\P(A)$ and a subset $X$ of $A$, we set $\bI\cap X=\{B\cap X\mid B\in\bI\}\backslash\{\emptyset\}$ in $\P(X)$.

Observe that if $X$ is a subset of $A$ and $\bI\in\P(X)$, then $\bI$ can be naturally regarded as a set partition of $A$ by inserting to it the singletons $\{a\}$ with $a\in A\backslash X$. Now, if $\bI\in\P(A)$ and $\bJ\in\P(B)$ for some set $B$, we will simply denote by $\bI\vee\bJ$ the product of $\bI$ with $\bJ$ regarded as elements of $\P(A\cup B)$.

A \emph{ramified partition} of $A$ is a pair $(\bI,\bJ)$ of set partitions of $A$ such that $\bI\preceq\bJ$. Observe that $(\bI,\bJ)$ can be regarded as a set partition of $\bI$, in which two blocks of $\bI$ belong to the same block of $(\bI,\bJ)$ whenever both are contained in the same block of $\bJ$. The collection of ramified partitions of $A$ is denoted by $\RP(A)$, and $\RP([n])$ is denoted by $\RP_n$. We call $\RP(A)$, together with the product $(\bI,\bJ)\vee(\bH,\bK)=(\bI\vee\bH,\bJ\vee\bK)$, the \emph{monoid of ramified partitions} of $A$.

\subsection{The partition and the ramified monoids}\label{101}

Here, we will denote by $\CC_n$ instead of $\P_{2n}$ the collection of set partitions of $[2n]$. Set partitions in $\CC_n$ will be represented by \emph{strand diagrams}, which are obtained from arc diagrams, by locating on top the elements of $[n]$ and at bottom the elements of $[n+1,2n]$ renumbered from $1$ to $n$. See Figure \ref{030}.
\begin{figure}[H]\figfif\caption{Set partition $(\{1,3,8\},\{2,5,7\},\{4,6\})$ of $[8]$.}\label{030}\end{figure}
For $\bI,\bJ\in\CC_n$, the \emph{concatenation} $\bI*\bJ$ of $\bI$ with $\bJ$, is the set partition obtained by identifying the bottom elements of $\bI$ with the top ones of $\bJ$ and so removing them. More specifically, if $X=\{x_1,\ldots,x_n\}$ is a set that contains no elements of $[2n]$, then $\bI*\bJ=(\bI_X\vee\bJ^X)\cap[2n]$, where $\bI_X$ is obtained from $\bI$ by replacing each $n+i$ by $x_i$ and $\bJ^X$ is obtained from $\bJ$ by replacing each $i$ by $x_i$. See Figure \ref{031}.
\begin{figure}[H]\figstn\caption{Set partition $(\{1,3,8\},\{2,5,7\},\{4,6\})$ of $[8]$.}\label{031}\end{figure}
The collection $\CC_n$ together with the \emph{concatenation product} is called the \emph{partition monoid} \cite{Jn94,Mr90,Mr94}. The collection of set partitions of $[2n]$ whose blocks contain exactly two elements forms a submonoid $\Br_n$ of $\CC_n$, called the \emph{Brauer monoid} \cite{Br37}, which is presented by generators $s_1,\ldots,s_{n-1}$ and $t_1,\ldots,t_{n-1}$, subject to the relations \cite[Theorem 3.1]{KuMa06}:\begin{gather}
s_i^2=1;\qquad s_is_js_i=s_js_is_j,\quad|i-j|=1;\qquad s_is_j=s_js_i,\quad|i-j|>1;\label{032}\\
t_i^2=t_i;\qquad t_it_jt_i=t_i,\quad|i-j|=1;\qquad t_it_j=t_jt_i,\quad|i-j|>1;\label{033}\\
s_it_jt_i=s_jt_i,\quad t_it_js_i=t_is_j,\quad|i-j|=1;\qquad t_is_i=s_it_i=t_i;\qquad t_is_j=s_jt_i,\quad|i-j|>1.\label{034}
\end{gather}
The submonoid of $\Br_n$ presented by generators $s_1,\ldots,s_{n-1}$ and relations in \eqref{032} corresponds to the symmetric group $\S_n$ of permutations of $[n]$, and coincides with the group of units of $\CC_n$. The submonoid $\J_n$ of $\Br_n$ presented by generators $t_1,\ldots,t_n$ and relations in \eqref{033} is called the \emph{Jones monoid}. It is well known that $|\S_n|$ is the $n$th factorial number $n!$ \cite[A000142]{OEIS}, and that $|\J_n|$ is the $n$th \emph{Catalan number} $\c_n$ \cite[A000108]{OEIS}.

For every submonoid $M$ of $\CC_n$, we denote by $\R(M)$ the collection of ramified partitions $(\bI,\bJ)$ of $[2n]$ such that $\bI\in M$. Ramified partitions $(\bI,\bJ)$ in $\R(\CC_n)$ are represented by \emph{diagrams of ties}, which are obtained by connecting by \emph{ties} the blocks in the strand diagram of $\bI$ that are contained in the same block of $\bJ$. See Figure \ref{035}. The collection $\R(M)$ together with the product $(\bI,\bJ)*(\bH,\bK)=(\bI*\bJ,\bH*\bK)$ is called the \emph{ramified monoid} of $M$. Observe that $M$ embeds inside $\R(M)$ through the identification $\bI\mapsto(\bI,\bI)$.\begin{figure}[H]\figtwofou\caption{Ramified partition $\left((\{1,3\},\{2\},\{4,6,8\},\{5\},\{7\}),(\{1,3,4,6,8\},\{2,5\},\{7\})\right)$ in $\R(\CC_4)$.}\label{035}\end{figure}

As mentioned in \cite[Proposition 2]{PreAiArJu22}, the monoid $\R(\{1\})$ is isomorphic to $\P_n$ via the identification of $e_{i,j}$ with the ramified partition obtained by connecting by a tie the $i$th strand with the $j$th strand of the identity of $\CC_n$. See Figure \ref{004}. This implies that every ramified monoid contains $\P_n$ as a submonoid. We will use the same notation to denote $e_{i,j}$ as an element of $\R(\{1\})$.
\begin{figure}[H]
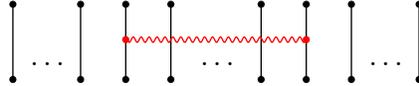
\figthi\caption{Generator $e_{i,j}$.}\label{004}\end{figure}

\subsection{Cellular algebras}\label{102}

In this subsection we recall briefly the concept of a \emph{cellular algebra}, introduced by Graham and Lerher in \cite{GrLe04}. Later, we recall the \emph{Murphy's standard basis} of the Iwahori--Hecke algebra, which results to be a cellular basis. We will use this basis as an inspiration to construct cellular bases for the two algebras that we study in this paper: the \emph{tied--boxed Hecke algebra} and the \emph{tied--boxed Temperley--Lieb algebra}.

Let $R$ be an integral domain, and let $A$ to be an $R$--algebra which is free as an $R$--module. Consider $(\Lambda,
\geq)$ to be a poset in which for each $\lambda\in\Lambda$ there is a finite indexing set $T(\lambda)$ and elements $c_{\s\t}^\lambda\in A$ such that $B:=\{c_{\s\t}^\lambda\mid\lambda\in\Lambda\text{ and }\s,\t\in T(\lambda)\}$ is an $R$--basis of $A$. The pair $(B,\Lambda)$ is said to be a \emph{cellular basis} of $A$ if:
\begin{enumerate}
\renewcommand{\labelenumi}{\textbf{(\roman{enumi})}}
\item The $R$--linear map $*:A\to A$ given by $(\mathop{c_{\s\t}^\lambda})^*=c_{\t\s}^\lambda$ is an algebra
antiautomorphism of $A$.
\item For any $\lambda\in\Lambda,\;\t\in T(\lambda)$ and $a\in A$ there is $r_\V\in R$ such that for all $\s\in T(\lambda)$\[c_{\s\t}^\lambda a\equiv\sum_{\V\in T(\lambda)}r_\V c_{\s\V}^\lambda\mod A^\lambda,\]where the coefficient $r_\V\in R$ do not depend on $\t$, and $A^\lambda$ is the $R$--submodule of $A$ with basis:\[\{c_{\u\V}^\mu\mid\mu\in\Lambda\text{ such that }\mu>\lambda\text{ and }\u,\V\in T(\mu)\}.\]
\end{enumerate}
In this case, $A$ is said to be a \emph{cellular algebra}, and the tuple $(\Lambda,T,B,\ast)$ is called the \emph{cell datum} of it.

\subsubsection{The Iwahori--Hecke algebra}\label{103}

Recall that the \emph{braid group} $\B_n$ on $n$ strands \cite{Ar25} is the one presented by generators $\sigma_1,\ldots,\sigma_{n-1}$ subject to the well known \emph{braid relations}, that is:\begin{equation}\label{024}\sigma_i\sigma_j\sigma_i=\sigma_j\sigma_i\sigma_j,\quad|i-j|=1;\qquad\sigma_i\sigma_j=\sigma_j\sigma_i,\quad|i-j|>1.\end{equation}The \emph{braid monoid} $\B_n^+$ \cite{Ga69} is the submonoid of the braid group defined with the same presentation of $\B_n$.

The \emph{Iwahori--Hecke algebra} $H_n(q)$ is the $\qring$--algebra presented by generators $h_1,\ldots,h_{n-1}$ subject to the braid relations together with a \emph{quadratic relation}:\begin{gather}\label{094}
h_ih_jh_i=h_jh_ih_j,\quad|i-j|=1;\qquad h_ih_j=h_jh_i,\quad|i-j|>1;\qquad h_i^2=1+(q-q^{-1})h_i.\end{gather}Observe that $H_n(q)$ can be concerned as the quotient of the monoid algebra of $\B_n^+$ over $\qring$, by the two--sided ideal defined by the quadratic relation in \eqref{094}. Also, if $q=1$ it coincides with the group algebra of $\S_n$.

Due to the Matsumoto's theorem \cite{Mt64}, for each $w\in\S_n$ and each reduced expression $s_{i_1}\cdots s_{i_k}$ of it, the element $h_w:=h_{i_1}\cdots h_{i_k}$ is well defined in $H_n(q)$. Moreover, it is well known that $H_n(q)$ is a free $\qring$--module with $\qring$--basis $\{h_w\mid w\in\S_n\}$. Hence $\dim(H_n(q))=n!$. Another important $\qring$--basis for $H_n(q)$ is the \emph{Murphy's standard basis} $B_{H_n}$ \cite{Mur95}, which is indexed by pairs of standard tableaux:\begin{equation}\label{Murphybasis}
B_{H_n}=\{m_{\s\t}^\lambda:=h_{d(\s)}^\ast m_\lambda h_{d(\t)}\mid\lambda\in\Par_n\text{ and }\s,\t\in\Std(\lambda)\}.\end{equation}
where $\ast:H_n(q)\to H_n(q)$ is the antiautomorphism given by $h_w^{*}=h_{w^{-1}}$ and $m_\lambda:=\sum_{w\in\S_\lambda} q^{\ell(w)}h_w$.

In \cite{Mat99}, it was studied in detail the representation theory of $H_n(q)$ through its cellularity. Here, the Murphy's standard basis $B_{H_n}$ is used as a cellular basis for the Iwahori--Hecke algebra.

\subsubsection{The Temperley--Lieb algebra}\label{104}

For every integer $n\geq3$, the \emph{Temperley--Lieb algebra} $TL_{n}(q)$ is the $\qring$--algebra presented by generators $h_1,\ldots,h_{n-1}$ satisfying \eqref{094}, subject to the following additional relations:
\begin{align}
h_{i,j}=0;\quad\text{where}\quad h_{i,j}=1+qh_i+qh_j+q^2h_ih_j+q^2h_jh_i+q^3h_ih_jh_{i},\qquad|i-j|=1.
\end{align}
The elements $h_{i,j}$ are called the \emph{Steinberg elements}. Observe that $TL_n(q)$ is the quotient of $H_n(q)$ by the two--sided ideal generated by these elements. It is well know that $\dim(TL_n(q))=\c_n$, the $m$th \emph{Catalan number} \cite[A000108]{OEIS}.

It is worth to mention that there are several definitions of the Temperley--Lieb algebra \cite{TeLi71,Ka90,Jn87}. In our definition, we consider the following changes with respect to the \emph{H\"{a}rterich's presentation} \cite{Ha99}.
\[\begin{array}{c|c}\text{H\"{a}rterich}&\text{Actual}\\\hline q&{q^2}^{\textcolor{white}{1}}\\T_{s_i}&qh_i\end{array}\]

In \cite{Ha99}, H\"{a}rterich established a relation between the diagrammatical cellular basis of $TL_{n}(q)$ and the Murphy's cellular basis of $H_n(q)$ given in \eqref{Murphybasis}. Indeed, from this basis he obtained a new cellular basis $B_{TL_{n}}$ for $TL_{n}(q)$ by considering only the standard tableaux with at most two columns. More precisely,\begin{equation}\label{cellularbasismurphyTL}B_{TL_{n}}=\{m_{\s\t}\mid \lambda\in\Par_n^{\leq2}\text{ and }\s,\t\in\Std(\lambda)\},\end{equation}where $\Par_n^{\leq2}$ is the subset formed by the partitions in $\Par_n$ whose Young diagrams have at most two columns.

\subsubsection{The Young Hecke algebra and the Young Temperley--Lieb algebra}\label{095}

A $\blam$--multitableau $\bt$ is called \emph{of the initial kind} if the $i$th component of $\bt$ contains exactly the numbers of the $i$th component of $\bt^\blam$. For every $\blam\in\MPar_n$, we set $T(\blam)=\{\bt\in\Std(\blam)\mid\bt\text{ is of the initial kind}\}$.

For every $\mu=(\mu_1,\ldots,\mu_k)\in\C_n$, the \emph{Young Hecke algebra} $H_\mu$ (resp. \emph{Young Temperley--Lieb algebra} $TL_\mu$) of $\mu$, is the subalgebra of $H_n(q)$ (resp. $TL_n(q)$) generated by the simple transpositions of $\S_\mu$, that is\[H_\mu\simeq H_{\mu_1}\otimes\cdots\otimes H_{\mu_k},\qquad TL_\mu\simeq TL_{\mu_1}\otimes\cdots\otimes TL_{\mu_k}.\]

It is well known that the tensorial product of cellular algebras is cellular, see for example \cite[Section 3.2]{GeGo13}. Indeed, by using a similar argument as the one used in \cite[Section 3.4]{GeGo13}, we obtain that the elements of the Murphy's basis of $H_\mu$ are indexed by pairs of standard $\blam$--multitableaux of the initial kind, say $\bs=(\s_1,\ldots,\s_k)$ and $\bt=(\t_1,\ldots,\t_k)$, with the condition that $\comp(\blam)=\mu$. More precisely, if $\blam=(\lambda^{(1)},\ldots,\lambda^{(k)})$, then\[B_{H_\mu}=\{m_{\bs\bt}^\blam\mid\bs,\bt\in T(\blam),\,\blam\in\MPar_n\text{ and }\comp(\blam)=\mu\},\quad\text{where}\quad m_{\bs\bt}^\blam=m_{\s_1\t_1}^{\lambda^{(1)}}\cdots m_{\s_k\t_k}^{\lambda^{(k)}}.\]Similarly, the Murphy's basis of $TL_\mu$ is the set\[B_{TL_\mu}=\{m_{\bs\bt}^\blam\mid\bs,\bt\in T(\blam),\,\blam\in\MPar_n^{\leq2}\text{ and }\comp(\blam)=\mu\},\]where $\MPar_n^{\leq2}=\{(\lambda^{(1)},\ldots,\lambda^{(k)})\in\MPar_n\mid\lambda^{(i)}\in\Par_n^{\leq2}\text{ for all }i\in[k]\}$.

\section{The algebra of braids and ties}\label{105}

The \emph{algebra of braids and ties} \cite{AiJu00}, or simply the \emph{bt--algebra}, is the $\qring$--algebra $\E_n(q)$ presented by generators $e_1,\ldots,e_{n-1}$ and $g_1,\ldots,g_{n-1}$, subject to the following relations:\begin{gather}
e_i^2=e_i;\qquad e_ie_j=e_je_i;\label{023}\\
g_ig_jg_i=g_jg_ig_j,\quad|i-j|=1;\qquad g_ig_j=g_jg_i,\quad|i-j|>1;\label{B2}\\
g_ie_i=e_ig_i;\qquad g_ie_j=e_jg_i,\quad|i-j|>1;\label{B4}\\
e_ig_jg_i=g_jg_ie_j,\quad e_ie_jg_j=e_ig_je_i=g_je_ie_j,\quad|i-j|=1;\label{B5}\\
g_i^2=1+(q-q^{-1})e_ig_i.\label{B6}
\end{gather}
Observe that, due to \eqref{B6}, each $g_i$ is invertible with inverse $g_i^{-1}=g_i-(q-q^{-1})e_i$.

\subsection{Representation theory}\label{107}

For $n,r$ positive integers satisfying $n\leq r$ and $V$ to be the free $\qring$--module with basis $\{v_{i}^a\mid i\in[n],a\in[r]\}$, we define the operators $\bE,\bG\in\End_S(V^{\otimes2})$ as follows:
\begin{equation}\label{opE}
(v_i^a\otimes v_j^b)\bE=\begin{cases}0&\text{if }a\neq b,\\v_i^a\otimes v_j^b&\text{if }a=b,\end{cases}\end{equation}and\begin{equation}\label{opG}
(v_i^a\otimes v_j^b)\bG=\begin{cases}v_i^a\otimes v_j^b&\text{if }a\neq b,\\
qv_i^a\otimes v_j^b&\text{if }a=b,\;i=j,\\v_j^a\otimes v_i^b&\text{if }a=b,\;i>j,\\
(q-q^{-1})v_i^a\otimes v_j^b+v_j^b\otimes v_i^a&\text{if }a=b,\;i<j.\end{cases}
\end{equation}
These ones can be extended to operators $\bE_i$ and $\bG_i$ acting on the tensor space $V^{\otimes n}$ by letting $\bE$ and $\bG$ act on the $i$th and $(i+1)$th factors.

\begin{thm}[{\cite[Corollary 4]{RyH11}}]
The map $\rho:\E_n\to\End_S(V^{\otimes n})$ given by $g_i\mapsto\bG_i$ and $e_i\mapsto\bE_i$ defines a faithful representation of $\E_n$.
\end{thm}

Now, we recall some important elements of $\E_n(q)$.

Given $w\in\S_n$ and $s_{i_1}\cdots s_{i_k}$ a reduced expression of it, we set $g_w=g_{i_1}\cdots g_{i_k}$, and put $g_\id=1\in\qring$. From Matsumoto's theorem \cite{Mt64}, the element $g_w$ is independent on the choice of the reduced expression.

For $i,j\in[n-1]$, define $e_{i,j}$, recursively, as follows:
\begin{equation}\label{Es}e_{i,i+1}=e_i\quad\text{and}\quad e_{i,j}=g_i\cdots g_{j-1}e_jg_{j-1}^{-1}\cdots g_i^{-1}\quad\text{whenever}\quad j>i+1.
\end{equation}
For every subset $B=\{i_1<\cdots<i_r\}$ of $[n]$, define
\begin{equation}\label{idemE}E_B=\prod_{j=1}^{r-1}e_{i_j,\,i_{j+1}},\quad\text{where}\quad E_B=1\quad\text{if}\quad|B|=1.\end{equation}
Given a set partition $\bI=\{I_1,\ldots,I_k\}\in\P_n$, we set\begin{equation}\label{idemEA1}
E_\bI=\prod_{i=1}^kE_{I_i}.
\end{equation}
The elements $E_B$ and $E_\bI$ were originally introduced by Ryom--Hansen \cite{RyH11} in a slightly different manner, however, as it was mentioned in \cite{AiJu21}, both definitions are equivalent. We will use the definitions in \eqref{idemE} and \eqref{idemEA1} because they are more useful for our purposes.

With the above notations it was shown in \cite[Theorem 3]{RyH11} that the bt--algebra is a free $\qring$--module with basis $\{E_{\bf I}g_w\mid\bI\in\P_n,\,w\in\S_n\}$. Thus, the dimension of $\E_n(q)$ is $\b_nn!$. Furthermore, the elements $E_\bI$ and $g_w$ satisfy the relation $E_\bI g_w=g_wE_{\bI w}$, where $w$ acts on $\bI$ by permuting the elements of its blocks.

\subsubsection{Ideal decomposition of $\E_n(q)$}\label{108}

Here, we recall some results from \cite[Subsection 6.1]{EsRyH18}.

Let $\mu:\P_n\times\P_n\to\Z$ to be the \emph{M\"{o}ebius function} over the lattice $(\P_n,\preceq)$, defined in \cite[Section 3.2]{EsRyH18}.

Now, for every $\bI\in\P_n$, we define
\begin{equation}\label{idemEI}\bbE_\bI=\sum_{\bJ\succeq\bI}\mu(\bI,\bJ)E_\bJ.
\end{equation}
From this standard construction, due to Solomon \cite{So67}, we obtain a set of orthogonal idempotents $\{\bbE_{\bI}\mid \bI\in\P_n\}$ over $\E_n(q)$. Moreover, the elements $\bbE_\bI$ satisfy the following relations with the basis elements of $\E_n(q)$.\begin{equation}\label{propIdempotentesEI}\bbE_\bI g_w=g_w\bbE_{\bI w}\qquad\text{ and }\qquad\bbE_\bI\bE_\bJ=\begin{cases}\bbE_\bI&\text{if }\bJ\preceq\bI,\\
0&\text{otherwise.}\end{cases}
\end{equation}
In particular, we observe that $\{\bbE_\bI g_w\mid w\in\S_n,\;\bI\in\P_n\}$ is another $\qring$--basis for $\E_n(q)$.

For every $\alpha\in\Par_n$, we say that $\bI=(I_{a_1},\ldots,I_{a_k})\in\P_n$ is \emph{of type $\alpha$}, denoted by $|\bI|=\alpha$, if there is a permutation $w\in\S_n$ such that $(|I_{a_1 w}|,\ldots,|I_{a_k w}|)=\alpha$. Thus, we define the following element of $\E_n(q)$:
\begin{equation}\label{idempotenteEn}\bbE_\alpha\,=\sum_{\bI\in\P_n,\,|\bI|=\alpha}\bbE_\bI.\end{equation}
The set $\{\bbE_\alpha\mid\alpha\in\Par_n\}$ is a family of central orthogonal idempotents of $\E_n(q)$ which is also \emph{complete}, that is:\[\sum_{\alpha\in\Par_n}\bbE_\alpha=1.\]As a consequence we obtain that $\E_n(q)$ can be decomposed as a direct sum of two--side ideals:\begin{equation}\label{093}\E_n(q)=\bigoplus_{\alpha\in\Par_n}\E_n^\alpha(q),\quad\text{where}\quad\E_n^\alpha(q)=\bbE_\alpha\E_n(q).\end{equation}
Moreover, observe that the ideals $\E_n^\alpha(q)$ are also $\qring$--algebras with identities $\bbE_\alpha$ and $\qring$--bases:\[\{\bbE_\bI g_w\mid w\in\S_n,\,|\bI|=\alpha\}\]In particular, $\dim\E_n^\alpha(q)=\b_n(\alpha)n!$, where $\b_n(\alpha)$ is the number of set partition having type $\alpha$ \cite[A036040]{OEIS}.

\subsubsection{Isomorphism theorem for $\E_n(q)$}\label{109}

It was shown in \cite[Theorem 64]{EsRyH18} that $\E_n(q)$ is isomorphic to a direct sum of matrix algebras with coefficients over certain wreath algebras introduced in \cite{GeGo13}. The key point to establish this isomorphism is the construction of a cellular basis $\{m_{\bms\bmt}^\Lambda\}$, which is indexed by a \emph{special type} of pairs of multitableaux.

We first recall the more relevant elements in the construction of cell datum of $\E_n(q)$ given in \cite{EsRyH18}. For more details of this construction the reader should consult \cite[Section 6]{EsRyH18}.

The \textbf{poset}, denoted by $\L_n$, is the set of pairs of multipartitions $(\blam \mid \bmu)$, where $\blam=(\lambda^{(1)},\ldots,\lambda^{(k)})$ is an \emph{increasing multipartition} of $n=n_1+\cdots+n_k$ (see this definition at the beginning of \cite[Section 3.3]{EsRyH18}) with $\lambda^{(j)}\in\C_{n_j}$ for all $j\in[k]$, and $\bmu=(\mu^{(1)},\ldots, \mu^{(j)})$ with $\mu^{(i)}\in\Par_{m_i}$ and each $m_i$ being the multiplicities of equal $\lambda^{(i)}$'s in $\blam$. See Figure \ref{136}.
\begin{figure}[H]
\centering
$\left(\left(\yyng{1},\yyng{1},\yyng{1},\yyng{1,1},\yyng{2,1},\yyng{2,1},\yyng{3}\right)\;\bigg|\;\left(\yyng{2,1},\yyng{1},\yyng{2},\yyng{1}\right)\right).$
\caption{An element of $\L_{14}$.}
\label{136}
\end{figure}

In what follows $\Lambda=(\blam\mid\bmu)\in\L_n$ with $\blam=(\lambda^{(1)},\ldots,\lambda^{(k)})$ and $\bmu=(\mu^{(1)},\ldots,\mu^{(j)})$ as above.

We define the \emph{linear set partition of $[n]$ associated to $\blam$} as follows. First, to each component $\lambda^{(j)}$ of $\blam$, we fix a subset $I_j$ of $[n]$ defined as follows:
\[I_j=\begin{cases}
[1,n_1]&\text{if } j=1,\\
[n_1+\cdots+n_{j-1}+1,n_1+\cdots+n_j]&\text{if } j>1.
\end{cases}\]
Then, the linear set partition associated to $\blam$ is given by $\bI_\blam=\{I_1,\ldots,I_k\}$. For example, the linear set partition associated to the multicomposition $\blam=((2,1),(3,2,1),(1,1,1))$ is $\bI_{\blam}=\{[1,3],[4,9],[10,12]\}$.

Let $\S_\Lambda$ the stabilizer subgroup of the set partition $\bI_{\blam}$. We denote by $\S_\Lambda^m$ the subgroup of $\S_\Lambda$ consisting of the order preserving permutations of the blocks of $\bI_\blam$ that correspond to equal $\lambda^{(i)}$'s, that is\[\S_\Lambda^m\simeq\S_{m_1}\times\cdots\times\S_{m_r}.\]Indeed, $\S_{\Lambda}^{m}$ is generated by the permutations $B_j\in \S_n$, of minimal length, that interchange the blocks $\bI_{j}$ and $\bI_{j+1}$ of $\bI_{\blam}$ whenever $\lambda^{(j)}=\lambda^{(j+1)}$. Moreover, we observe that $\S_\Lambda^m$ does not depend on $\bmu$ but on $\blam$.

For example, if $\blam=((1),(1),(1),(1,1,1),(1,1,1),(2,1),(2,1),(4))$, the following elements generate the subgroup $\S_\Lambda^m\leq\S_{20}$ with $\Lambda=(\blam\mid\underline{\quad}\,)$:\[B_1=s_1,\quad B_2=s_2,\quad B_3=(s_6s_5s_4)(s_7s_6s_5)(s_8s_7s_6)\quad\text{ and }\quad B_4=(s_{12}s_{11}s_{10})(s_{13}s_{12}s_{11})(s_{14}s_{13}s_{12}).\]

Now, suppose that $s_{i_1}\cdots s_{i_r}$ is a reduced expression for $B_i\in \S_{\Lambda}^{m}$.  Due to \cite[Lemma 46]{EsRyH18}, we have an embedding $\iota:\qring\S_{\Lambda}^{m}\hookrightarrow\E_n$ given by $\iota(B_i)=\bbB_i$, where $\bbB_i=\bbE_{\bI_{\blam}}g_{i_1}\cdots g_{i_r}$. Furthermore, if $w=s_{j_1}\cdots s_{j_k}$ is a reduced expression for $w$, we can define $\bbB_w=\bbB_{j_1}\cdots\bbB_{j_k}$, which does not depend on the choice of the reduced expression for $w$.

We say that $\bms$ $=(\bs\mid\bu)$ is a \emph{standard $\Lambda$--tableau} if the following conditions hold:
\begin{enumerate}
\item $\bs$ is a standard $\blam$--multitableau in the usual sense.
\item $\bs=(\s^{(1)},\ldots,\s^{(k)})$ is an \emph{increasing multitableau}. By increasing we here mean that $i<j$ if and only if $\min(\s^{(i)})<\min(\s^{(j)})$ whenever $\shape(\s^{(i)})=\shape(\s^{(j)})$, where $\min(\s)$ is the function that reads off the minimal entry of the tableau $\s$.
\item $\bu$ is a stardard $\bmu$--multitableau of the initial kind.
\end{enumerate}
We will denote by $\Std(\Lambda)$ the set of standard $\Lambda$--tableaux.

The Murphy element corresponding to $\Lambda$ is $m_{\Lambda}=\bbE_{\bI_{\blam}}x_{\blam}b_{\bmu}$, where $x_{\blam}=x_{\lambda^{(1)}}\cdots x_{\lambda^{(k)}}$ and $b_{\bmu}=b_{\mu^{(1)}}\cdots b_{\mu^{(j)}}$, whose factors we explain below.

The element $\bbE_{\bI_{\blam}}$ is the idempotent associated to the linear set partition $\bI_{\blam}$ in (\ref{idemEI}). If $H(\lambda^{(j)})$ denotes the Iwahori--Hecke algebra with linear basis $\{h_i\mid i\in \bI_\blam\}$, then the factor $x_{\lambda^{(i)}}$ can be seen as the image of the Murphy element of the Iwahori--Hecke algebra $H(\lambda^{(j)})$ through of the embedding $H(\lambda^{(j)})\hookrightarrow\E_n$ given by $h_i\mapsto g_i$. Similarly, the factor $b_{\mu^{(i)}}$ can be seen as the image of the $q=1$ specialization of the Murphy element $m_{\mu^{(i)}}\in H(\mu^{(j)})$ through of the embedding $\iota:\qring\S_\Lambda^m\hookrightarrow\E_n$ defined above.

For standard $\Lambda$--tableaux $\bms$ $=(\bs\mid\bu)$ and $\bmt$ $=(\bt\mid\bv)$, the elements of the cellular basis of $\E_n$ are given by:
\begin{equation}\label{elementosbasales}
\text{$m_{\bms\bmt}^{\Lambda}$}=\,g_{d(\bs)}^*\bbE_{\bI_{\blam}}x_{\blam}\bbB_{d(\bu)}^*b_{\bmu}\bbB_{d(\bv)}g_{d(\bt)},
\end{equation}where $\ast$ is the antiautomorphism given by $e_i^*=e_i$ and $g_i^*=g_i$.

\begin{thm}[{\cite[Theorem 57]{EsRyH18}}]
For each $\alpha\in\P_n$, the subalgebras $\E_n^{\alpha}(q)$ are cellular with cellular bases:\[\B_{\alpha}=\{m_{\bms\bmt}^{\Lambda} \mid \Lambda\in \mathcal{L}_n(\alpha),\, \bms,\bmt \in {\rm Std}(\Lambda)\},\]where $\mathcal{L}_n(\alpha)=\{(\blam\mid \bmu) \mid \comp(\blam)=\alpha\}$.

As an immediate consequence, we get that $\E_n(q)$ is cellular with cellular basis:\[\B_{\E_n}=\{m_{\bms\bmt}^{\Lambda} \mid \Lambda\in\L_n,\, \bms,\bmt \in \Std(\Lambda)\}.\]
\end{thm}

Before state the main theorem of this subsection we need to recall some important ingredients.

Similarly to the definition of the subgroup $\S_\Lambda^m$ of $\S_\Lambda$, we define $\S_\Lambda^k$ as the subgroup of $\S_\Lambda$ consisting of the order preserving permutations of the blocks of $\bI_\blam$ that correspond to equal $|\lambda^{(i)}|$'s. Indeed, $\S_\Lambda^k$ is generated by the permutations $B_j\in\S_n$, of minimal length, that interchanges the blocks $\bI_j$ and $\bI_{j+i}$ of $\bI_\blam$ whenever $|\lambda^{(j)}|=|\lambda^{(j+1)}|$. In particular, we observe that $\S_\Lambda^m\leq\S_\Lambda^k\leq\S_\Lambda$ and the correspondence $B_j\mapsto\bbB_j$ defines an embedding of $\qring\S_\Lambda^k$ into $\E_n$.

For a $\Lambda$--multitableau $\bms$ $=(\bs\mid\bu)$, we say that $\bms$ (and $\bs$) is of \emph{wreath type} for $\Lambda$ if and only if there is a $\blam$--multitableau of the initial kind, $\bs_0$, and a permutation $B_w\in \S_\Lambda^k$ such that $\bs=\bs_0B_w$. It was shown in \cite{EsRyH18} that $d(\bs)=d(\bs_1)z_\bs$ for all standard $\Lambda$--tableau $\bms=(\bs\mid \bu)$, where $z_{\bs}$ is a coset representative of $\S_\Lambda$ in $\S_n$ and $\bs_1$ is the standard $\blam$--multitableau of wreath type associated to $\bs$. This implies that the elements $m_{\bms\bmt}$ of the cellular basis can be uniquely written as $m_{\bms\bmt}=g_{d(z_\bs)}^*m_{\bms_1\bmt_1}g_{d(z_{\bt})}$.

Multitableaux of wreath type were motivated by the combinatorial objects used by Geetha and Goodman in \cite{GeGo13} to define a cellular basis for the so--called \emph{wreath Hecke algebra}. Indeed, for each $\alpha=(k_1^{m_1},\ldots,k_r^{m_r})\in\P_n$ such that $k_1>\cdots>k_r$, the \emph{wreath Hecke algebra} is defined as follows:
\begin{equation}
H_\alpha^{wr}(q)=H_{k_1}(q)\wr\S_{m_1}\otimes\cdots\otimes H_{k_r}(q)\wr\S_{m_r}.
\end{equation}
It was shown in \cite{GeGo13} that $H_\alpha^{wr}(q)$ is a cellular algebra indexed by certain pairs of multitableaux, having an obvious correspondence with the cellular basis of the subalgebra $\E_n^\alpha(q)$ and the standard $\Lambda$--multitableaux of wreath type (see \cite[Remark 61]{EsRyH18}).

\begin{thm}[{\cite[Theorem 64]{EsRyH18}}] \label{isomorfismoconmatrices}
For each $\alpha\in \P_n$, the $\qring$--linear map\[\begin{array}{ccc}\Psi_\alpha:\E_n^\alpha(q)&\to&\Mat_{b_n(\alpha)}\big(\H^{wr}_{\alpha}(q)\big)\\[3mm]m_{\bms\bmt}&\mapsto&m_{\es_1\et_1}M_{\bs\bt}\end{array}\]is an isomorphism of $\qring$-algebras, where $M_{\bs\bt}$ denotes the elementary matrix of $\Mat_{b_n(\alpha)}\big(\H^{wr}_{\alpha}(q)\big)$, which is equal to $1$ at the intersection of the row and column indexed by $ z_{\bs}$ and $ z_{\bt}$, and $0$ otherwise.

As an immediate consequence we obtain the following decomposition for the bt--algebra:\[\E_n(q)\simeq\displaystyle\bigoplus_{\alpha\in \P_n}\Mat_{b_n(\alpha)}\big(\mathcal{H}^{wr}_{\alpha}(q)\big).\]
\end{thm}

\subsection{The ramified symmetric monoid}\label{106}

Here, we recall the ramified monoid of the symmetric group and study its center. It is important to mention that, when $q=1$, the algebra of braids and ties will be the monoid algebra generated by this monoid.

The ramified monoid of $\S_n$ is the semidirect product $\R(\S_n)=\P_n\rtimes\S_n$ \cite[Theorem 19]{AiArJu23}, presented by generators $s_1,\ldots,s_{n-1}$ satisfying \eqref{032}, and $e_{i,j}$ with $i<j$ satisfying \eqref{047}, both subject to the relations:\begin{equation}\label{048}s_i\,e_{j,k}=e_{s_i(j),\,s_i(k)}\,s_i.\end{equation}However, by defining $e_i=e_{i,i+1}$ for all $i\in[n-1]$, and due to \eqref{048}, we obtain the following. See Figure \ref{046}.\begin{equation}\label{045}e_{i,j}=s_i\cdots s_{j-2}e_{j-1}s_{j-2}\cdots s_i=s_{j-1}\cdots s_{i+1}e_is_{i+1}\cdots s_{j-1}.\end{equation}Thus, the ramified monoid $\R(\S_n)$ can be presented by the generators $s_1,\ldots,s_{n-1}$ satisfying \eqref{032}, and the generators $e_1,\ldots,e_{n-1}$ satisfying \eqref{023}, subject to the following relations:\begin{gather}
e_is_js_i=s_js_ie_j,\quad e_ie_js_i=e_js_ie_j=s_ie_je_i,\quad|i-j|=1;\qquad s_ie_j=e_js_i,\quad|i-j|\neq1.\label{049}
\end{gather}
\begin{figure}[H]
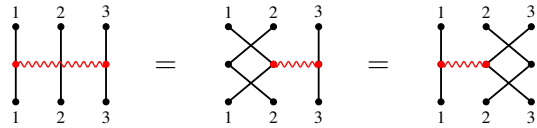
\figntn\caption{Generator $e_{1,3}$.}\label{046}\end{figure}

As it was mentioned above, for $q=1$, the algebra $\E_n(1)$ coincides with the monoid algebra $\qring[\R(\S_n)]$. This is the reason we use the same symbols to denote both the abstract generators of $\E_n(q)$ and the set partitions $e_{i,i+1}$ that generate $\R(\S_n)$, both satisfying \eqref{023}.

\begin{pro}\label{088}
We have $Z(\R(\S_n))=\langle e_1\cdots e_{n-1}\rangle$ for all $n\geq3$.
\end{pro}
\begin{proof}
Let $z=se\in Z(\R(\S_n))$ with $s\in\S_n$ and $e\in\P_n$. Then, for every $h=te'\in\R(\S_n)$ with $t\in\S_n$ and $e'\in\P_n$, we have $(st)t(e)e'=sete'=gh=hg=te'se=(ts)s(e')e$. This implies that $s\in Z(\S_n)=\{1\}$. Thus $g\in\P_n$. Since $g$ must commute, in particular, with every permutation in $\S_n$, then \eqref{048} implies that either $g=1$ or $g=e_1\cdots e_{n-1}$. Therefore $Z(\R(\S_n))=\{1,e_1\cdots e_{n-1}\}=\langle e_1\cdots e_{n-1}\rangle$.
\end{proof}
Observe that $Z(\R(\S_n))$ with $n\geq3$ is a monogenic free idempotent commutative monoid.

The \emph{tied braid monoid} $T\B_n$ \cite{AiJu16} is the one presented by generators $\sigma_1,\ldots,\sigma_{n-1}$ satisfying the braid relations in \eqref{024}, generators $\sigma_1^{-1},\ldots,\sigma_{n-1}^{-1}$, and generators $e_1,\ldots,e_{n-1}$ satisfying \eqref{023}, subject to the relations:\begin{gather}
\sigma_i\sigma_i^{-1}=1=\sigma_i^{-1}\sigma_i;\qquad\sigma_ie_j=e_j\sigma_i,\quad|i-j|\neq1;\label{040}\\
\sigma_i\sigma_je_i=e_j\sigma_i\sigma_j,\quad
\sigma_i\sigma_j^{-1}e_i=e_j\sigma_i\sigma_j^{-1},\quad
\sigma_ie_je_i=e_j\sigma_ie_j=e_ie_j\sigma_i,\quad|i-j|=1.\label{041}
\end{gather}Observe that $\E_n(q)$ is isomorphic to the quotient of the monoid algebra of $T\B_n$ by the relation \eqref{B6}. It was shown in \cite[Theorem 3]{AiJu21} that the tied braid monoid can be decomposed as $T\B_n=\P_n\rtimes\B_n$.

\section{Boxed ramified monoids}\label{012}

Here, we introduce \emph{boxed ramified monoids}, which are submonoids of the ramified monoids defined in Subsection \ref{101}. These ones are obtained by imposing linearity on the right set partitions and will serve as the bases for the algebras we will study later, providing them with a rich diagrammatic interpretation. As we will observe shortly, none of the nontrivial elements in boxed ramified monoids will possess inverses.

A set partition $\bI$ of $[n]$ is called \emph{linear} or \emph{convex} if each of its blocks is an interval. See Figure \ref{000}. As shown in \cite[Lemma 3, Lemma 4]{AiArJu23}, the collection $\LP_n$ of linear set partitions of $[n]$ is the submonoid of $\P_n$ with $2^{n-1}$ elements, which is presented by generators $e_1,\ldots,e_{n-1}$, subject to the relations in \eqref{023}.
\begin{figure}[H]
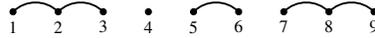
\figzer\caption{The convex set partition $([1,3],[4,4],[5,6],[7,9])=e_1e_2e_5e_7e_8$ of $[9]$.}\label{000}\end{figure}
Observe that the mapping $e_i\mapsto\|e_i\|$ defines an isomorphism of linear set partitions with compositions. Thus, the monoid $\C_n$ is generated by the compositions of length $n-1$ whose parts are all $1$ except one of them which is $2$. For example, the generators of $\C_5$ are $(2,1,1,1),(1,2,1,1),(1,1,2,1),(1,1,1,2)$.

A set partition $\bI\in\CC_n$ is called \emph{boxed} if $\bI\succeq1$ and $\bI\cap[n]$ is linear \cite[Definition 53]{AiArJu23}. See Figure \ref{001}. The collection $b\CC_n$ of boxed set partitions of $[2n]$ forms a submonoid of the well known \emph{monoid of uniform blocks permutations} \cite{Fi03,Ko00,Ko01}, which is presented by generators $b_1,\ldots,b_{n-1}$ subject to the relations $b_i^2=b_i$ and $b_ib_j=b_jb_i$ for all $i,j\in[n-1]$, where each $b_i$ is the set partition of $[2n]$ obtained by joining the $i$th block with the $(i+1)$th block of the identity of $\CC_n$. See Figure \ref{002} and Figure \ref{003}. Thus, the mapping $e_i\mapsto b_i$ defines an isomorphism of compositions with boxed set partitions. Therefore $\LP_n\simeq\C_n\simeq b\CC_n$.
\begin{figure}[H]
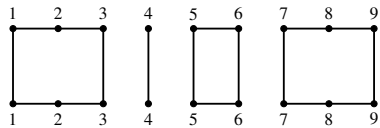
\figone\caption{The boxed set partition $b_1b_2b_5b_7b_8$ of $[18]$.}\label{001}\end{figure}
\begin{figure}[H]\figtwo\caption{Box generator $b_i$ en $\CC_n$.}\label{002}\end{figure}
\begin{figure}[H]
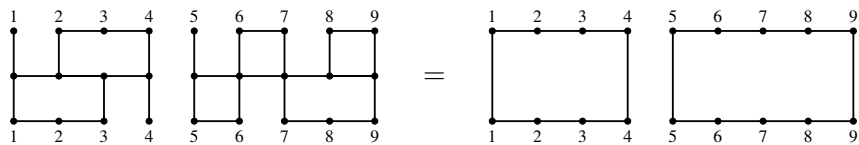
\figthr\caption{Product of boxed set partitions: $b_2b_3b_6b_8\cdot b_1b_2b_5b_7b_8=b_1b_2b_3b_5b_6b_7b_8$.}\label{003}\end{figure}

Since $\P_n=\R(\{1\})$, the submonoid $\C_n$ can be regarded as a submonoid of $\R(\CC_n)$ as well. For example, the element in Figure \ref{000} is represented in $\R(\CC_n)$ as in Figure \ref{001}.\begin{figure}[H]
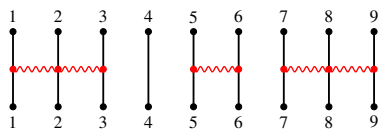
\figfou\caption{The element $e_1e_2e_5e_7e_8$ in $\C_9$.}\label{006}\end{figure}

For every nonnegative integers $r\leq s$, the \emph{$(r,s)$--shift} of a set partition $\bI\in\CC_n$ is the collection\[\bI[r,s]=\{B[r,s]\mid B\in\bI\},\quad\text{where}\quad B[r,s]=\{b+r\mid b\in B,\,b\leq n\}\cup\{b+s\mid b\in B,\,b>n\}.\]Note that coarsening is invariant under shifts, that is, $\bI\preceq\bJ$ if and only if $\bI[r,s]\preceq\bJ[r,s]$.

Given $(\bI,\bJ)\in\CC_m\times\CC_n$, the \emph{over product} of $\bI$ with $\bJ$ is the collection $\bI/\bJ=\bI[0,n]\cup\bJ[m,2m]$. Observe that $\bI/\bJ$ is a set partition of $[2(m+n)]$. See Figure \ref{009}.
\begin{figure}[H]
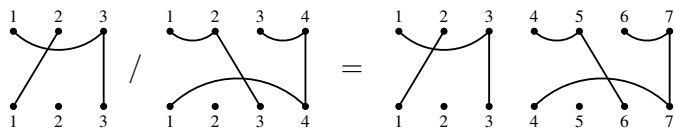
\figsev\caption{Over product of two set partitions.}\label{009}\end{figure}

We say that $\bI\in\CC_n$ is \emph{$r$--separable} for some $r\in[n-1]$ if there is $(\bJ,\bK)\in\CC_r\times\CC_{n-r}$ such that $\bI=\bJ/\bK$. Given $\mu=(\mu_1,\ldots,\mu_k)\in\C_n$, we say that $\bI$ is a \emph{$\mu$--partition} if it is $(\mu_1+\cdots+\mu_i)$--separable for all $i\in[k-1]$. See Figure \ref{010}. Note that if $\bI$ is a $\mu$--partition, then it is a $\mu'$--partition for all composition $\mu'\succeq\mu$. Observe that the over product $/$ defines a semigroup structure on $\CC=\bigsqcup_{m\geq1}\CC_m$. Thus, if $\bI$ is a $\mu$--partition with $\mu=(\mu_1,\ldots,\mu_k)$, then, for each $i\in[k]$, there are $\tilde{\bI}_i\in\CC_{m_i}$ with $m_i=\mu_1+\cdots+\mu_i$ such that $\bI=\tilde{\bI}_1/\cdots/\tilde{\bI}_k$. This decomposition is unique when $k$ is maximal and, in this case, it is called the \emph{boxed decomposition} of $\bI$.
\begin{figure}[H]
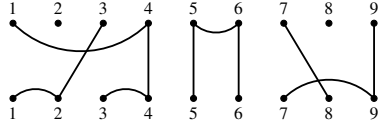
\figeig\caption{A $(4,2,3)$--partition of $[18]$.}\label{010}\end{figure}

For every $\mu\in\C_n$, there is a unique boxed set partition that is a $\mu$--partition and is not a $\mu'$--partition for all $\mu'\prec\mu$. This one is called the \emph{$u$--boxed partition} and is denoted by $b_\mu$. For example, the $(3,1,2,3)$--boxed partition of $[18]$ is the one in Figure \ref{001}. Conversely, every boxed set partition $\bI=(I_1,\ldots,I_k)$ is the $u$--boxed partition defined by $\mu=(\frac{|I_1|}{2},\ldots,\frac{|I_k|}{2})$. This defines the isomorphism of boxed set partitions with $\C_n$.

\begin{pro}\label{013}
For every composition $\mu=(\mu_1,\ldots,\mu_k)\in\C_n$, there are $\b_{2\mu_1}\cdots\b_{2\mu_k}$ $\mu$--partitions.
\end{pro}
\begin{proof}
It is a consequence of the fact that $\bI$ is a $\mu$--partition if and only if $\bI\preceq b_\mu$.
\end{proof}
See Figure \ref{011} for an example.
\begin{figure}[H]
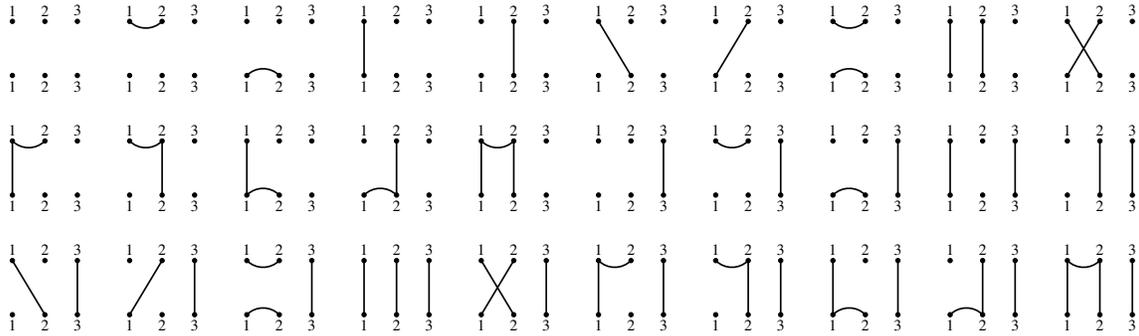
\fignin\caption{The $(2,1)$--partitions of $[6]$.}\label{011}\end{figure}

The \emph{boxed ramified monoid} of a submonoid $M$ of $\CC_n$ is the submonoid $\BR(M)$ of $\R(M)$ formed by the ramified partitions $(\bI,\bJ)$ such that $\bI\in M$ and $\bJ$ is boxed \cite[Cf. Definition 55]{AiArJu23}. See Figure \ref{008}.
\begin{figure}[H]
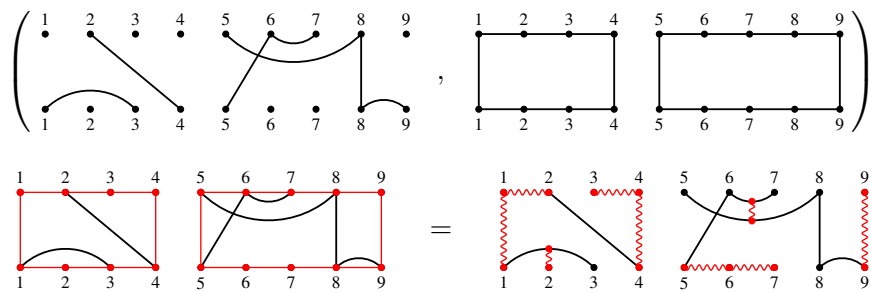
\figsix\caption{A ramified partition in $\BR(\CC_n)$.}\label{008}\end{figure}

\begin{rem}[Embedding]\label{043}
Observe that the map $\bI\mapsto(\bI,\bI)$ is an embedding of $M$ into $\BR(M)$ only if $M$ is formed by boxed set partitions. In particular, $b\CC_n$ embeds into $\BR(b\CC_n)$ via this identification.

On the other hand, as $b_i^2=b_i$ for all $i\in[n-1]$, then, for each $g,h\in\BR(M)$, we have $gh=1$ if and only if $g=h=1$. Therefore $s\BR(M):=\BR(M)\backslash\{1\}$ is a semigroup. Thus $\BR(M)\cap\S_n=\{1\}$, which implies that $s\BR(M)$ is a subsemigroup of the \emph{singular part} $s\R(M):=\R(M)\backslash\S_n$ of $\R(M)$.

Since $e_i=(1,b_i)$ for all $i\in[n-1]$, we have $\BR(M)\cap\P_n=\C_n$ for all submonoid $M$ of $\CC_n$. Now, consider the boxed set partition $b=b_1\cdots b_{n-1}$. Observe that $(1,b)=e_1\cdots e_{n-1}=:e$. Clearly, the map $\bI\mapsto(\bI,\bI)e=(\bI,b)$ defines a semigroup monomorphism, called \emph{boxed morphism}, from $M$ to $\BR(M)$. Moreover, the image of $M$ under the boxed morphism is a monoid with identity $e$. Indeed, since $b^2=b$, we have $(\bI,b)e=(\bI,b)(1,b)=(\bI,b)=(1,b)(\bI,b)=e(\bI,b)$ for all $\bI\in M$.

Observe that $\C_n\subseteq Z(\BR(M))$, indeed since $b\CC_n$ is commutative, then $(\bI,\bJ)e_i=(\bI,\bJ b_i)=(\bI,b_i\bJ)=e_i(\bI,\bJ)$ for all $(\bI,\bJ)\in\BR(M)$ and for all $i\in[n-1]$.
\end{rem}

\begin{rem}\label{021}
Proposition \ref{013} implies that for every boxed set partition $\bJ=(J_1,\ldots,J_k)$, there are $\b_{|J_1|}\cdots\b_{|J_k|}$ ramified partitions $(\bI,\bJ)$. Indeed, $\bI\preceq\bJ$ if and only if $\bI$ is a $(\frac{|J_1|}{2},\ldots,\frac{|J_k|}{2})$--partition. See Figure \ref{014}. Thus,\[|\BR(\CC_n)|=\sum_{(J_1,\ldots,J_k)}\b_{|J_1|}\cdots\b_{|J_k|}=\sum_{(\mu_1,\ldots,\mu_k)}\b_{2\mu_1}\cdots\b_{2\mu_k}=(2,19,271,5373,142923,4945661,\ldots),\]where $(J_1,\ldots,J_k)$ is a boxed set partition of $[2n]$, and $(\mu_1,\ldots,\mu_k)$ is a composition of $n$.
\end{rem}
\begin{figure}[H]
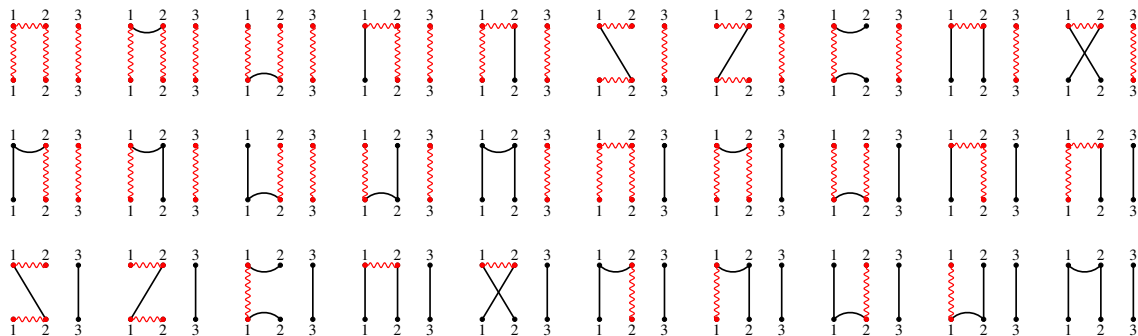
\figten\caption{The ramified partitions $(I,b_1)$ of $[6]$.}\label{014}\end{figure}

As coarsening is invariant under shifts, then $\bI\preceq\bJ$ and $\bH\preceq\bK$ implies $\bI/\bH\preceq\bJ/\bK$, hence the \emph{over product} given by $(\bI,\bJ)/(\bH,\bK)=(\bI/\bH,\bJ/\bK)$ defines a semigroup structure on $\R(\CC)=\bigsqcup_{m\geq1}\R(\CC_m)$.

\section{The tied--boxed Hecke algebra}\label{135}

In this section we introduce the \emph{tied--boxed Hecke algebra} by means generators and relations. Then, we study its representation theory and give a cellular basis for it (Subsection \ref{110}). Further, we show a diagrammatic realization through the boxed ramified monoid of the symmetric group (Subsection \ref{111}). Additionally, we study the singular part of the ramified symmetric monoid (Subsection \ref{112}).

For every positive integer $n$, the \emph{tied--boxed Hecke algebra} $bH_n(q)$ is the $\qring$--algebra presented by generators $e_1,\ldots,e_{n-1}$ satisfying \eqref{023}, and generators $z_1,\ldots,z_{n-1}$, subject to the following relations:
\begin{gather}
z_iz_jz_i=z_jz_iz_j,\quad|i-j|=1;\qquad z_iz_j=z_jz_i,\quad|i-j|>1;\label{h2}\\
e_iz_i=z_i;\qquad e_iz_j=z_je_i;\label{h3}\\
z_i^2=e_i+(q-q^{-1})z_i.\label{h4}
\end{gather}

\subsection{Representation theory}\label{110}

Here we study the representation theory of the tied--boxed Hecke algebra. In particular, we get the dimension of $bH_n(q)$ (Theorem \ref{130}) and show that it can be embedded into the algebra of braids and ties (Corollary \ref{126}). Furthermore, we give a cell datum for $bH_n(q)$ and show that it is cellular (Theorem \ref{131}).

\begin{pro}\label{123}
There is a homomorphism of $\qring$--algebras $\iota_1:bH_n(q)\to\E_n(q)$ satisfying $\iota_1(e_i)=e_i$ and $\iota_1(z_i)=e_ig_i$.
\end{pro}
\begin{proof}
We need to show that $\iota_1(e_i)$ and $\iota_1(z_i)$ satisfy the defining relations of $bH_n(q)$. Relations in \eqref{023} are directly verified from the definition. The second relation in \eqref{h2} and the first one in \eqref{h3} are consequences of \eqref{023} and \eqref{B2}. To prove the first relation in \eqref{h2} and the second one in \eqref{h3} we will use many times \eqref{B5} and the commutativity relations:
\begin{equation}\label{demohomo}z_iz_jz_i=e_ig_ie_jg_je_ig_i=g_i(e_ie_jg_j)e_ig_i=g_ig_je_ie_je_ig_i=g_ig_je_i(e_je_ig_i)=g_ig_je_ig_ie_je_i=g_ig_jg_ie_ie_je_i.\end{equation}The first relation in \eqref{h2} is concluded by using the braid relation in \eqref{B2} and the inverse process of \eqref{demohomo}. The second relation in \eqref{h3} is a direct computation, as we show next\[e_iz_j=e_ie_jg_j=g_je_ie_j=g_je_je_i=e_jg_je_i=z_je_i.\]Finally, the quadratic relation in \eqref{h4} is proved as follows\[z_i^2=(e_ig_i)^2=e_ig_i^2=e_i(1+(q-q^{-1})e_ig_i)=e_i+(q-q^{-1})e_ig_i=e_i+(q-q^{-1})z_i.\]
\end{proof}

To study the algebra $bH_n(q)$, we will use an adaptation of the element $E_\bI$ in \eqref{idemEA1}. This modification makes sense in $bH_n(q)$ only if $\bI$ is a linear set partition. To be more precise, for $\bI=\{I_1,\ldots,I_r\}\in\LP_n$, set
\begin{equation}\label{idemEA2}
\rE_\bI=\prod_{i=1}^rE_{I_i}\,=\prod_{i\,\sim_\bI\,(i+1)}e_i.
\end{equation}

\begin{lem}\label{042}
The algebra $bH_n(q)$ is generated, as an $\qring$--module, by $\G_n:=\{\rE_\bI z_w\mid\bI\in\LP_n,\,w\in\S_{\|\bI\|}\}$.
\end{lem}
\begin{proof} From \eqref{h3} we get that each word $x$ in the generators $e_i$ and $h_i$ can be written as $e_{i_1}\cdots e_{i_r}z_{j_1}\cdots z_{j_s}$. Moreover, from \eqref{h3} and \eqref{idemEA2} we can choose the maximal $\bI\in\LP_n$ such that $\rE_\bI x=x$, that is, if $\bJ$ is another linear set partition satisfying this condition, then $\bJ\preceq\bI$. Thus, $x$ can be rewritten as $x=\rE_\bI\bzz=\rE_\bI z_{j_1}\cdots z_{j_r}$, where $j_i\sim_\bI(j_i+1)$ for all $i\in[r]$. Observe that, under multiplication by the respectively $e_i$'s, the generators $h_i$ satisfy the defining relations of the Hecke algebra. In fact, we have $e_iz_i^2=e_i(e_i+(q-q^{-1})z_i)=e_i(1+(q-q^{-1})z_i)$. This implies that $x$ is a linear combination of $\rE_\bI z_w$ with $w\in\S_{\|\bI\|}$ a reduced expression.
\end{proof}

\begin{lem}\label{mapped}
Let $\bI\in\LP_n$, and let $w\in S_{\|\bI\|}$. Then $\rE_\bI z_w$ is mapped to $E_\bI g_w$ under the homomorphism $\iota_1$.
\end{lem}
\begin{proof}
Let $s_{i_1}\cdots s_{i_r}$ be a reduced expression of $w$. Then\begin{equation}\label{Eh}\iota_1(\rE_\bI z_w)=E_\bI\iota_1(z_{i_1}\cdots z_{i_r})=E_\bI e_{i_1}g_{i_1}\cdots e_{i_r}g_{i_r}.\end{equation}Now, observe that $E_\bI e_{i_j}g_{i_j}=E_\bI g_{i_j}=g_{i_j}E_{\bI s_j}=g_{i_j}E_\bI$ because $i_j\sim_\bI(i_j+1)$ for all $j\in[r]$. By using this observation on the last equality of \eqref{Eh}, we obtain\[E_\bI e_{i_1}g_{i_1}\cdots e_{i_r}g_{i_r}=g_{i_1}\cdots g_{i_r}E_\bI=g_w E_\bI=E_{\bI w} g_w.\]The last expression is equal to $E_\bI g_w$ because $w$ belong to the stabilizer of $\bI$.  This concludes the proof.
\end{proof}

\begin{thm}\label{130}
The algebra $bH_n(q)$ is a free $\qring$--module with basis $\G_n$. In particular, we get\[\dim bH_n(q)=\sum_{(\mu_1,\ldots,\mu_k)\,\in\,\C_n}\mu_1!\cdots\mu_k!.\]
\end{thm}
\begin{proof}
By Lemma \ref{042}, it is enough to prove that $\G_n$ is a linearly independent subset of $bH_n(q)$, which is already satisfied because, due to Lemma \ref{mapped}, the subset $\iota_1(\G_n)$ is linearly independent in the algebra $\E_n(q)$.
\end{proof}

\begin{crl}\label{126}
The homomorphism $\iota_1:bH_n(q)\to\E_n(q)$ is an embedding of $\qring$--algebras.
\end{crl}

\begin{crl}\label{rep}
There is a faithful representation $\varphi$ of $bH_n(q)$ in $V^{\otimes n}$ given by $e_i\mapsto\bE_i$ and $z_i\mapsto \bE_i\circ\bG_i$.
\end{crl}
\begin{proof}
This representation is immediately given by the composition $\rho\circ\iota_1$. The faithfulness of $\varphi$ follows from injectivity of $\rho$ and $\iota_1$.
\end{proof}

\begin{rem}\label{127}
It is known that the algebra of braids and ties $\E_n(q)$ can be regarded as a subalgebra of the \emph{$d$--modular Yokonuma--Hecke algebra $Y_{n,d}(q)$} \cite{Yo67,Ju98,JuKa01,Ju04}. Indeed, $\E_n(q)$ is the subalgebra of $Y_{n,d}(q)$ generated by the elements $g_1,\ldots,g_{n-1}$ together with the idempotents $e_1,\ldots,e_{n-1}$ defined as follows:\begin{equation}\label{128}e_i=\frac{1}{d}\sum_{s=0}^{d-1}t_i^st_{i+1}^{-s},\end{equation}where $Y_{n,d}(q)$ is the algebra presented by generators $g_1,\ldots,g_{n-1}$ satisfying \eqref{B2}, and generators $t_1,\ldots,t_n$ subject to the following relations \cite{Ju98,JuKa01,Ju04,ChPo14}:\begin{gather}t_i^{\,d}=1;\qquad t_it_j=t_jt_i;\qquad g_it_j=t_{s_i(j)}g_i;\qquad g_i^2=1+(q-q^{-1})e_ig_i.\end{gather}Thus, due to Corollary \ref{126}, the tied--boxed Hecke algebra $bH_n(q)$ is the subalgebra of $Y_{n,d}(q)$ generated by the idempotents in \eqref{128}, together with the elements $z_1,\ldots,z_{n-1}$ defined as follows:\[z_i=\frac{1}{d}\sum_{s=0}^{d-1}t_i^st_{i+1}^{-s}g_i.\]
\end{rem}

In what follows we will show that the tied--boxed Hecke algebra $bH_n(q)$ is cellular in the sense of Graham--Lerher \cite{GrLe96}. Similarly to the $bt$--algebra \cite{EsRyH18}, we first construct an ideal decomposition of $bH_n(q)$.

First, we define a \emph{Mo\"{e}bius function} on the lattice $(\LP_n,\preceq)$ as the map $\mu:\LP_n\times\LP_n\to\Z$ such that\[\mu(\bI,\bJ)=\begin{cases}(-1)^{|\bI|-|\bJ|}&\text{if }\bI\preceq\bJ,\\0&\text{otherwise.}\end{cases}\]This Mo\"{e}bius function will allows us to construct a set of orthogonal idempotents on $bH_n(q)$, which are a special case of the general construction given in \cite{So67,Gr73}.

For each $\bI\in\LP_n$, we define\[\bbE_\bI=\sum_{\bJ\supseteq\bI}\mu(\bI,\bJ)E_\bJ.\]

For example, for $n=3$, we have the following idempotents:\[\begin{array}{lcl}
\bbE_{\{\{1\},\{2\},\{3\}\}}&=&\bE_{\{\{1\},\{2\},\{3\}\}}-\bE_{\{\{1,2\},\{3\}\}}-\bE_{\{\{1\},\{2,3\}\}}+\bE_{\{\{1,2,3\}\}}=1-e_1-e_2+e_1e_2,\\
\bbE_{\{\{1,2\},\{3\}\}}&=&\bE_{\{\{1,2\},\{3\}\}}-\bE_{\{\{1,2,3\}\}}=e_1-e_1e_2,\\
\bbE_{\{\{1\},\{2,3\}\}}&=&\bE_{\{\{1\},\{2,3\}\}}-\bE_{\{\{1,2,3\}\}}=e_2-e_1e_2,\\
\bbE_{\{\{1,2,3\}\}}&=&\bE_{\{\{1,2,3\}\}}=e_1e_2.
\end{array}\]

\begin{pro}\label{propE}
The following properties hold.
\begin{enumerate}
\item $\I_\bbE:=\{\bbE_\bI\mid\bI\in\LP_n\}$ is a complete set of central orthogonal idempotent elements of $bH_n$.
\item For all $\bI\in\LP_n$ we have $\bbE_\bI\bE_\bJ=\begin{cases}\bbE_\bI&\text{if }\bJ\preceq\bI,\\
0&\text{otherwise.}\end{cases}$\label{129}
\end{enumerate}
\end{pro}
\begin{proof}
The proof is similar to the one of \cite[Proposition 39]{EsRyH18}.
\end{proof}

As an immediate consequence of Proposition \ref{propE} we obtain that $bH_n(q)$ can be decomposed as a direct sum of two--side ideals.
\begin{equation}\label{096}
bH_n(q)\simeq\displaystyle\bigoplus_{\mu\in\C_n}bH_n(q)^\mu,\end{equation}
where $bH_n(q)^\mu=\bbE_\bI bH_n(q)$ and $\bI$ is the unique linear set partition satisfying $\|\bI\|=\mu$. Moreover, every two--side ideal $bH_n(q)^\mu$ is also a subalgebra of $bH_n(q)$.

\begin{crl}\label{isoconHecke}
Let $\bI\in\LP_n$ and suppose that $\|\bI\|=\mu=(\mu_1,\ldots,\mu_k)\in\C_n$. Then, the $\qring$--linear map\[\begin{array}{ccccc}bH_n(q)^\mu&\to&H_\mu&\simeq&H_{\mu_1}\otimes\cdots\otimes H_{\mu_k}\\[3mm]\bbE_\bI z_w&\mapsto&z_w=z_{w_1}\cdots z_{w_k}&\mapsto&z_{w_1}\otimes\cdots\otimes z_{w_k}\end{array}\]is an isomorphism of $\qring$--algebras. In particular $\dim bH_n(q)^\mu=\mu_1!\cdots\mu_k!$.
\end{crl}
\begin{proof}
It is enough to note that $bH_n(q)^\mu$ is a free $\qring$--module with $\qring$--basis $\{\bbE_\bI z_w\mid w\in\S_\mu\}$, which is obtained by using the linear basis $\{\bE_{\bI} z_w\}$ of $bH_n(q)$ together with Proposition \ref{propE}(\ref{129}).
\end{proof}

In virtue of Corollary \ref{isoconHecke}, the decomposition in \eqref{096} can be rewritten as follows:
\begin{equation}\label{decompositionbH}
bH_n(q)\,\,\,\simeq\displaystyle\bigoplus_{(\mu_1,\ldots,\mu_k)\in\C_n}
H_{\mu_1}\otimes\cdots\otimes H_{\mu_k}.\end{equation}
Due to this isomorphism, the representation theory of the algebra $bH_n(q)$ is well known. In fact, we can construct a cell datum for $bH_n(q)$ as follows:
\begin{enumerate}
\item The poset is the set $\LP_n$ of linear multipartitions of $n$ together with the usual \emph{dominance order} of multipartitions, which we will denote by $\triangleright$.
\item For each $\blam\in \LP_n$ we consider $T(\blam)=\{\bt\in\Std(\blam)\mid\bt\text{ of the initial kind}\}$.
\item The antiautomorphism $*:bH_n(q)\to bH_n(q)$ is given by $e_i^*=e_i$ and $z_i^*=z_i$.
\item For each pair $\bs,\bt \in T(\blam)$, we define\[\bmm_{\bs\bt}^\blam=z_{d(\bs)}^\ast\bmm_\blam z_{d(\bt)},\quad\text{where}
\quad\bmm_\blam=\bbE_{\bI_\blam}\sum_{w\in\S_\blam}q^{\ell(w)}z_w.\]
\end{enumerate}

\begin{thm}\label{131}
The tied--boxed Hecke algebra $bH_n(q)$ is cellular with cellular basis:\[B=\left\{\bmm_{\bs\bt}^\blam\mid\blam\in\LP_n\text{ and }\bs,\bt\in T(\blam)\right\}.\]
\end{thm}
\begin{proof}
Since the Young Hecke algebras $H_\mu$ are cellular, see Subsection \ref{095}, the isomorphism in Corollary \ref{isoconHecke} implies that the algebras $bH_n(q)^\mu$ are cellular as well. Moreover, a cellular basis for $bH_n(q)^\mu$ is given by\[B_\mu=\left\{\bmm_{\bs\bt}^\blam\mid\comp(\blam)=\mu\text{ and }\bs,\bt\in T(\blam)\right\}.\]Then, from \eqref{decompositionbH}, it is immediate that $bH_n(q)$ is a free $\qring$--module with basis $B=\coprod_\mu B_\mu$. Before proving that $B$ satisfies the defining conditions of a cellular basis, observe that the elements $\bmm_{\bs\bt}^\blam$ can be decomposed as a product of classical Murphy's elements. Indeed, if $\bs=(\s_1,\ldots,\s_k)$ and $\bt=(\t_1,\ldots,\t_k)$ are standard $\blam=(\lambda^{(1)},\ldots,\lambda^{(k)})$ tableaux of the initial kind, then\[\begin{array}{rcl}
\bmm_{\bs\bt}^\blam&=&z_{d(\bs)}^\ast\left(\bbE_{\bI_\blam}\sum_{w\in\S_\blam}q^{\ell(w)}z_w\right)z_{d(\bt)},\\[2mm]
&=&\bbE_{\bI_\blam}(z_{d(\s_1)}\cdots z_{d(\s_k)})^\ast\cdot(m_{\lambda^{(1)}}\cdots m_{\lambda^{(k)}})\cdot z_{d(\t_1)}\cdots z_{d(\t_k)},\\[2mm]
&=&\bbE_{\bI_\blam}z_{d(\s_k)}^\ast\cdots z_{d(\s_1)}^*\cdot(m_{\lambda^{(1)}}\cdots m_{\lambda^{(k)}})\cdot z_{d(\t_1)}\cdots z_{d(\t_k)}.\end{array}\]Now, as $z_{d(\s_i)}$ and $m_{\lambda^{(j)}}$ belong to different components of the decomposition of $bH_n(q)$ in \eqref{decompositionbH}, we have
\begin{equation}\label{124}
\bmm_{\bs\bt}^\blam=\bbE_{\bI_\blam}(z_{d(\s_l)}^\ast m_{\lambda^{(1)}} z_{d(\t_1)})\cdots(z_{d(\s_k)}^\ast m_{\lambda^{(k)}} z_{d(\t_k)})=\bbE_{\bI_\blam}m^{\lambda^{(1)}}_{\s_1\t_1}\cdots m^{\lambda^{(k)}}_{\s_k\t_k}.\end{equation}Since the factors in the above decomposition commute with each other and $\bbE_{\bI_\blam}$ is invariant under $\ast$ (it is immediate from the definitions), we have\[\left(\bmm_{\bs\bt}^\blam\right)^\ast=\bbE_{\bI_\blam}^\ast \left(m^{\lambda^{(1)}}_{\s_1\t_1}\right)^\ast\cdots\left(m^{\lambda^{(k)}}_{\s_k\t_k}\right)^\ast=\bbE_{\bI_\blam}m^{\lambda^{(1)}}_{\t_1\s_1}\cdots m^{\lambda^{(k)}}_{\t_k\s_k}=\bmm_{\bt\bs}^\blam.\]To prove Condition \textbf{(ii)} it is enough to show that it holds for the cases when $a=e_i$ and $a=z_j$. The case $a=e_i$ is trivial because $\bbE_{\bI_\blam}e_i=\bbE_{\bI_\blam}$ or $\bbE_{\bI_\blam}e_i=0$, see Proposition \ref{propE}. The same occurs for the case when $a=z_i$, with $i$ and $i+1$ belong to different blocks of the set partition $\bI_\blam$, because $z_i=e_iz_i$. Finally, suppose that $a=z_i$, with $i$ and $i+1$ belonging to the block $I_j$ of $\bI_\blam$. Then,\[\bmm_{\bs\bt}^\blam z_i=\left(\bbE_{\bI_\blam} m^{\lambda^{(1)}}_{\s_1\t_1}\cdots m^{\lambda^{(j)}}_{\s_j\t_j} \cdots m^{\lambda^{(k)}}_{\s_k\t_k}\right)z_i=\bbE_{\bI_\blam} m^{\lambda^{(1)}}_{\s_1\t_1}\cdots \left(m^{\lambda^{(j)}}_{\s_j\t_j} z_i\right) \cdots  m^{\lambda^{(k)}}_{\s_k\t_k}.\]By applying Condition \textbf{(ii)} in the factor $m^{\lambda^{(j)}}_{\s_j\t_j} h_i$, we obtain\begin{equation}\label{125}\bmm_{\bs\bt}^\blam z_i=\sum_{\V_j\in \Std(\lambda^{(j)})} r_{\V_j}\bbE_{\bI_\blam} m^{\lambda^{(1)}}_{\s_1\t_1}\cdots m^{\lambda^{(j)}}_{\s_j\V_j} \cdots  m^{\lambda^{(k)}}_{\s_k\t_k} \quad \mod{H_{n_j}^{\lambda^{(j)}}}\end{equation}
where $n_j$ is the size of $\lambda^{(j)}$. Note that if $\mu\triangleright\lambda^{(j)}$, then $\bmu=(\lambda^{(1)},\ldots,\mu,\ldots\lambda^{(k)})\triangleright(\lambda^{(1)},\ldots, \lambda^{(j)}, \lambda^{(k)})=\blam$. Thus, by using the inverse process of the decomposition in \eqref{124}, the equality \eqref{125} can be rewritten as\[\bmm_{\bs\bt}^\blam z_i=\sum_{\bv_j\in\Std(\blam)}r_{\bv_j}\bmm_{\bs\bv_j}^\blam\mod{bH_n^\blam},\]where $\bv_j=(\t_1,\ldots,\V_j,\ldots,\t_k)$. This concludes the proof.
\end{proof}

\begin{rem}\label{celular a celular}
Let $\mu\in \C_n$ and suppose that $\alpha $ is the partition of $n$ obtained by reordering the parts of $\mu$. Then, the embedding $\iota_1:bH_n(q)\hookrightarrow\E_n(q)$ induces an embedding $\iota_1^\mu:bH_n(q)^\mu\hookrightarrow\E_n(q)^\alpha$, which is compatible with the cell structure of these $\qring$--algebras. Indeed, $\iota_1^{\mu}$ sends the basis element $\bmm_{\bs\bt}^\blam$ of $bH_n(q)^\mu$ to the basis element $m_{\bms\bmt}^\Lambda$ of $\E_n(q)^\alpha$ whose indexes we will explain below.

First,  $\Lambda$ is the pair of multipartitions $(\blam^\ord\mid\bmu)$, where $\blam^\ord$ is the increasing multipartition obtained by reordering the parts of $\blam=(\lambda^{(1)},\ldots,\lambda^{(k)})$ and $\bmu=((1^{m_1}),\ldots,(1^{m_l}))$, where $m_i$ are the multiplicities of equal $\lambda^{(i)}$'s in $\blam^\ord$. On the other hand, $\bms$ (resp. $\bmt$) is the standard $\Lambda$--tableau $(\bs^\ord\mid\bt^\bmu)$  (resp. $(\bt^\ord\mid\bt^\bmu)$), where $\bs^\ord$ (resp. $\bt^\ord$) is the increasing $\blam^\ord$--multitableau obtained by reordering the parts of $\bs$ (resp. $\bt$). See Figure \ref{122} for an example.\begin{figure}[H]
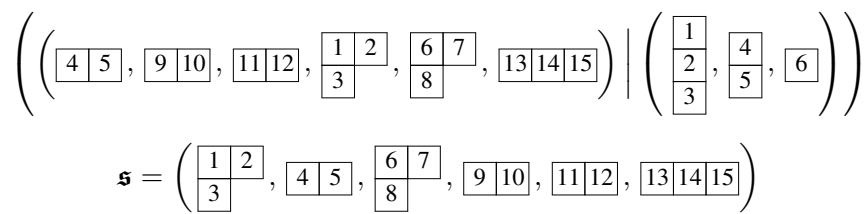

\figtwofiv\caption{The $\Lambda$--tableaux associated to the multitableau $\bs$.}\label{122}
\end{figure}
\end{rem}

\subsection{Diagrammatic realization of $bH_n(q)$}\label{111}

Here we study the boxed ramified monoid associated with the symmetric group, $\BR(\S_n)$, which can be regarded as a basis for the algebra $bH_n(q)$. Specifically, we get that $bH_n(q)$ is a $q$--deformation of the monoid algebra generated by $\BR(\S_n)$ (Theorem \ref{017}), providing $bH_n(q)$ with a structure of diagram algebra. Also, we show that the center of $\BR(\S_n)$ coincides with the monoid of compositions of $n$ (Proposition \ref{089}).

It is well known that $|\S_m|=m!$ for all integer $m\geq1$ \cite[A356634]{OEIS}. So, similarly as it was done in Proposition \ref{013} and Remark \ref{021}, we have\[|\BR(\S_n)|=\sum_{(\mu_1,\ldots,\mu_k)}\mu_1!\cdots\mu_k!=(1,3,11,47,231,1303,8431,62391,524495,4960775,\ldots),\]where $(\mu_1,\ldots,\mu_k)$ is a composition of $n$. See \cite[A051296]{OEIS}.

For each $i\in[n-1]$, define $z_i$ to be the ramified partition $e_is_i$ in $\R(\S_n)$, which is represented as in Figure \ref{015}. Due to the relations that define $\R(S_n)$, the elements $z_i$ satisfy the following relations:\begin{gather}
z_iz_jz_i=z_jz_iz_j,\quad|i-j|=1;\qquad z_iz_j=z_jz_i,\quad|i-j|>1;\qquad e_iz_j=z_je_i;\label{022}\\
z_i^2=e_i;\qquad e_iz_i=z_i.\label{026}
\end{gather}
\begin{figure}[H]
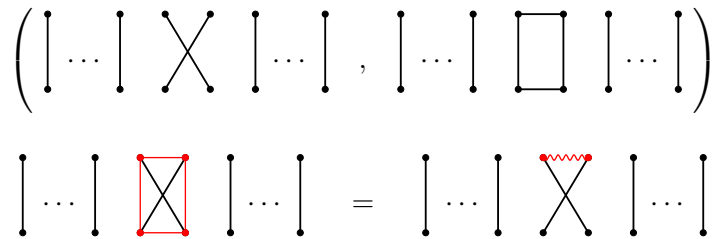
\figele\caption{Generator $z_i=e_is_i$.}\label{015}\end{figure}
\begin{figure}[H]\figsvt\caption{Some relations in $\BR(\S_n)$.}\label{039}\end{figure}

\begin{pro}\label{025}
The monoid $\BR(\S_n)$ is generated by $e_1,\ldots,e_n$ and $z_1,\ldots,z_{n-1}$.
\end{pro}
\begin{proof}
Let $(I,J)\in\BR(\S_n)$. We have $(I,J)=(\tilde{I}_1,\tilde{J}_1)/\cdots/(\tilde{I}_k,\tilde{J}_k)$ with $(\tilde{I}_i,\tilde{J}_i)\in\R(\CC_{m_i})$ for all $i\in[k]$, where $I=\tilde{I}_1/\cdots/\tilde{I}_k$, and $J=\tilde{J}_1/\cdots/\tilde{J}_k$ is the boxed decompositions of $J$. Thus, for each $i\in[k]$, we have $(\tilde{I}_i,\tilde{J}_i)=(e_1\cdots e_{m_i})p_i$, where $p_i\in\S_{m_i}$. In general, if $s=s_{i_1}\cdots s_{i_q}$ is a permutation in $\S_m$, we have $(e_1\cdots e_m)s=(e_1\cdots e_m)s_{i_1}\cdots s_{i_q}=(e_1\cdots e_m)(e_{i_1}s_{i_1})\cdots(e_{i_q}s_{i_q})=(e_1\cdots e_m)z_{i_1}\cdots z_{i_q}$ because of \eqref{026}. Therefore $\BR(\S_n)$ is generated by $e_1,\ldots,e_n$, $z_1,\ldots,z_{n-1}$.
\end{proof}

\begin{thm}\label{017}
The monoid $\BR(\S_n)$ is presented by generators $e_1,\ldots,e_{n-1}$ satisfying \eqref{023}, and generators $z_1,\ldots,z_{n-1}$ subject to \eqref{022} and \eqref{026}.
\end{thm}
\begin{proof}
Observe that $\C_m\times\B_n^+$ is presented by generators $e_1,\ldots,e_{n-1}$ and $\sigma_1,\ldots,\sigma_{n-1}$ satisfying \eqref{023} and \eqref{024} together with the relation $e_i\sigma_j=\sigma_je_i$. So, Proposition \ref{025} and \eqref{022} imply that $\varphi:\C_n\times\B_n^+\to\BR(\S_n)$ defined by $e_i\mapsto e_i$ and $\sigma_i\mapsto z_i$ is an epimorphism. Now, let $K$ be the congruence closure generated by the pairs $(\sigma_i^2,e_i)$ and $(e_i\sigma_i,\sigma_i)$ for all $i\in[n-1]$. Observe that $K\subseteq\Ker(\varphi)$ because of \eqref{026}.

Consider $e,e'\in\C_n$ and $\sigma,\sigma'\in\B_n^+$ such that $\sigma=\sigma_{a_1}\cdots\sigma_{a_p}$, $\sigma'=\sigma_{b_1}\cdots\sigma_{b_q}$ and $\varphi(e\sigma)=\varphi(e'\sigma')$. Then, because of \eqref{026}, we have\begin{equation}\label{027}\varphi(e\sigma)=ee_{a_1}\cdots e_{a_p}s_{a_1}\cdots s_{a_p},\qquad \varphi(e'\sigma')=e'e_{b_1}\cdots e_{b_q}s_{b_1}\cdots s_{b_p}.\end{equation}Observe that $e\sigma\equiv_Kee_{a_1}\cdots e_{a_p}\sigma$ and $e'\sigma'\equiv_Ke'e_{b_1}\cdots e_{b_q}\sigma'$ because $(\sigma_i^2,e_i)\in K$. On the other hand, since $\R(\S_n)=\P_n\rtimes\S_n$, the decompositions of $\varphi(e\sigma)$ and $\varphi(e'\sigma')$ in \eqref{027} are unique, hence\[ee_{a_1}\cdots e_{a_p}=e'e_{b_1}\cdots e_{b_q},\qquad s_{a_1}\cdots s_{a_p}=s_{b_1}\cdots s_{b_q}.\]Note that each relation of $\S_n$ used to get $s_{b_1}\cdots s_{b_q}$ from $s_{a_1}\cdots s_{a_p}$ can be translated into a relation of $z_i$'s because if $s_i$ occurs above then $e_i$ occurs as well. Thus $e\sigma\equiv_Kee_{a_1}\cdots e_{a_p}\sigma\equiv_Ke'e_{b_1}\cdots e_{b_q}\sigma'\equiv_Ke'\sigma'$, and then $\Ker(\varphi)\subseteq K$. Therefore $\BR(\S_n)\simeq(\C_n\times\B_n^+)/K$, which proves the theorem.
\end{proof}

It is immediate from Theorem \ref{017} that $bH_n(q)$ is a $q$--deformation of the monoid algebra generated by $\BR(S_n)$. Thus, the algebra $bH_n(q)$ inherits the diagram combinatorics of $\BR(\S_n)$. See Figure \ref{121}.
\begin{figure}[H]
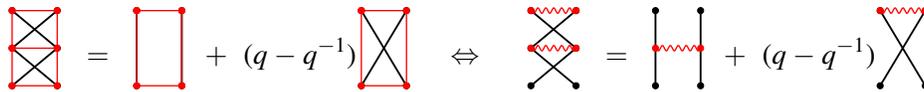

\figtwosix\caption{Relation \eqref{h4} in terms of diagrams.}\label{121}
\end{figure}

\begin{rem}\label{087}
By using Tietze transformations, and due to the fact that $e_i=z_i^2$, the monoid $\BR(\S_n)$ can be presented by generators $z_1,\ldots,z_{n-1}$, subject to the following relations:\[z_iz_jz_i=z_jz_iz_j,\quad|i-j|=1;\qquad z_iz_j=z_jz_i,\quad|i-j|>1;\qquad z_i^3=z_i;\qquad z_i^2z_j^2=z_j^2z_i^2;\qquad z_i^2z_j=z_jz_i^2.\]
\end{rem}

\begin{rem}[Normal form]\label{075}
As shown in the proofs of Proposition \ref{025} and Theorem \ref{017}, every element $g\in\BR(\S_n)$ can be uniquely written as $ez$ with $e\in\C_n$ and $z=z_{i_1}\cdots z_{i_k}$, where $es_{i_1}\cdots s_{i_k}$ is a normal decomposition given by the semidirect product structure of $\R(\S_n)$. Moreover, as $e$ is formed by blocks, then $g=ze$ as well. See Figure \ref{076}.
\end{rem}
\begin{figure}[H]
\figtwothr\caption{Normal form of $e_1e_2e_3s_2s_1s_3s_2s_3=e_1e_2e_3z_2z_1z_3z_2z_3=z_2z_1z_3z_2z_3e_1e_2e_3$.}
\label{076}
\end{figure}

\begin{pro}\label{089}
We have $Z(\BR(\S_n))=\C_n$ for all $n\geq3$.
\end{pro}
\begin{proof}
Due to Theorem \ref{017} and the relations that define $\S_n$, we conclude that the map $\phi:\BR(\S_n)\to\S_n$ sending $z_i\mapsto s_i$ and $e_i\mapsto1$ is a monoid epimorphism. So, $\phi(g)\in Z(\S_n)=\{1\}$ for all $g\in Z(\BR(\S_n))$. Observe that $\phi(\C_n)=\{1\}$ and $\C_n\subseteq\BR(\S_n)$ as mentioned in Remark \ref{043}. On the other hand, as each relation in \eqref{032} can be obtained from one in \eqref{022} and \eqref{026}, then every element of $\BR(\S_n)$ can be written as a product $ez$, where $e\in\C_n$ and $z=z_{i_1}\cdots z_{i_k}$ such that $\phi(z)=s_{i_1}\cdots s_{i_k}$ is a reduced expression in $\S_n$. This implies that $\phi(g)=1$ for some $g\in\BR(\S_n)$ if and only if $g\in\C_n$. Therefore $Z(\BR(\S_n))=\C_n$.
\end{proof}

\begin{rem}
Just as the symmetric group can be obtained by imposing the relations $\sigma_i^2=1$ in the braid group, the ramified symmetric monoid can be obtained in this manner from the tied braid monoid. It is expected that a monoid playing this role for $\BR(\S_n)$ is the submonoid $bT\B_n$ of $T\B_n$ generated by $\tau_i:=e_i\sigma_i$ and $\ttau_i:=e_i\sigma_i^{-1}$ for all $i\in[n-1]$, together with $e_1,\ldots,e_{n-1}$ satisfying \eqref{023}. This submonoid will be called the \emph{tied--boxed braid monoid}. Observe that, due to \eqref{040} and \eqref{041}, the generators $\tau_i$ and $\ttau_i$ satisfy the relations below. It is an open problem to check if these relations are sufficient to give a presentation.
\begin{gather*}
\tau_i\tau_j\tau_i=\tau_j\tau_i\tau_j,\quad\ttau_i\ttau_j\ttau_i=\ttau_j\ttau_i\ttau_j,\quad|i-j|=1;\\
\tau_i\ttau_i=e_i=\ttau_i\tau_i;\qquad e_i\tau_j=\tau_je_i;\qquad e_i\ttau_j=\ttau_je_i;\qquad e_i\tau_i=\tau_i;\qquad e_i\ttau_i=\ttau_i.
\end{gather*}
\end{rem}

\subsection{The singular part of $\R(\S_n)$}\label{112}

As mentioned in Remark \ref{043}, $\BR(\S_n)\backslash\{1\}$ is a subsemigroup of the singular part of $\R(\S_n)$. Hence, the monoid $s\R(\S_n)^1$ can be regarded as an extension of $\BR(\S_n)$. Here, we study the singular part of the ramified symmetric monoid, in particular we give a presentation for it (Theorem \ref{055}).

In this subsection, for each $r_1,\ldots,r_k\in[n-1]$ and each $i\in[n]$, we will denote the image of $i$ under the permutation $s_{r_1}\cdots s_{r_k}$ by $(r_1\circ\cdots\circ r_k)(i)$ instead.

Observe that $s\R(\S_n)=\{(I,J)\mid I\in\S_n,\,I\prec J\}$ and $|s\R(\S_n)|=n!\b_n-n!=n!(\b_n-1)$.
See Figure \ref{050}.\begin{figure}[H]\figtty\caption{The $24$ elements of $s\R(\S_3)$.}\label{050}\end{figure}

For each $i,j,r\in[n-1]$ with $j<k$, we define $z^r_{i,j}=e_{i,j}s_r$. See Figure \ref{044}. Note that $z^i_{i,i+1}=z_i$. Also, there are $\frac{n^2(n-1)}{2}$ \cite[A006002]{OEIS} of these elements, and each $z_{i,j}^r$ belongs to $s\R(\S_n)$ due to occurrence of $e_{i,j}$.
\begin{figure}[H]
\figetn\caption{Elements $z^2_{1,4}$ and $z^1_{2,4}$ in $\R(\S_n)$.}\label{044}
\end{figure}

\begin{pro}\label{056}
The semigroup $s\R(\S_n)$ is generated by the elements $z^r_{i,j}$ and $e_{i,j}$.
\end{pro}
\begin{proof}
Since $s\R(\S_n)$ is a subsemigroup of $\R(S_n)=\P_n\rtimes\S_n$, every element $g\in s\R(\S_n)$ can be uniquely written as a product $es$, where $e\in\P_n$ and $s\in\S_n$. Let $e_{p_1,q_1}\cdots e_{p_h,q_h}$ be the normal form of $e$ as in \cite[Proposition 3.3]{ArJu21}, and let $s_{r_1}\cdots s_{r_k}$ be a word of minimal length representing $s$. Then, we have\begin{equation}\label{068}g=e_{p_1,q_1}\cdots e_{p_h,q_h}s_{r_1}\cdots s_{r_k}=e_{p_2,q_2}\cdots e_{p_h,q_h}(e_{\bar{p}_1,\bar{q}_1}s_{r_1})\cdots(e_{\bar{p}_k,\bar{q}_k}s_{r_k})=e_{p_2,q_2}\cdots e_{p_h,q_h}z_{\bar{p}_1,\bar{q}_1}^{r_1}\cdots z_{\bar{p}_k,\bar{q}_k}^{r_k},\end{equation}where $e_{\bar{p}_i,\bar{q}_i}=s_{r_{i-1}}\cdots s_{r_1}e_{p_1,q_1}s_{r_1}\cdots s_{r_{i-1}}$. Therefore $s\R(\S_n)$ is generated by the elements $z^r_{i,j}$ and $e_{i,j}$.
\end{proof}

\begin{rem}[Normal form]\label{083}
Observe that the word $u=e_{p_2,q_2}\cdots e_{p_h,q_h}z_{\bar{p}_1,\bar{q}_1}^{r_1}\cdots z_{\bar{p}_k,\bar{q}_k}^{r_k}$ representing the element $g$ in \eqref{068} is uniquely defined up to the choice of the minimal length word $v=s_{r_1}\cdots s_{r_k}$. We may assume that $v$ is chosen in some minimal length normal form of $s$. Thus, $u$ will be a normal form for $g$. See Figure \ref{072} for an instance.
\end{rem}
\begin{figure}[H]
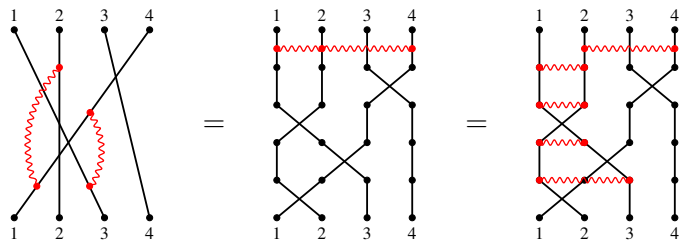
\figtwotwo\caption{Normal form of $e_{1,2}e_{2,4}s_3s_1s_2s_1=e_{2,4}z_{1,2}^3z_{1,2}^1z_{1,2}^2z_{1,3}^1$.}\label{072}\end{figure}

Due to the fact that $\R(\S_n)=\P_n\rtimes\S_n$ and the relations that define it, we can show that the generators $z_{i,j}^r$ satisfy the following relations:\begin{gather}
z^r_{i,j}z^t_{r(i),r(j)}z^r_{(t\circ r)(i),(t\circ r)(j)}=z^t_{i,j}z^r_{t(i),t(j)}z^t_{(r\circ t)(i),(r\circ t)(j)},\quad|r-t|=1;\label{051}\\
z^r_{i,j}z^t_{r(i),r(j)}=z^t_{i,j}z^r_{t(i),t(j)},\quad|r-t|>1;\label{052}\\
z^r_{i,j}z^r_{h,k}=e_{i,j}e_{r(h),r(k)};\qquad e_{i,j}z^r_{i,j}=z^r_{i,j};\qquad z^r_{i,j}e_{h,k}=e_{r(h),r(k)}z^r_{i,j}.\label{053}
\end{gather}Observe that \eqref{051} and \eqref{052} are inherited from the braid relations that the generators $s_i$ hold. See Figure \ref{054}. Furthermore, \eqref{053} implies it is enough to consider the generators $z^r_{i,j}$ because $e_{i,j}=e_{i,j}^2=z^r_{i,j}z^r_{r(i),r(j)}$.
\begin{figure}[H]
\figtwoone\caption{Relations $z^1_{2,3}z^2_{1,3}z^1_{1,2}=z^2_{2,3}z^1_{2,3}z^2_{1,3}$ and $z^1_{2,4}z^2_{1,4}=z^2_{2,4}z^1_{2,3}$.}\label{054}
\end{figure}

The main goal of this subsection is to prove the following result.

\begin{thm}\label{055}
The semigroup $s\R(\S_n)$ is presented by generators $e_{i,j}$ satisfying \eqref{047}, and generators $z^r_{i,j}$, subject to the relations \eqref{051} to \eqref{053}.
\end{thm}

\begin{rem}\label{069}
Observe that $gh,hg\in s\R(\S_n)$ for all $g\in\S_n$ and $h\in s\R(\S_n)$. Hence $s\R(\S_n)$ is an ideal of the monoid $\R(\S_n)$. So, another manner to show a presentation for $s\R(\S_n)$ can be obtained by applying the Reidemeister--Schreier type rewriting for ideals given in \cite{Ru95,CaRoRuTh95}.
\end{rem}

\subsubsection{Proof of Theorem \ref{055}}\label{071}

Let $S$ be the semigroup presented by generators ${\tt e}_{i,j}$ and ${\tt z}^r_{i,j}$ with $r\in[n-1]$ and $i,j\in[n]$ such that $i<j$, subject to the following relations:\begin{gather}
\te_{i,j}^2=\te_{i,j},\quad\te_{i,j}\te_{r,s}=\te_{r,s}\te_{i,j},\quad\te_{i,j}\te_{i,k}=\te_{i,j}\te_{j,k}=\te_{i,k}\te_{j,k};\label{058}\\
\tz^r_{i,j}\tz^t_{r(i),r(j)}\tz^r_{(t\circ r)(i),(t\circ r)(j)}=\tz^t_{i,j}\tz^r_{t(i),t(j)}\tz^t_{(r\circ t)(i),(r\circ t)(j)},\quad|r-t|=1;\label{059}\\
\tz^r_{i,j}\tz^t_{r(i),r(j)}=\tz^t_{i,j}\tz^r_{t(i),t(j)},\quad|r-t|>1;\label{060}\\
\tz^r_{i,j}\tz^r_{h,k}=\te_{i,j}\te_{r(h),r(k)};\qquad\te_{i,j}\tz^r_{i,j}=\tz^r_{i,j};\qquad\tz^r_{i,j}\te_{h,k}=\te_{r(h),r(k)}\tz^r_{i,j}.\label{061}
\end{gather}

As relations \eqref{047} and \eqref{051}--\eqref{053} hold in $s\R(\S_n)$, the map $\phi$ sending $\te_{i,j}\mapsto e_{i,j}$ and $\tz_{i,j}^r\mapsto z_{i,j}^r$ for all $r\in[n-1]$ and $i,j\in[n]$ with $i<j$, defines a semigroup epimorphism from $S$ to $s\R(\S_n)$.

\begin{lem}\label{065}
The following relations hold in $S$.
\begin{enumerate}
\item $\te_{i,j}\tz_{h,k}^r=\te_{h,k}\tz_{i,j}^r$.\label{062}
\item $\tz^r_{i,j}\tz^t_{h,k}\tz^r_{p,q}=\tz^t_{i,j}\tz^r_{(t\circ r)(h),(t\circ r)(k)}\tz^t_{(t\circ r)(p),(t\circ r)(q)}$ if $|r-t|=1$.\label{063}
\item $\tz_{i,j}^r\tz_{h,k}^t=\tz_{i,j}^t\tz_{(t\circ r)(h),(t\circ r)(k)}^r$ if $|r-t|>1$.\label{064}
\end{enumerate}
\end{lem}
\begin{proof}
By using \eqref{061}, we show \eqref{062} as follows:\[\te_{i,j}\tz_{h,k}^r=\te_{i,j}\te_{h,k}\tz_{h,k}^r=\tz^r_{i,j}\tz^r_{r(h),r(k)}\tz_{h,k}^r=\tz^r_{i,j}\te_{r(h),r(k)}\te_{r(h),r(k)}=\tz^r_{i,j}\te_{r(h),r(k)}=\te_{h,k}\tz^r_{i,j}.\]To show \eqref{063} and \eqref{064}, we apply \eqref{059} and \eqref{060} together with \eqref{062} and \eqref{061}, as follows:\[\begin{array}{rcl}
\tz^r_{i,j}\tz^t_{h,k}\tz^r_{p,q}&=&\te_{i,j}^2\tz^r_{i,j}\tz^t_{h,k}\tz^r_{p,q},\\
&=&\tz^r_{i,j}\te_{r(i),r(j)}\tz^t_{h,k}\te_{(t\circ r)(i),(t\circ r)(j)}\tz^r_{p,q},\\
&=&\tz^r_{i,j}\te_{h,k}\tz^t_{r(i),r(j)}\te_{p,q}\tz^r_{(t\circ r)(i),(t\circ r)(j)},\\
&=&\te_{r(h),r(k)}\te_{(r\circ t)(p),(r\circ t)(q)}\tz^r_{i,j}\tz^t_{r(i),r(j)}\tz^r_{(t\circ r)(i),(t\circ r)(j)},\\
&=&\te_{r(h),r(k)}\te_{(r\circ t)(p),(r\circ t)(q)}\tz^t_{i,j}\tz^r_{t(i),t(j)}\tz^t_{(r\circ t)(i),(r\circ t)(j)},\\
&=&\tz^t_{i,j}\te_{(t\circ r)(h),(t\circ r)(k)}\tz^r_{t(i),t(j)}\te_{(t\circ r)(p),(t\circ r)(q)}\tz^t_{(r\circ t)(i),(r\circ t)(j)},\\
&=&\tz^t_{i,j}\te_{t(i),t(j)}\tz^r_{(t\circ r)(h),(t\circ r)(k)}\te_{(r\circ t)(i),(r\circ t)(j)}\tz^t_{(t\circ r)(p),(t\circ r)(q)},\\
&=&\te_{i,j}^2\tz^t_{i,j}\tz^r_{(t\circ r)(h),(t\circ r)(k)}\tz^t_{(t\circ r)(p),(t\circ r)(q)},\\
&=&\tz^t_{i,j}\tz^r_{(t\circ r)(h),(t\circ r)(k)}\tz^t_{(t\circ r)(p),(t\circ r)(q)}.
\end{array}
\begin{array}{rcl}
\tz_{i,j}^r\tz_{h,k}^t&=&\te_{i,j}\tz_{i,j}^r\tz_{h,k}^t,\\
&=&\tz_{i,j}^r\te_{r(i),r(j)}\tz_{h,k}^t,\\
&=&\tz_{i,j}^r\te_{h,k}\tz_{r(i),r(j)}^t,\\
&=&\te_{r(h),r(k)}\tz_{i,j}^r\tz_{r(i),r(j)}^t,\\
&=&\te_{r(h),r(k)}\tz_{i,j}^t\tz_{t(i),t(j)}^r,\\
&=&\tz_{i,j}^t\te_{(t\circ r)(h),(t\circ r)(k)}\tz_{t(i),t(j)}^r,\\
&=&\tz_{i,j}^t\te_{t(i),t(j)}\tz_{(t\circ r)(h),(t\circ r)(k)}^r,\\
&=&\te_{i,j}\tz_{i,j}^t\tz_{(t\circ r)(h),(t\circ r)(k)}^r,\\
&=&\tz_{i,j}^t\tz_{(t\circ r)(h),(t\circ r)(k)}^r.
\end{array}\]
\end{proof}

Parts \eqref{063} and \eqref{064} of Lemma \ref{065} imply that, independently of the subindices, the generators $\tz^r_{i,j}$ satisfy the braid relations with respect of their superindices, that is\begin{equation}\label{066}\tsz^r\tsz^t\tsz^r=\tsz^t\tsz^r\tsz^t,\quad|i-j|=1;\qquad\tsz^r\tsz^t=\tsz^t\tsz^r,\quad|i-j|>1;\end{equation}where the square is hiding the subindices of each generator. Similarly, relations in \eqref{061} can be written as:\begin{equation}\label{067}\tz_\square^r\tz_\square^r=\te_\square;\qquad \te_\square\tz_\square^r=\tz_\square^r;\qquad\tz_\square^r\te_\square=\te_\square\tz_\square^r;\end{equation}Thus, we obtain the following result.

\begin{crl}
The map $\phi_s:S\to\S_n$ sending $\te_{i,j}\mapsto1$ and $\tz_{i,j}^r\mapsto s_r$ defines a semigroup endomorphism.
\end{crl}

\begin{pro}\label{070}
Every element $g\in S$ can be represented by a word in the generators $\te_{i,j}$ and $\tz_{i,j}^r$, such that  the word obtained from it when replacing $\te_{i,j}$ by $e_{i,j}$ and $\tz_{i,j}^r$ by $z_{i,j}^r$ is a normal form of $\phi(g)$.
\end{pro}
\begin{proof}
Due to the third relation in \eqref{067}, every element $g\in S$ can be represented by a word $eu$, where $e$ is a word in the generators $\te_{i,j}$ and $u$ is a word in the generators $\tz_{i,j}^r$. Moreover, because of \eqref{066} and \eqref{067}, we can assume that $u=\tsz^{r_1}\cdots\tsz^{r_k}$ is of minimal length, that is, $\phi_s(eu)=s_{r_1}\cdots s_{r_k}$ is of minimal length. Observe that, by applying \eqref{067}, we have\[u\equiv(\underbrace{\te_{\bar{p}_1,\bar{q}_1}\cdots \te_{\bar{p}_k,\bar{q}_k}}_{e'})\underbrace{\tsz^{r_1}\cdots\tsz^{r_k}}_z,\quad\text{where}\quad\tsz^{r_i}=\tz_{p_i,q_i}^{r_i},\quad\te_{\bar{p}_i,\bar{q}_i}=\te_{(r_1\circ\cdots\circ r_{i-1})(p_i),(r_1\circ\cdots\circ r_{i-1})(q_i)}.\]Thus $\phi(g)=\phi(ee')\phi(z)=\phi(ee')\phi_s(z)$ because each tie connecting strands in $\phi(z)$ is already occurring in $\phi(ee')$. Moreover, due to \eqref{058}, we can assume that $ee'$ is written in normal form as in \cite[Proposition 3.3]{ArJu21}. Finally, by using \eqref{062} in Lemma \ref{065}, the first generator $e_{a_1,a_2}$ of $ee'$ can be removed and each generator $\tz^{r_i}_{p_i,q_i}$ can be replaced by $\tz^{r_i}_{\tilde{p}_i,\tilde{q}_i}$, where $\tilde{p}_i=(r_{i-1}\circ\cdots\circ r_1)(p_i)$ and $\tilde{q}_i=(r_{i-1}\circ\cdots\circ r_1)(q_i)$. Thus, the word obtained from $ee'z$ when replacing $\te_{i,j}$ by $e_{i,j}$ and $\tz_{i,j}^r$ by $z_{i,j}^r$ is a normal form of $\phi(g)$.
\end{proof}

\begin{proof}[Proof of Theorem \ref{055}]
Let $g,g'\in S$ such that $\phi(g)=\phi(g')$, and let $u$ be the normal form of $\phi(g)$ as in Remark \ref{083}. Proposition \ref{070} implies both $g,g'$ have word representatives in the generators $\te_{i,j}$ and $\tz_{i,j}^r$, which become $u$ when replace $\te_{i,j}$ by $e_{i,j}$ and $\tz_{i,j}^r$ by $z_{i,j}^r$. Since $u$ is uniquely defined, these words must be identical. Thus $g=g'$. Therefore $\phi$ is an isomorphism.
\end{proof}

\section{The tied--boxed Temperley--Lieb algebra}\label{113}

In this section, inspired by the quotient of $\E_n(q)$ in \cite{RyH22}, which in turn was inspired by \cite{Ha99}, we introduce the \emph{tied--boxed Temperley--Lieb algebra}. Then, we give a cellular basis for it (Theorem \ref{132}) and study its connection with the \emph{partition Temperley--Lieb algebra} introduced by Juyumaya \cite{Ju13} (Subsection \ref{114}). Further, we show a diagrammatic realization through the boxed ramified monoid of the Jones monoid (Subsection \ref{115}). Additionally, we study the boxed ramified monoid of the Brauer monoid (Subsection \ref{116}).

The \emph{tied--boxed Temperley--Lieb algebra} $bTL_n(q)$ is the quotient of $bH_n(q)$ by the two--sided ideal generated by the following \emph{Steinberg elements}:
\begin{equation}\label{steinberg}
z_{i,\,j}=e_ie_j(1+qz_i+qz_j+q^2z_iz_j+q^2z_jz_i+q^3z_iz_jz_i),\qquad |i-j|=1.
\end{equation}

\begin{pro}\label{134}
The tied--boxed Temperley--Lieb algebra $bTL_n(q)$ is presented by generators $e_1,\ldots,e_{n-1}$ satisfying \eqref{023}, and generators $d_i=q^{-1}e_i+z_i$ for all $i\in[n-1]$, subject to the following relations:\begin{gather}
d_i^2=(q+q^{-1})d_i;\qquad d_id_jd_i=e_jd_ie_j,\quad|i-j|=1;\label{R1}\\
d_id_j=d_jd_i,\quad|i-j|>1;\qquad d_ie_i=d_i;\qquad d_ie_j=e_jd_i;\label{R2}
\end{gather}
\end{pro}
\begin{proof}
By definition, the algebra $bTL_n(q)$ is presented by the same generators and relations than $bH_n(q)$, together with the additional relation $z_{i,\,j}=0$ for each Steinberg element in \eqref{steinberg}. Observe that $z_i=d_i-q^{-1}e_i$ for all $i\in[i-1]$, then, in order to get a presentation for $bTL_n(q)$ by generators $d_i$ and $e_i$, it is enough to replace each $z_i$ with $d_i-q^{-1}e_i$ in the relations of this algebra.

Relations in \eqref{R2} are directly equivalent to the second relation in \eqref{h2} and the ones in \eqref{h3}, respectively. By using \eqref{R2} we easily show that \eqref{h4} is equivalent to the quadratic relation in \eqref{R1}. Thus, we get\[z_{i,\,j}=e_ie_j+qe_jz_i+qe_iz_j+q^2z_iz_j+q^2z_jz_i+q^3z_iz_jz_i=q^3d_id_jd_i-q^2e_jd_i^2+qe_jd_i=q^3(d_id_jd_i-e_jd_i).\]This implies that $z_{i,\,j}=0$ is equivalent to the second relation in \eqref{R1}. Finally, observe that\[z_iz_jz_i=-q^{-1}d_jd_i-q^{-1}d_id_j+q^{-2}e_id_j+q^{-2}e_jd_i-q^{-3}e_ie_j\quad\text{for all}\quad|i-j|=1.\]Thus, by using \eqref{R1} and \eqref{R2}, the first relation in \eqref{h2} is superfluous. This concludes the proof.
\end{proof}

As with the tied--boxed Hecke algebra, for each $\mu\in\C_n$, we can define subalgebras $bTL_n^\mu=\bbE_\bI bTL_n^\mu$ of $bTL_n$, where $\bI$ is the unique linear set partition satisfying $\|\bI\|=\mu$. Moreover, there is a two--side ideal decomposition of $bTL_n$ given as follows:\begin{equation}\label{descomposiciondebTLn}bTL_n=\bigoplus_{\mu\in\C_n}bTL_n^\mu\end{equation}

A similar argue to the one used in the proof of Corollary \ref{isoconHecke} implies that\[bTL_n^\mu\simeq TL_\mu=TL_{\mu_1}\otimes\cdots\otimes TL_{\mu_k},\qquad\mu=(\mu_1,\ldots,\mu_k)\in\mathcal{C}_n.\]Indeed, from the decomposition in factors given in (\ref{124}), we get the following isomorphism of $\qring$--algebras in terms of the cellular basis:\[\begin{array}{ccc}
TL_{\mu_1}\otimes\cdots\otimes TL_{\mu_k}&\to&bTL_n^\mu\\[2.5mm]m^{\lambda^{(1)}}_{\s_1\t_1}\otimes\cdots\otimes m^{\lambda^{(k)}}_{\s_k\t_k}&\mapsto&\bbE_{\bI_\blam}m^{\lambda^{(1)}}_{\s_1\t_1}\cdots m^{\lambda^{(k)}}_{\s_k\t_k}=\bmm_{\bs\bt}^\blam\end{array}\]where $\blam=(\lambda^{(1)},\ldots,\lambda^{(k)})\in\MPar_n^{\leq2}$, and $\bs=(\s_1,\ldots,\s_k)$ and $\bt=(\t_1,\ldots,\t_k)$ are standard $\blam$--multitableaux of the initial kind. Thus, due to \eqref{descomposiciondebTLn}, we obtain the following isomorphism of $\qring$--algebras:\[bTL_n\,\simeq\,\bigoplus_{\mu\in\C_n} TL_{\mu_1}\otimes\cdots\otimes TL_{\mu_k}.\]In particular, $\dim bTL_n(q)=\sum_{\mu\in \mathcal{C}_n}\c_{\mu_1}\cdots\c_{\mu_k}$.

The following theorem is immediate from the previous results.
\begin{thm}\label{132}
For each $\mu\in\C_n$, the algebra $bTL_n^\mu$ is cellular with cellular basis\[B_{bTL_n^\mu}=\{\bmm_{\bs\bt}^\blam\mid \bs,\bt\in T(\blam),\,\blam\in\MPar_n^{\leq2}\text{ and }\comp(\blam)=\mu\}.\]In consequence, the algebra $bTL_n(q)$ is cellular with cellular basis $B_{bTL_n}=\{\bmm_{\bs\bt}^\blam\mid\blam\in\MPar_n^{\leq2},\,\bs,\bt\in T(\blam)\}$.
\end{thm}

\subsection{Connection with the partition Temperley--Lieb algebra}\label{114}

Here we explore a strong connection between $bTL_n(q)$ and the \emph{partition Temperley--Lieb algebra} $\PTL_n(q)$ \cite{Ju13}. It was shown in \cite{RyH22} that a generalized version of $\PTL_n(q)$ is cellular by finding a concrete cellular basis for it. Here we rediscover this result for the particular case of the $\PTL_n(q)$. Further, we establish an explicit isomorphism between $\PTL_n(q)$ and a direct sum of matrix algebras (Theorem \ref{PTLn}).

For every integer $n\geq3$, \emph{the partition Temperley--Lieb algebra} $\PTL_n(q)$ is the quotient $\E_n(q)/J_n$, where $J_n$ is the two--sided ideal of $\E_n(q)$ generated by the elements $e_ie_jg_{i,j}$, where $g_{i,j}$ are the \emph{Steinberg elements}:\[g_{i,j}=1+qg_i+qg_j+q^{2}g_ig_j+q^2g_jg_i+q^3g_ig_jg_i,\qquad|i-j|=1.\]

Due to the decomposition of $\E_n(q)$ as a direct sum of ideals $\E_n^\alpha(q)$ in \eqref{093}, we get the following induced decomposition of $\PTL_n(q)$:
\begin{equation}
\PTL_n(q)\,=\bigoplus_{\alpha\in\Par_n}\PTL_n^\alpha(q),\quad\text{where}\quad\PTL_n^\alpha(q)=\bbE_\alpha\PTL_n(q).
\end{equation}
Now, for each $\alpha=(k_1^{m_1},\ldots,k_r^{m_r})\in\Par_n$ such that $k_1>\cdots>k_r$, we define a Temperley--Lieb version of the algebra $H^{wr}_\alpha(q)$, as follows:\[TL_\alpha^{wr}(q)=TL_{k_1}(q)\wr\S_{m_1}\otimes\cdots\otimes TL_{k_r}(q)\wr\S_{m_r}.\]

The following result is partially inspired by \cite[Theorem 4.7]{ChPo17}, which establishes an isomorphism between a Temperley--Lieb--type quotient of the Yokonuma--Hecke algebra $\FTL_{r,n}(q)$ and a direct sum of matrix algebras over certain Hecke algebras.
\begin{thm}\label{PTLn}
For each partition $\alpha=(n_1,\ldots,n_r)$ of $n$, the isomorphism $\Psi_\alpha$ sending $m_{\es\et}\mapsto x_{\es_0\et_0}M_{\bs\bt}$ in Theorem \ref{isomorfismoconmatrices}, induces an isomorphism\[\PTL_n^\alpha(q)\simeq\Mat_{b_n(\alpha)}\big(TL^{wr}_\alpha(q)\big).\]
\end{thm}

Taking into account the cellular basis of $TL_n(q)$ in \eqref{cellularbasismurphyTL}, we obtain the following results as immediate consequences of Theorem \ref{PTLn}.
\begin{crl}
The following is a cellular basis for $\PTL_n(q)$:\[B_{TL^{\leq 2}_n}=\left\{m_{\es\et}\mid\es,\et\in\Std(\Lambda)\text{ and }\Lambda\in\L_n^{\leq 2}\right\},\]where $\L_n^{\leq2}$ denotes the collection of pairs $(\blam\mid\bmu)$ in $\L_n$ such that each component of the Young diagram of $\blam$ has at most two columns.
\end{crl}

\begin{crl}
The dimension of $\PTL(q)$ is given by the following formula:\[\dim\PTL_n(q)\,\,\,=\sum_{(k_1^{m_1},\ldots,k_r^{m_r})\in\Par_n}\!b_n(\alpha)^2(\c_{k_1})^{m_1}m_1!\cdots (\c_{k_r})^{m_r}m_r!,\]where $\c_k$ denotes the $k$th \emph{Catalan number}.
\end{crl}

\begin{proof}[Proof of Theorem \ref{PTLn}]
Let $\pi_1^\alpha:\Mat_{b_n(\alpha)}\big(H^{wr}_\alpha(q)\big)\to\Mat_{b_n(\alpha)}\big(TL^{wr}_\alpha(q)\big)$ be the natural surjection induced by the projection $H^{wr}_\alpha(q)\twoheadrightarrow TL_\alpha^{wr}(q)$. We will first prove that $\pi_1^\alpha\circ\Psi_\alpha$ factors through $\bbE_\alpha\E_n/\bbE_\alpha J_n\simeq\PTL_n^\alpha$. To show this, it is enough to check that\begin{equation}\label{zero}\pi_1^\alpha\circ\Psi_\alpha(\bbE_\alpha e_1e_2g_{1,2})=0,\end{equation}because the Steinberg elements are all conjugate to each other in $\E_n^\alpha(q)$ \cite[Proposition 4.5]{Ju13}. For short, we set $\bI_{\{1,2,3\}}=\{\{1,2,3\},\{4\},\ldots,\{n\}\}$. In particular, we note that $E_{\bI_{\{1,2,3\}}}=e_1e_2$. Then, for the case when $n_i\leq2$ for all $i$, \eqref{zero} holds trivially since $\bbE_\bI E_{\bI_{\{1,2,3\}}}=0$ for all $\bI$ of type $\alpha$. Suppose now that $\alpha\in\Par_n$ such that $n_j\geq3$ for some $j$. Directly from the definition \ref{idempotenteEn} of $\bbE_\alpha$ together with \eqref{propIdempotentesEI} we have that\[\Psi_\alpha(\bbE_\alpha e_1e_2g_{1,2})=\mathop{\sum_{\bI_{\{1,2,3\}}\subseteq \bI}}_{|\bI|=\alpha}\Psi_\alpha\left(\bbE_\bI g_{1,2}\right)=\sum_{w\in \mathbf{D}_{\{1,2,3\}}^\alpha}\Psi_\alpha\left(\bbE_{\bI_\blam w}g_{1,2}\right),\]where $\blam$ is any multipartition of type $\alpha$ and $\mathbf{D}_{\{1,2,3\}}^\alpha$ is the set of distinguished right coset representatives $w$ of $\S_{\bI_\blam}$ in $\S_n$ such that $\bI_{\{1,2,3\}}\subseteq\bI_\blam w$. In particular, we can consider $\blam$ as the increasing multipartition obtained by reordering the components of the multipartition $((1^{n_1}),\ldots,(1^{n_{i-1}}),(3,1^{n_i-3}),(1^{n_{i+1}}),\ldots,(1^{n_r}))$, where $i$ is the index of the component of $\bI_\blam w$ containing the set $\{1,2,3\}$. Thus, we have\begin{equation}\label{basalelementtrans}\bbE_{\bI_\blam w}g_{1,2}=\bbE_{\bI_\blam w_\bs}g_{1,2}=g_{w_\bs}^\ast\bbE_{\bI_\blam}g_{w_\bs}g_{1,2}=g_{w_\bs}^\ast\bbE_{\bI_\blam}g_{k,k+1}g_{w_\bs},\end{equation}where $k=n_1+\cdots+n_{i-1}+1$ and $\bs$ is the standard $\blam$-multitableau associated with the distinguished right coset representative $w$, that is, $w_\bs=w$. From the definition of $\Psi_\alpha$ and \eqref{basalelementtrans} it now follows that\[\Psi_\alpha(\bbE_\alpha e_1e_2g_{1,2})=\mathop{\sum_{\bs\in\Std(\blam)}}_{\bI_{\{1,2,3\}}\subseteq \bI_\blam w_\bs}\Psi_\alpha\left(g_{w_\bs}^\ast\bbE_{\bI_\blam}g_{k,k+1}g_{w_\bs}\right)=\mathop{\sum_{\bs\in\Std(\blam)}}_{\bI_{\{1,2,3\}}\subseteq\bI_\blam w_\bs}g_{k,k+1}M_{\bs\bs}\]which is indeed equal to zero in $\Mat_{b_n(\alpha)}(TL^{wr}_\alpha)$ since each $g_{k,k+1}$ is equal to zero in $TL^{wr}_\alpha(q)$. Thus we get the following commutative diagram of $\qring$--homomorphisms\[\xymatrixcolsep{4pc}\xymatrix{\E_n^\alpha(q) \ar[r]^(.35){\Psi_\alpha}\ar[d]_{\pi_2^\alpha}&\Mat_{b_n(\alpha)}\big(H^{wr}_\alpha(q)\big)\ar[d]^{\pi_1^\alpha}\\\PTL_n^\alpha(q)\ar[r]_(.35){\theta_\alpha}&\Mat_{b_n(\alpha)}\big(TL^{wr}_\alpha(q)\big)}\]
where $\pi_1^\alpha$ and $\pi_2^\alpha$ are the natural projections. Similarly, by using the inverse map of $\Psi_\alpha$, we have that there is a unique $\qring$--homomorphism $\overline{\theta_\alpha}:\Mat_{b_n(\alpha)}(TL^{wr}_\alpha(q))\to \PTL_n^\alpha(q)$ such that $\overline{\theta_\alpha}\circ\pi_1^\alpha=\pi_2^\alpha\circ\Psi_\alpha^{-1}$. Combining these two facts we have that\begin{equation}\label{compo}(\overline{\theta_\alpha}\circ\theta_\alpha)\circ\pi_2^\alpha=\pi_2^\alpha\quad\text{ and }\quad(\theta_\alpha\circ\overline{\theta_\alpha})\circ\pi_1^\alpha=\pi_1^\alpha\end{equation}Since $\pi_1^\alpha$ and $\pi_2^\alpha$ are both surjective, the relations in \eqref{compo} implies that $\overline{\theta_\alpha}\circ\theta_\alpha=\id_{\PTL_n^\alpha}$ and $\theta_\alpha\circ\overline{\theta_\alpha}=\id_{\Mat_{b_n(\alpha)}(TL^{wr}_\alpha(q))}$. This concludes the proof.
\end{proof}

\begin{crl}\label{inclusionbTLn}
There is an embedding of $\qring$--algebras $\iota_2:bTL_n(q)\hookrightarrow\PTL_n(q)$ given by the following commutative diagram:\[\xymatrix{bH_n(q)\ar[r]^{\iota_1}\ar[d]_{\pi_2}&\E_n(q)\ar[d]^{\pi_1}\\bTL_n(q)\ar[r]_{\iota_2}&\PTL_n(q)}\]
\end{crl}
\begin{proof}
It is a consequence of that fact that $\iota_2$ sends cellular basis elements of $bTL_n(q)$ to cellular basis elements of $\PTL_n(q)$. 
\end{proof}

The following diagram of $\qring$--algebra homomorphisms summarize the relationship between the most important algebras studied in this paper:
\[\xymatrix{\displaystyle\bigoplus_{\mu\in\C_n}H_\mu\ar[r]^\simeq\ar@{>>}[d]&
bH_n(q)\,\ar@{^(->}[r]^{\iota_1}\ar[d]_{\pi_2}&\E_n(q)\ar[d]^{\pi_1}\ar[r]^{\simeq\kern+1.5cm}&
\displaystyle\bigoplus_{\alpha\in\Par_n}\Mat_{b_n(\alpha)}\big(H^{wr}_\alpha(q)\big)\ar@{>>}[d]\\
\displaystyle\bigoplus_{\mu\in\C_n}TL_\mu\ar[r]_\simeq&bTL_n(q)\,\ar@{^(->}[r]_{\iota_2}&
\PTL_n(q)\ar[r]_{\simeq\kern+1.5cm}&\displaystyle\bigoplus_{\alpha\in\Par_n}\Mat_{b_n(\alpha)}\big(TL^{wr}_\alpha(q)\big)}\]

\subsection{Diagrammatic realization of $bTL_n(q)$}\label{115}

Here we recall the boxed ramified monoid associated with the Jones monoid, $\BR(\J_n)$, introduced as the monoid $b\J_n$ in \cite[Section 7]{AiArJu23}, which can be regarded as a basis for the algebra $bH_n(q)$. Specifically, we get that $bTL_n(q)$ is a $q$--deformation of the monoid algebra generated by $\BR(\J_n)$ (Theorem \ref{073}), providing $bTL_n(q)$ with a structure of diagram algebra.

It is well known that, for every $m\geq1$, $|\J_m|=\c_m$, so, similarly as it was done in Proposition \ref{013} and Remark \ref{021}, we have\[|\BR(\J_n)|=\sum_{(\mu_1,\ldots,\mu_k)}\c_{\mu_1}\cdots\c_{\mu_k}=\binom{2n-1}{n}=(1,3,10,35,126,462,1716,6435,24310,92378,\ldots),\]where $(\mu_1,\ldots,\mu_k)$ is a composition of $n$. See \cite[A001700, A088218]{OEIS} and \cite[Proposition 61]{AiArJu23}.

For every $i\in[n-1]$, denote by $d_i$ the ramified partition $(t_i,b_i)$. See Figure \ref{007}.
\begin{figure}[H]\figfiv\caption{Tied tangle generator $d_i$ en $\R(\CC_n)$.}\label{007}\end{figure}

\begin{thm}[{\cite[Theorem 56]{AiArJu23}}]\label{073}
The monoid $\BR(\J_n)$ is presented by generators $e_1,\ldots,e_{n-1}$ satisfying \eqref{023} and generators $d_1,\ldots,d_{n-1}$ subject to the following relations:\begin{gather}
d_i^2=d_i;\qquad d_id_j=d_jd_i,\quad|i-j|>1;\label{036}\\
d_ie_j=e_jd_i;\qquad d_ie_i=d_i;\qquad d_id_jd_i=e_jd_ie_j,\quad|i-j|=1.\label{037}
\end{gather}
\end{thm}
Observe that $\BR(\J_n)$ is a quotient of the direct product of $\C_n$ with a right--angled Artin--Tits monoid.

It is immediate from Theorem \ref{073} and Proposition \ref{134} that $bTL_n(q)$ is a $q$--deformation of the monoid algebra of $\BR(J_n)$. Thus, the algebra $bTL_n(q)$ inherits the diagram combinatorics of $\BR(\J_n)$. See Figure \ref{133}.
\begin{figure}[H]
\figtwosev
\caption{Quadratic relation in \eqref{R1} in terms of diagrams.}\label{133}
\end{figure}

\begin{rem}[Singular part]
Due to \cite[Proposition 12]{AiArJu23}, we have $\R(\J_n)^\times=\J_n^\times=\J_n\cap\S_n=\{1\}$. Thus, the singular part of $\R(\J_n)$, which contains $\BR(\J_n)\backslash\{1\}$ (Remark \ref{043}), is the ideal $s\R(\J_n)=\R(J_n)\backslash\{1\}$. However, at the time of writing this paper, there is no even a known presentation for the monoid $\R(J_n)$.
\end{rem}

\begin{rem}
As of now, there is no known diagrammatic realization for the algebra $PTL_n$, however, due to Theorem \ref{073}, the embedding in Corollary \ref{inclusionbTLn} suggest a potential diagrammatic realization for the algebra $PTL_n$, through the diagrams of $bTL_n$, as the approach employed by H\"arterich in \cite{Ha99}.
\end{rem}

\subsection{The boxed ramified Brauer monoid}\label{116}

As both the symmetric group and the Jones monoid are submonoids of $\Br_n$, the boxed ramified monoid of the Brauer monoid $\BR(\Br_n)$ is an extension of $\BR(\S_n)$ and $\BR(\J_n)$ at once. Here we study this monoid, in particular we give a presentation for it (Theorem \ref{016}).

It is well known that $|\Br_m|=(2m-1)!!$ for all $m\geq1$ \cite[A001147]{OEIS}. So, similarly as it was done in Proposition \ref{013} and Remark \ref{021}, we have\[|\BR(\Br_n)|=\sum_{(\mu_1,\ldots,\mu_k)}(2\mu_1-1)!!\cdots(2\mu_k-1)!!=(1,4,22,154,1330,13882,171802,2474098,\ldots),\]where $(\mu_1,\ldots,\mu_k)$ is a composition of $n$. See \cite[A295553]{OEIS}.

\begin{pro}\label{028}
The monoid $\BR(\Br_n)$ is generated by $d_1,\ldots,d_{n-1}$, $e_1,\ldots,e_{n-1}$ and $z_1,\ldots,z_{n-1}$.
\end{pro}
\begin{proof}
Let $(I,J)\in\BR(\Br_n)$. Since $I\in\Br_n$, \cite[Proposition 27]{AiArJu23} implies that there are unique $s,s'\in\S_n$ and $k\leq\frac{n}{2}$ such that $I=st_1t_3\cdots t_{2k-1}s'$. Also, since $J$ is boxed, we have $(I,J)=(s,J)(t_1t_3\cdots t_{2k-1},J)(s',J)$. Observe that $(s,J)$ and $(s',J)$ belong to $\BR(\S_n)$, and $(t_1t_3\cdots t_{2k-1},J)$ belongs to $\BR(\J_n)$. So, Proposition \ref{025} and Theorem \ref{007} imply that $\BR(\Br_n)$ is generated by $d_1,\ldots,d_{n-1}$, $e_1,\ldots,e_{n-1}$ and $z_1,\ldots,z_{n-1}$.
\end{proof}

Due to the relations that define $\R(\Br_n)$, we have the following relations:\begin{equation}\label{038}z_id_jd_i=z_jd_i,\quad d_id_jz_i=d_iz_j,\quad|i-j|=1;\qquad z_id_i=d_i=d_iz_i;\qquad z_id_j=d_jz_i,\quad|i-j|>1.\end{equation}

\begin{rem}[Normal form]\label{074}
As shown in the proof of Proposition \ref{028}, every element $(I,J)\in\BR(\Br_n)$ can uniquely written as a product $(s,J)(t_1t_3\cdots t_{2k-1},J)(s',J)$. Thus, $(I,J)=ezd_1d_3\cdots d_{2k-1}z'$ with $e=(1,J)$, where $ez$ and $ez'$ are the normal forms of $(s,J)$ and $(s',J)$, respectively, as in Remark \ref{075}.
\end{rem}

\begin{thm}\label{016}
The monoid $\BR(\Br_n)$ is presented by generators $e_1,\ldots,e_{n-1}$ satisfying \eqref{023}, generators $z_1,\ldots,z_{n-1}$ satisfying \eqref{022} and \eqref{026}, and generators $d_1,\ldots,d_{n-1}$ subject to \eqref{036}, \eqref{037} and \eqref{038}.
\end{thm}

\begin{rem}[Singular part]
It is an open problem to give a presentation for the singular part $s\R(\Br_n)$ of $\R(\Br_n)$. Observe that both $s\R(\S_n)$ and $s\R(\J_n)$ are subsemigroups of $s\R(\Br_n)$.
\end{rem}

\subsubsection{Proof of Theorem \ref{016}}\label{117}

Let $M$ be the monoid presented by generators $\te_1,\ldots,\te_{n-1}$, $\tz_1,\ldots,\tz_{n-1}$ and $\td_1,\ldots,\td_{n-1}$, subject to the following relations:\begin{gather}
\te_i^2=\te_i;\qquad \te_i\te_j=\te_j\te_i;\label{077}\\
\tz_i\tz_j\tz_i=\tz_j\tz_i\tz_j,\quad|i-j|=1;\qquad\tz_i\tz_j=\tz_j\tz_i,\quad|i-j|>1;\label{078}\\
\te_i\tz_j=\tz_j\te_i;\qquad\tz_i^2=\te_i;\qquad\te_i\tz_i=\tz_i;\label{079}\\
\td_i^2=\td_i;\qquad\td_i\td_j=\td_j\td_i,\quad|i-j|>1;\label{080}\\
\td_i\te_j=\te_j\td_i;\qquad\td_i\te_i=\td_i;\qquad\td_i\td_j\td_i=\te_j\td_i\te_j,\quad|i-j|=1;\label{081}\\
\tz_i\td_j\td_i=\tz_j\td_i,\quad\td_i\td_j\tz_i=\td_i\tz_j,\quad|i-j|=1;\qquad\!\tz_i\td_i=\td_i=\td_i\tz_i;\qquad\!\tz_i\td_j=\td_j\tz_i,\quad|i-j|>1.\label{082}
\end{gather}

\begin{lem}\label{084}
The map $\phi:M\to\BR(\Br_n)$ sending $\te_i\mapsto e_i$, $\tz_i\mapsto z_i$ and $\td_i\mapsto d_i$ is a monoid epimorphism.
\end{lem}
\begin{proof}
Observe that, under $\phi$, relations in \eqref{077} to \eqref{079} are sent to \eqref{023}, \eqref{022} and \eqref{026}, which hold in $\BR(\S_n)$, and relations in \eqref{080} and \eqref{081} are sent to \eqref{036} and \eqref{037}, which hold in $\BR(\J_n)$. On the other hand, relations in \eqref{082} are sent to \eqref{038}. Therefore $\phi$ is a monoid epimorphism.
\end{proof}

\begin{rem}\label{085}
By a direct inspection of the relations that define $M$, we can conclude that the map $\phi_b:M\to\Br_n$ sending $\tz_i\mapsto s_i$, $\td_i\mapsto t_i$ and $\te_i\mapsto1$ is a monoid epimorphism. Moreover, each relation in \eqref{032} to \eqref{034} can be obtained, under $\phi_b$, from a relation in \eqref{077} to \eqref{082}. Thus, as each $\te_i\in Z(M)$, then every relation in $M$ can be written as $(er,e'r')$, where $e,e'$ are words in the generators $\te_i$, and $r,r'$ are words in the generators $\tz_i,\td_j$, representing elements that are sent, under $\phi_b$, to the same element of $\Br_n$.
\end{rem}

\begin{pro}\label{086}
Every element $g\in M$ can be represented by a word in the generators $\te_i$, $\tz_i$ and $\td_i$, such that the word obtained from it when replacing $\te_i$ by $e_i$, $\tz_i$ by $z_i$ and $\td_i$ by $d_i$ is a normal form of $\phi(g)$.
\end{pro}
\begin{proof}
Remark \ref{085} and the fact that each $\te_i$ belongs to the center of $M$, imply that $g$ can be represented by a word $eu$, where $e$ is a word in the generators $\te_i$, and $u$ is a word in the generators $\tz_i$ and $\td_i$, which becomes a normal form in $\Br_n$ \cite[Proposition 27]{AiArJu23} when replacing $\tz_i$ by $s_i$ and $\td_i$ by $t_i$. Observe that if $u=u_{i_1}\cdots u_{i_k}$ with $u_{i_j}\in\{\tz_{i_j},\td_{i_j}\}$, then $eu$ is equivalent to $(e\te_{i_1}\cdots\te_{i_k})u$. Since $e':=e\te_{i_1}\cdots\te_{i_k}$ contains all possible generators obtained from $u$, then the word obtained from $e'u$ when replacing $\te_i$ by $e_i$, $\tz_i$ by $z_i$ and $\td_i$ by $d_i$ is a normal form of $\phi(g)$ as in Remark \ref{074}.
\end{proof}

\begin{proof}[Proof of Theorem \ref{016}]
Let $g,g'\in M$ such that $\phi(g)=\phi(g')$, and let $u$ be the normal form of $\phi(g)$ as in Remark \ref{074}. Proposition \ref{086} implies both $g,g'$ have word representatives in the generators $\te_{i,j}$, $\tz_i$ and $\td_i$, which become $u$ when replace $\te_{i,j}$ by $e_{i,j}$, $\tz_i$ by $z_i$ and $\td_i$ by $d_i$. Since $u$ is uniquely defined, these words must be identical. Thus $g=g'$. Therefore $\phi$ is a monoid isomorphism.
\end{proof}


\subsubsection*{Acknowledgements}

We thank to J. Juyumaya for encouraging us to explore a potential algebra associated with the monoid $bJ_n$.

The first author was supported by the grant DIDULS PR232146.

\bibliographystyle{plainurl}
\bibliography{../Files/bibtex.bib}

\end{document}